\documentclass[11pt]{amsart}

\usepackage[english]{babel}
\usepackage{booktabs}
\usepackage{graphicx}
\usepackage{sidecap}
\usepackage[hypertexnames=false]{hyperref}
\hypersetup{
	colorlinks,
	linkcolor={red!50!black},
	citecolor={blue!50!black},
	urlcolor={blue!80!black}
}
\usepackage{amssymb,amsmath,amsthm}
\usepackage{mathtools}
\usepackage{accents}
\usepackage{dsfont}
\usepackage{diagbox,multirow}
\usepackage{hhline}
\usepackage[dvipsnames,x11names]{xcolor}
\usepackage{array}
\usepackage{colortbl}
\usepackage{subfig}
\usepackage{enumerate}
\usepackage[shortlabels]{enumitem}
\usepackage{makecell}
\usepackage{float}
\usepackage{framed}
\usepackage{varwidth}
\usepackage[export]{adjustbox}
\usepackage{xparse}
\usepackage {tikz}
\usetikzlibrary{cd}
\usetikzlibrary{arrows,decorations.pathmorphing,backgrounds,positioning,fit,matrix}
\usetikzlibrary{decorations.markings}

\newlength\mylen
\tikzset{
	bicolor/.style 2 args={
		dashed,dash pattern=on 20pt off 20pt,-,#1,
		postaction={draw,dashed,dash pattern=on 20pt off 20pt,-,#2,dash phase=20pt}
	},
}

\usepackage{etoolbox}
\apptocmd{\sloppy}{\hbadness 10000\relax}{}{}
\usepackage{amsbsy}

\usepackage[normalem]{ulem}
\newcommand\redsout{\bgroup\markoverwith{\textcolor{red}{\rule[0.5ex]{2pt}{0.4pt}}}\ULon}

\usepackage{todonotes}

\newtheorem{theorem}{Theorem}
\newtheorem{lemma}{Lemma}[section]
\newtheorem{thm}[lemma]{Theorem}
\newtheorem{proposition}[lemma]{Proposition}
\newtheorem{corollary}[lemma]{Corollary}

\theoremstyle{definition}
\newtheorem{definition}[lemma]{Definition}
\newtheorem*{example}{Example}

\theoremstyle{remark}
\newtheorem{remark}{Remark}

\numberwithin{equation}{section}

\newcommand{\Hm}[1]{\leavevmode{\marginpar{\tiny%
			$\hbox to 0mm{\hspace*{-0.5mm}$\leftarrow$\hss}%
			\vcenter{\vrule depth 0.1mm height 0.1mm width \the\marginparwidth}%
			\hbox to 0mm{\hss$\rightarrow$\hspace*{-0.5mm}}$\\\relax\raggedright
			#1}}}

\newcommand{\N}{\mathbb{N}}
\newcommand{\R}{\mathbb{R}}

\colorlet{red2}{red!65!black}
\definecolor{azure}{rgb}{0.0, 0.5, 1.0}
\definecolor{darkpastelgreen}{rgb}{0.01, 0.75, 0.24}
\definecolor{lightgreen}{rgb}{0.56, 0.93, 0.56}
\definecolor{lightgray}{rgb}{0.83, 0.83, 0.83}
\definecolor{gray}{rgb}{0.5, 0.5, 0.5}
\definecolor{darkspringgreen}{rgb}{0.09, 0.45, 0.27}

\DeclareMathSymbol{\shortminus}{\mathbin}{AMSa}{"39}

\newcommand{\vertiii}[1]{{\left\vert\kern-0.25ex\left\vert\kern-0.25ex\left\vert #1
		\right\vert\kern-0.25ex\right\vert\kern-0.25ex\right\vert}}

\newcommand{\dom}{\textnormal{dom}}

\newcommand{\Deg}{\operatorname{Deg}}

\newcommand{\Lmin}{\mathcal{L}_{n}}

\newcommand{\Deltadir}{\Delta_{\textnormal{dir}}}
\newcommand{\Deltadirn}{\Delta_{\textnormal{dir},n}}

\makeatletter
\newcommand{\sbullet}{%
	\hbox{\fontfamily{lmr}\fontsize{.6\dimexpr(\f@size pt)}{0}\selectfont\textbullet}}
\DeclareRobustCommand{\mathbullet}{\accentset{\sbullet}}
\makeatother

\usepackage[style=numeric,backend=bibtex,giveninits=true,maxbibnames=99,sortcites=true]{biblatex}
\title[The generalized porous medium equation on graphs]{The generalized porous medium equation on graphs: existence and uniqueness of solutions with $\ell^1$ data}
\author {Davide Bianchi}
  \address{
School of Science\\
Harbin Institute of Technology (Shenzhen)\\
Shenzhen (China)}
\email{bianchi@hit.edu.cn}

\author{Alberto G. Setti}
\address{Dipartimento di Scienze e Alta Tecnologia\\
	Universit\`a dell'Insubria\\
	Como (Italy)}
\email{alberto.setti@uninsubria.it}

\author{Rados{\l}aw K. Wojciechowski}
\address{Department of Mathematics and Computer Science
	\\
	York College -- CUNY \\ Jamaica (USA)}
\address{Department of Mathematics
	\\
	Graduate Center -- CUNY \\ New York (USA)\\
	}
	
\email{rwojciechowski@gc.cuny.edu}

\begin{document}	
\begin{abstract}
We study solutions of the generalized porous medium equation on infinite graphs. For nonnegative or nonpositive integrable data, we prove the existence and uniqueness of mild solutions on any graph. For changing sign integrable data, we show existence and uniqueness under extra assumptions such as local finiteness or a uniform lower bound on the node measure.
\end{abstract}
\maketitle %

\section{Introduction}
Our model formal equation is the following:
	\begin{equation}\tag*{GPME}\label{model_problem}
	\partial_t u(t,x) + \Delta\Phi u(t,x) = f(t,x) \quad  \mbox{for every } (t,x) \in (a,b)\times X.
	\end{equation}
This equation is called the \emph{generalized porous medium equation} (GPME) or \emph{filtration equation} whenever $\Delta$ is the Laplace operator and $\Phi$ is the canonical extension to a function space of a map $\phi \colon  \R \to \R$ such that $\phi$ is  strictly monotone increasing, $\phi(\R)=\R$ and $\phi(0)=0$. If $\phi(s)= s^m \coloneqq s|s|^{m-1}$, then the above equation is known as  the \emph{porous medium equation} (PME) when $m>1$ and the \emph{fast diffusion equation} (FDE) when $0<m<1$. Clearly, when $m=1$ and $f\equiv0$ we recover the classic heat equation.

The \ref{model_problem} has a long story and we invite the interested reader to look at the seminal book by J. L. V\'{a}zquez \cite{vazquez2007porous} for a detailed and exhaustive account. In recent years, research interest about properties of solutions of the \ref{model_problem} has focused on the Riemannian setting as can be seen by the increasing number of related works, see, for example, \cite{bonforte2008fast,lu2009local,grillo2014radial,vazquez2015fundamental,grillo2016smoothing,grillo2017porous,grillo2018porous,grillo2018porous2,bianchi2018laplacian,grillo2019blow,grillo2020nonlinear,grillo2021fast,grillo2021global,meglioli2021blow,dipierro2021global} and references therein for an overview of the most significant developments.

In contrast, in the graph setting there are still relatively few results for the \ref{model_problem}. This is despite the fact that, on the one hand, the \ref{model_problem} is being used as a model equation for several  real-world phenomena (e.g., the flow of gas through a porous medium, water infiltration or population dynamics) and, on the other hand, graphs are ubiquitous in many applied fields: in physics \cite{nakanishi1971graph}, biology \cite{lesne2006complex,stam2007graph,lieberman2005evolutionary}, image and signal processing \cite{shuman2013emerging,ta2010nonlocal,elmoataz2015p}, engineering \cite{deo2016graph}, etc.

 To make our setting more precise, let us fix a graph $G=(X,w,\kappa,\mu)$ where $X$ is a countable node set, $w \colon X \times X \to [0,\infty)$ is a symmetric map with zero diagonal, $\kappa \colon X \to [0, \infty)$ is a possibly nontrivial killing term and $\mu \colon X \to (0, \infty)$ is a strictly positive node measure on $X$.

 For notational convenience, let us fix $a=0$ and $b=T\in (0,\infty]$. We will focus our attention on the following Cauchy problem posed on $G$:
\begin{equation}\tag*{Cauchy-GPME}\label{Model_Equation_graph}
\begin{cases}
\partial_t u(t,x) + \Delta\Phi u(t,x) = f(t,x) & \mbox{for every } (t,x) \in (0,T)\times X,\\
\lim_{t\to 0^+} u(t,x) = u_0(x) & \mbox{for every } x\in X
\end{cases}
\end{equation}
where $f \colon (0,T)\times X \to \R$ and $u_0 \colon X \to \R$ are generic functions at the moment. In this setting, $\Delta$ represents the (formal) graph Laplacian operator defined by the formula
$$
\Delta u(x) \coloneqq \frac{1}{\mu(x)}\sum_{y \in X}w(x,y)\left(u(x) - u(y)\right) + \frac{\kappa(x)}{\mu(x)}u(x).
$$
 The \ref{model_problem} on graphs belongs to the broader class of nonlinear diffusion equations with nonconstant diffusion since the edge-weight function $w(\cdot,\cdot)$ can be seen as a counterpart of the nonconstant diffusion coefficients $\{a_{i,j}(x)\}_{i,j=1}^d$ which characterize the second-order differential operator $\sum_{i=1}^d \partial_i \left( a_{i,j}(x)\sum_{j=1}^d\partial_ju(x)\right)$ acting on smooth functions on $\R^d$.

 We now give a brief overview of some recent results concerning nonlinear equations in the graph setting.
 For the counterpart of the Kazdan-Warner equation, see \cite{grigoryan2016kazdan,
keller2018kazdan, liu2020multiple}.
Concerning the existence and uniqueness of solutions for reaction-diffusion type equations on the lattice $\mathbb{Z}$, see \cite{slavik2019well,stehlik2017exponential}.
 For the Yamabe and other equations, see \cite{grigoryan2016yamabe,
grigoryan2017existence, lin2021heat}. For parabolic equations involving the $p$-Laplacian, see \cite{mugnolo2013parabolic,hua2015time}.
Finally, we mention some results concerning the existence and nonexistence of global nonnegative solutions of an abstract semilinear heat equation given in \cite{lin2017existence, lin2018blow-up, wu2021blow-up} which were recently extended to a general setting in \cite{lenz2021blowup}.

 With reference to the PME in the discrete setting, we highlight \cite{erbar2014gradient} where the authors study the (finite) discrete analogue of the Wasserstein gradient flow structure for the PME in $\R^n$. Concerning the existence and uniqueness of solutions of the \ref{Model_Equation_graph} to the best of our knowledge there is an almost complete lack of a systematic treatment even in the case of finite graphs with one notable exception: In \cite[Corollary 5.4]{mugnolo2013parabolic}, exploiting an interesting link between the PME and the $p$-heat equation (which is well-known in $\R$, see, e.g., \cite[Section 3.4.3]{vazquez2007porous}), it is shown that if $G$ is an infinite tree, uniformly locally finite with $\mu\equiv 1$, then there exists a unique solution of the \ref{Model_Equation_graph} for $\phi(s)=s|s|^{m-1}$ for any $u_0 \in \ell^2(X,\mu)$ and $f\equiv0$. For more details and the definition of solutions in that setting we invite the interested reader to look at the mentioned paper.

Our approach is different. The main goal of this article is to prove existence and uniqueness results for classes of solutions of the \ref{Model_Equation_graph} problem under the weakest possible hypotheses on the graph $G$, on the initial datum $u_0$ and on the forcing term $f$. To achieve this, we will borrow techniques from the theory of semigroups on Banach spaces and, as it will become clear later, $\ell^1(X,\mu)$  will turn out to be the ideal space for our considerations.

We will consider three classes of solutions: \emph{mild} (Definition \ref{def:weak_solution}), \emph{strong} (Definition \ref{def:strong_solution}) and \emph{classic} (Definition \ref{def:classical_solution}) which are characterized by an increasing  ``regularity.'' In particular, mild solutions $u$ are limits of $\epsilon$-approximations $u_\epsilon$ that satisfy the \ref{Model_Equation_graph} for a time discretization well-adapted to $f$.
If the operator $\mathcal{L}\coloneqq\Delta\Phi$ is $m$-accretive, then it is possible to immediately infer the existence and uniqueness of mild solutions for the  \ref{Model_Equation_graph} problem by appealing to well-known results, see \cite{benilan1972equations,benilan1988evolution,crandall1971generation}. Thus, the bulk of our work consists of establishing the $m$-accretivity of (a restriction of) the operator $\mathcal{L}$ on an appropriate Banach space.

Accretivity of an operator $\mathcal{L}$ with respect to a norm $\|\cdot\|$ on a real Banach space $\mathfrak{E}=(E,\|\cdot\|)$ means that
$$
\left\|(u - v) + \lambda \left(\mathcal{L}u - \mathcal{L}v\right) \right\|\geq \|u - v\|
$$
for every $u,v \in \dom\left(\mathcal{L}\right)\subseteq E$ and for every $\lambda >0$. Furthermore, $m$-accretivity means that $\mathcal{L}$ is accretive and $\operatorname{id} +\lambda\mathcal{L}$ is surjective for every $\lambda>0$. We note that accretivity implies that $\operatorname{id} +\lambda\mathcal{L}$ is injective, thus, $m$-accretivity gives that $\operatorname{id} +\lambda\mathcal{L}$ is bijective.
For a more detailed introduction to the concepts of accretivity and $m$-accretivity, see Subsection \ref{ssec:m-accretivity}.

As can be seen directly, the accretivity property depends on both the operator $\mathcal{L}$ and on the underlying Banach space. For example, the graph Laplacian $\Delta$ on $\ell^p(X,\mu)$ is $m$-accretive on any finite graph for $p \in [1,\infty)$, see Proposition \ref{lem:m-accretivity_for_finite_graphs}. On the other hand, the nonlinear operator $\mathcal{L}$ can fail to be accretive with respect to the $\ell^2$-norm, see Example~\ref{ex:1}. What is crucial for our analysis is that the restriction of $\mathcal{L}$ to a suitable dense subset of 
$\ell^1(X,\mu)$ will be shown to be accretive for any graph. For $m$-accretivity to hold some additional hypothesis are required as will be discussed in what follows. We note that there is a parallel development concerning the surjectivity of the formal graph Laplacian $\Delta$ which is always surjective on infinite, locally finite graphs but not necessarily surjective in the not locally finite case, see \cite{ceccherini2012surjectivity,koberstein2020note}.

For a complete introduction to the notation we refer to Section \ref{sec:preliminaries}. We denote by $\ell^{1,+}(X,\mu)$ and $\ell^{1,-}(X,\mu)$ the cones of nonnegative and nonpositive integrable functions, respectively, and by $\mathcal{L}$ the operator
\begin{align*}
		&\mathcal{L} \colon \dom\left( \mathcal{L} \right)\subseteq \ell^{1}\left(X,\mu\right) \to \ell^{1}\left(X,\mu\right) , \\
		&\dom\left( \mathcal{L} \right)\coloneqq\left\{ u \in  \ell^{1}\left(X,\mu\right) \mid  \Phi u\in \dom\left(\Delta\right), \Delta\Phi u \in  \ell^{1}\left(X,\mu\right)  \right\}
\end{align*}
whose action is given by
$$
\mathcal{L}u\coloneqq \Delta\Phi u.
$$
For a subset $\Omega \subseteq \dom\left( \mathcal{L} \right)$, we write $\mathcal{L}_{|\Omega}$
for the restriction of $\mathcal{L}$ to $\Omega$.

We now state the main results, whose proofs can be found in Section \ref{ssec:proofs}.
The first main result discusses the accretivity and $m$-accretivity of $\mathcal{L}$.

\begin{theorem}\label{thm:main1}
	Let $G=(X,w,\kappa,\mu)$ be a graph. Then, there exists a dense subset $\Omega \subseteq \dom(\mathcal{L})$ such that $\mathcal{L}_{|\Omega}$ is accretive.  Moreover, for every $\lambda >0$ and for every $g \in \ell^{1,\pm}(X,\mu)$ there exists a unique $u \in \ell^{1,\pm}(X,\mu)\cap\Omega$ such that
\begin{equation*}\label{eq:main_equation}
	\left( \operatorname{id} + \lambda\mathcal{L}\right)u = g.
\end{equation*}
If one of the following extra hypotheses holds:
	\begin{enumerate}[label={\upshape(\bfseries H\arabic*)},wide = 0pt, leftmargin = 3em]
		\item\label{m-accretivity_A}$G$ is locally finite;
		\item\label{m-accretivity_B} $\inf_{x \in X}\mu(x)>0$;
		\item\label{m-accretivity_C} $\sup_{x \in X}\frac{\sum_{y \in X}w(x,y)}{\mu(x)}<\infty$ and $\Phi\left(\ell^1(X,\mu)\right)\subseteq \ell^1(X,\mu)$;
	\end{enumerate}
then $\operatorname{id}+\lambda\mathcal{L}$ restricted to $\Omega$ is also surjective.  
In particular, $\mathcal{L}_{|\Omega}$ is $m$-accretive. 
Moreover, in all cases, the solution $u$ satisfies the 
contractivity estimate  $$||u||\leq ||g||.$$ 
\end{theorem}

The second main result uses general theory along with the $m$-accretivity established in the first result to yield existence and uniqueness of mild solutions for the \ref{Model_Equation_graph}. 
For this, we consider two cases, namely, when the initial data is nonnegative or nonpositive
and when the initial data changes sign. In the second case, we need to add one of the extra 
hypotheses appearing in the first result above to guarantee existence.
For the definitions of the various types of solutions for the \ref{Model_Equation_graph} and the connections
between them, see Section~\ref{sec:Cauchy_model_problem}.
\begin{theorem}\label{thm:main2}
	Let $G=(X,w,\kappa,\mu)$ be a graph. Let
	\begin{enumerate}[i)]
		\item\label{hp1} $u_0 \in \ell^1(X,\mu)$;
		\item\label{hp2} $f \in L^1_{\textnormal{loc}}\left([0,T]; \ell^1\left(X,\mu\right)\right)$.
	\end{enumerate}
	If one of the following additional conditions holds:
	\begin{enumerate}[a)]
		\item\label{item:nonnegativity/nonpositivity} $u_0,f(t)\geq 0$ (or $\leq 0$) for all $t\geq 0$;
		\item\label{hpA} $u_0$ or $f(t)$ changes sign and at least one of  \ref{m-accretivity_A}, \ref{m-accretivity_B} or \ref{m-accretivity_C} is satisfied;
	\end{enumerate}
	then there exists a unique mild solution $u$ of the \ref{Model_Equation_graph}.
	
	Furthermore, $u(t) \in \ell^1(X,\mu)$ for all $t \in [0,T]$ and for every $\epsilon >0$ there exists a continuous function $\delta \colon [0,\infty) \to [0,\infty)$ such that $\delta(0)=0$ and if $u_\epsilon$ is an $\epsilon$-approximate solution of the \ref{Model_Equation_graph}, then
	\begin{equation}\label{uniform_limit}
		\| u(t) - u_\epsilon(t)\|\leq \delta(\epsilon) \qquad \mbox{for } t \in [0,T-\epsilon].
	\end{equation}	
	Moreover, for any pair $(u_0,f),(\hat{u}_0,\hat{f})$, the corresponding mild solutions $u,\hat{u} \in C\left([0,T];  \ell^1\left(X,\mu\right)\right)$ satisfy
	\begin{equation}\label{contraction_of_solutions}
		\left\| u(t_2) - \hat{u}(t_2)  \right\| \leq \left\| u(t_1) - \hat{u}(t_1)  \right\|   + \int_{t_1}^{t_2}\left\| f(s) - \hat{f}(s) \right\|\, ds, \quad \forall \, 0\leq t_1<t_2 \leq T.
	\end{equation}
	Finally, under hypothesis \ref{item:nonnegativity/nonpositivity}, $u(t)\geq 0$ (or $\leq 0$) for every $t\geq 0$.
\end{theorem}

The paper is organized in the following way:
\begin{itemize}
	\item In Section \ref{sec:preliminaries} we present the main definitions and describe the tools that we will use.
	\item In Section \ref{sec:Cauchy_model_problem} we introduce the abstract Cauchy problem along with a classification of types of solutions.
	\item Section \ref{sec:exisntence&uniqueness} is the core of the paper: We present the proofs of Theorem~\ref{thm:main1} and Theorem~\ref{thm:main2} with an introductory part about the main issues to be addressed. As a concluding application, in Corollary \ref{cor:application} we prescribe some hypotheses on the graph that guarantee that a mild solution is indeed a classic solution.
\end{itemize}
Since the proofs involved in Section \ref{sec:exisntence&uniqueness} are technical and rely on several auxiliary results, we collect them in Appendix \ref{sec:auxiliary_results} and Appendix \ref{sec:appendix2}. 
\section{Preliminaries}\label{sec:preliminaries}
In this section we collect background material for the graph setting and the main mathematical tools that we will use in our proofs.

\subsection{Notation}
Given a set $X$ and a real function space $\mathfrak{F}\subseteq\{ u \colon X \to \R \}$, we denote by $\operatorname{id}\colon \mathfrak{F} \to \mathfrak{F}$ the identity operator. If $\phi \colon \dom\left(\phi\right)\subseteq \R \to \R$ is a function, then we denote by the capital letter $\Phi$ the canonical extension of $\phi$ to $\mathfrak{F}$, that is, the operator $\Phi \colon \dom\left(\Phi\right)\subseteq \mathfrak{F} \to \mathfrak{F}$ given by
\begin{align*}
&\dom\left(\Phi\right)\coloneqq \left\{ u \in \mathfrak{F} \mid u(x) \in \dom\left(\phi\right)\, \forall\, x \in X  \right\},\\
&\Phi u(x)\coloneqq \phi(u(x)).
\end{align*}
 Given a pair of real-valued functions $u$ and $v$ on $X$, we write $u\geq v$ if $u(x)\geq v(x)$ for every $x \in X$. All other ordering symbols are defined accordingly.

 Given a real Banach space $\mathfrak{E}=(E,\|\cdot\|)$, consider an  $E$-valued function $f \colon [0,T] \subset \R \to E$, $t \mapsto f(t) \in E$. Such a function $f$ is called simple if $f$ is of the form
 $$
 f(t) = \sum_{k=1}^n e_k\mathds{1}_{I_k}(t),\qquad e_k \in E,
 $$
 where $I_k$ are Lebesgue measurable subsets of $[0,T]$ and $\mathds{1}_{I_k}$ is the indicator function of $I_k$. The integral of an $E$-valued simple function is defined by
 $$
 \int_0^T f(t) \, dt\coloneqq \sum_{k=1}^n e_k m\left(I_k\right),
 $$
 where $m(\cdot)$ is the Lebesgue measure on $[0,T]$. A function $f$ is (strongly) measurable if there exists a sequence $\{f_n\}_{n\in \N}$ of simple functions such that $f_n(t) \to f(t)$  in norm for almost every (a.e.) $t$ in $[0,T]$.

 A strongly measurable function $f$ is Bochner integrable if there exists a sequence of simple functions such that $f_n \to f$ pointwise a.e.\ in $[0,T]$ and
 $$
 \lim_{n\to \infty}\int_0^T \|f_n(t) - f(t)\| \, dt =0,
 $$
 or equivalently, by a theorem of Bochner, if and only if $\int_0^T \|f(t)\|\, dt < \infty$.
 The integral of $f$ is then defined by
 $$
 \int_0^T f(t) \, dt=  \lim_{n\to \infty}\int_0^T f_n(t) \, dt.
 $$
We denote the space of Bochner integrable functions from $[0,T]$ to $E$ by
$$
L^1([0,T] ; E)\coloneqq \left\{ f \colon [0,T] \to E \mbox{ measurable}\mid \int_0^T \|f(t)\|\, dt < \infty\right\}.
$$
In the same fashion, if $T=\infty$, we denote by $L^1_{\textnormal{loc}}([0,T] ; E)$ the space of $E$-valued functions that are locally Bochner integrable, that is, $f \in L^1_{\textnormal{loc}}([0,T] ; E)$ if and only if $f \in L^1([0,a] ; E)$ for every $a\in (0,\infty)$. Clearly, if $T<\infty$, then $L^1_{\textnormal{loc}}([0,T] ; E)=L^1([0,T] ; E)$.

A function $f \in L^1_{\textnormal{loc}}([0,T] ; E)$ is weakly differentiable with weak derivative $g \in L^1_{\textnormal{loc}}([0,T] ; E)$ if
 $$
 \int_0^T f(t)\eta'(t)\, dt = - \int_0^T g(t)\eta(t)\, dt, \quad\forall\, \eta \in C_c^\infty(0,T),
 $$
 where the integrals are understood in the Bochner sense. The first Sobolev space for locally Bochner integrable functions is defined as
 $$
 W^{1,1}_{\textnormal{loc}}([0,T] ;E)\coloneqq \left\{  f \in L^1_{\textnormal{loc}}([0,T] ; E) \mid f \mbox{ is weakly differentiable} \right\}.
 $$
Let us point out that $f \in W^{1,1}_{\textnormal{loc}}([0,T] ;E)$ if and only if
$$
f(t) = e_0 + \int_0^t g(s)\, ds.
$$
Moreover, $f$ is absolutely continuous and a.e. differentiable in $[0,T]$ with $f'(t)=g(t)$.

For a review of integration and weak derivatives of vector-valued functions, see, for example, \cite[Chapter 5, Sections 4 and 5]{yosida1965functional} and \cite[Chapter 1, Section 4.5]{cazenave1998introduction}.

Given an operator $\mathcal{L} \colon \dom(\mathcal{L}) \subseteq E \to E$ and a subset $\Omega \subseteq \dom(\mathcal{L})$, then the restriction of $\mathcal{L}$ to $\Omega$ is the operator $\mathcal{L}_{|\Omega} \colon \dom(\mathcal{L}_{|\Omega}) \subseteq E \to E$ such that 
\begin{equation*}
\dom(\mathcal{L}_{|\Omega})= \Omega, \qquad \mathcal{L}_{|\Omega}u= \mathcal{L}u \quad \forall\, u \in \Omega.
\end{equation*}

As a final piece of notation we mention that
if $E\subset \{u \colon X \to \R\}$, then we write $f(t,x)$ to indicate the value of $f(t)\in E$ at $x\in X$.

\subsection{The graph setting}
For a detailed introduction to the graph setting as presented here, see \cite{keller2021graphs}. 
We begin with the definition of a graph.
\begin{definition}[Graph]
A \textit{graph} is a quadruple $G=(X,w,\kappa,\mu)$ given by
\begin{itemize}
	\item  a countable set of \emph{nodes} $X$;
	\item a nonnegative \emph{edge-weight} function $w\colon X\times X \to [0,\infty)$;
	\item  a nonnegative \emph{killing term} $\kappa \colon X \to [0,\infty)$;
	\item a positive \emph{node measure} $\mu \colon X \to (0,\infty)$
\end{itemize}
where the edge-weight function $w$ satisfies:
\begin{enumerate}[label={\upshape(\bfseries A\arabic*)},wide = 0pt, leftmargin = 3em]
	\item\label{assumption:symmetry} Symmetry: $w(x,y)=w(y,x)$ for every $x,y \in X$;
	\item\label{assumption:loops} No loops: $w(x,x)=0$ for every $x \in X$;
	\item\label{assumption:degree} Finite sum: $\sum_{y\in X} w(x,y) < \infty$ for every $x \in X$.
\end{enumerate}
\end{definition}

If the cardinality of the node set is finite, i.e., $|X|<\infty$, then $G$ is called a \emph{finite graph}, otherwise, $G$ is called an \emph{infinite graph}. The non-zero values $w(x,y)$ of the edge-weight function $w$ are called \emph{weights} associated with the \emph{edge} $\{x,y\}$. In this case we will write $x\sim y$ meaning that $x$ is \emph{connected} to $y$. On the other hand, if $w(x,y)=0$, then we will write $x\nsim y$ meaning that $x$ and $y$ are not connected by an edge. A \emph{walk} is a (possibly infinite) sequence of nodes $\{x_{i}\}_{i\geq 0}$ such that $x_{i}\sim x_{i+1}$. A \emph{path}
is a walk with no repeated nodes.  A graph is \emph{connected}
if there is a finite walk connecting every pair of nodes, that is, for any pair of nodes $x, y$ there is a finite walk such that $x = x_{0}\sim x_{1}\sim\cdots\sim x_{n}=y$. Moreover, we will say that a subset $A \subseteq X$  is connected if for every pair of nodes $x,y \in A$ there exists a finite walk connecting $x$ and $y$ all of whose nodes are in $A$. A subset $A \subseteq X$ is a \emph{connected component} of $X$ if $A$ is maximal with respect to inclusion.

A graph is said to be \emph{locally finite} if for every $x \in X$ there are at most a finite number of nodes $y$ such that $w(x,y)\neq 0$. 
We define the \emph{degree} $\deg$ and \emph{weighted degree} $\Deg$ of a node $x$ as
$$ \deg(x) \coloneqq \sum_{{y\in X}} w(x,y) + \kappa(x) \quad \textup{ and } \quad \Deg(x) \coloneqq \frac{\deg(x)}{\mu(x)}.$$
	Clearly, by \ref{assumption:degree}, $\deg(x)$ and $\Deg(x)$ are finite for every $x\in X$. Observe that, if $\kappa\equiv 0$ and $w(x,y)\in \{0,1\}$, then $\deg$ corresponds to the standard definition in the literature on finite graphs (e.g., \cite{estrada2015first}).
	
	The set of real-valued functions on $X$ is denoted by $C(X)$ and $C_c(X)$ denotes the set of functions on $X$ with finite support. As usual, for $p\in [1,\infty]$ we define the $\ell^p(X,\mu)$ subspaces as
	$$
	\ell^p(X,\mu)\coloneqq\begin{cases}
	\left\{ u \in C(X) \mid \sum_{x\in X} |u(x)|^p\mu(x)<\infty \right\} & \mbox{for } p \in [1,\infty),\\
	\left\{ u \in C(X) \mid \sup_{x\in X}|u(x)|<\infty \right\} & \mbox{for } p =\infty
	\end{cases}
	$$
	with their norms
	$$
	\|u\|_p\coloneqq \begin{cases}
	\left(\sum_{x\in X} |u(x)|^p\mu(x)\right)^{{1}/{p}} & \mbox{for } p \in [1,\infty),\\
	\sup_{x\in X}|u(x)| & \mbox{for } p =\infty
	\end{cases}
	$$	
and with the standard remark that the $\ell^2$-norm is induced by the inner product
$$
\langle u, v \rangle_{\ell^2} \coloneqq \sum_{x\in X} u(x)v(x)\mu(x)
$$
making $\ell^2(X,\mu)$ into a Hilbert space. In general, we will use the convention $\|\cdot\|\coloneqq\|\cdot\|_1$ since we will work almost always with the $\ell^1$-norm. However, in case of possible ambiguity in the text, we will specify the norm.
In addition to the previous standard definitions, we introduce the following restrictions to the nonnegative/nonpositive cones:
\begin{align*}
\ell^{1,+}\left(X,\mu\right)\coloneqq \ell^{1}\left(X,\mu\right) \cap \left\{ u \in C(X) \mid u\geq 0 \right\}, \quad
\ell^{1,-}\left(X,\mu\right)\coloneqq -\ell^{1,+}\left(X,\mu\right).
\end{align*}	

We now define the \emph{formal graph Laplacian}  $\Delta \colon \dom\left(\Delta\right) \subseteq C(X) \to C(X)$ associated to the graph $G=(X,w,\kappa,\mu)$ by
		\begin{subequations}
			\begin{equation}
			\dom\left(\Delta\right)\coloneqq\{ u \in C(X) \mid \sum_{y\in X}w(x,y)|u(y)| < \infty \quad \forall x \in X  \},\label{formal_laplacian1}
			\end{equation}
			\begin{align}
			\Delta u(x)&\coloneqq \frac{1}{\mu(x)}\sum_{y\in X} w(x,y)\left(u(x) - u(y)\right) + \frac{\kappa(x)}{\mu(x)}u(x)\label{formal_laplacian2} \\
			&= \Deg(x)u(x) - \frac{1}{\mu(x)}\sum_{y\in X} w(x,y) u(y).\nonumber
			\end{align}
		\end{subequations}
		
		\begin{remark}\label{rem:1}
		We observe that if $u\geq 0$, then $\Delta u(x)$ is always defined  as an extended real-valued function taking values in $[-\infty,\infty)$. Furthermore, if $u \geq 0$, then $u \in \dom\left(\Delta\right)$ if and only if $|\Delta u(x)|< \infty$ for every $x \in X$ if and only if  $\Delta u(x)>-\infty$ for every $x\in X$ if and only if $\sum_{y\in X} w(x,y) u(y)>-\infty$ for every $x\in X$.
		\end{remark}
	\begin{figure}[!b]
	\centering
	\begin{minipage}{.45\textwidth}
		\begin {center}
		\begin {tikzpicture}[-latex ,auto ,node distance =1.5 cm and 1cm ,on grid ,
		semithick ,
		whitestyle/.style={circle,draw,fill=white,minimum size=1cm},
		ghost/.style={circle,fill=white,minimum size=0.7cm}]
		\node[whitestyle] (C){$x_8$};
		\node[whitestyle] (A) [above left=of C] {$x_6$};
		\node[whitestyle] (B) [above right =of C] {$x_5$};
		\node[whitestyle] (D) [below left=of A] {$x_7$};
		\node[whitestyle] (E) [below right=of B] {$x_9$};
		\node[whitestyle] (F) [above right=of A] {$x_{4}$};
		\node[ghost] (g1) [above =of D] {};
		\node[ghost] (g2) [above =of E] {};
		\node[ghost] (g3) [above =of g1] {};
		\node[ghost] (g4) [above =of g2] {};
		\node[whitestyle] (G) [above =of g4] {$x_3$};
		\node[whitestyle] (H) [above =of g3] {$x_0$};
		\node[whitestyle] (I) [above=of H] {$x_1$};
		\node[whitestyle] (L) [above=of G] {$x_2$};
		\path (C) edge [double=black,-] node[] {} (A);
		\path (A) edge [double=black,-] node[] {} (C);
		\path (A) edge [double=black,-] node[] {} (B);
		\path (B) edge [double=black,-] node[] {} (A);
		\path (C) edge [double=black,-] node[] {} (B);
		\path (B) edge [double=black,-] node[] {} (C);
		\path (B) edge [double=black,-] node[] {} (F);
		\path (F) edge [double=black,-] node[] {} (B);
		\path (F) edge [double=black,-] node[] {} (H);
		\path (H) edge [double=black,-] node[] {} (F);
		\path (H) edge [double=black,-] node[] {} (D);
		\path (D) edge [double=black,-] node[] {} (H);
		\path (D) edge [double=black,-] node[] {} (C);
		\path (C) edge [double=black,-] node[] {} (D);
		\path (D) edge [double=black,-] node[] {} (A);
		\path (A) edge [double=black,-] node[] {} (D);
		\path (C) edge [double=black,-] node[] {} (E);
		\path (E) edge [double=black,-] node[] {} (C);
		\path (E) edge[double=black,-] node[] {} (G);
		\path (G) edge[double=black,-] node[] {} (E);
		\path (G) edge[double=black,-] node[] {} (L);
		\path (L) edge[double=black,-] node[] {} (G);
		\path (I) edge[double=black,-] node[] {} (H);
		\path (H) edge[double=black,-] node[] {} (I);
		\path (I) edge[double=black,-] node[above] {} (L);
		\path (L) edge[double=black,-] node[above] {} (I);
		\path (F) edge[double=black,-] node[] {} (G);
		\path (G) edge[double=black,-] node[] {} (F);
		\draw[latex'-latex',double] (G) edge[double=black,-] node [left] {} (H);
	\end{tikzpicture}
\end{center}\end{minipage}
\begin{minipage}{.45\textwidth}
\centering
\begin {tikzpicture}[-latex ,auto ,node distance =1.5 cm and 1cm ,on grid ,
semithick ,
whitestyle/.style={circle,draw,fill=white,minimum size=1cm},
blackstyle/.style ={ circle ,top color =black, bottom color = black,
	draw, white, minimum size =1cm},
gray-orangestyle/.style ={ circle ,top color =lightgray!70 , bottom color = lightgray!70 ,
	draw,  black, text=black, minimum size =1cm},
green-redstyle/.style ={ circle ,top color = darkpastelgreen , bottom color =darkpastelgreen ,
	draw, black , text=black, minimum size=1cm},
greenstyle/.style={circle ,top color =darkspringgreen , bottom color = darkspringgreen,
	draw,black , text=black ,minimum size=1cm},
lightgreenstyle-red/.style={circle ,top color =lightgreen!70 , bottom color = lightgreen!70,
	draw, black, text=black ,minimum size=1cm},
ghost/.style={circle,fill=white,minimum size=0.7cm}]
\node[greenstyle] (C){$x_8$};
\node[greenstyle] (A) [above left=of C] {$x_6$};
\node[greenstyle] (B) [above right =of C] {$x_5$};
\node[lightgreenstyle-red] (D) [below left=of A] {$x_7$};
\node[lightgreenstyle-red] (E) [below right=of B] {$x_9$};
\node[lightgreenstyle-red] (F) [above right=of A] {$x_{4}$};
\node[ghost] (g1) [above =of D] {};
\node[ghost] (g2) [above =of E] {};
\node[ghost] (g3) [above =of g1] {};
\node[ghost] (g4) [above =of g2] {};
\node[gray-orangestyle] (G) [above =of g4] {$x_3$};
\node[gray-orangestyle] (H) [above =of g3] {$x_0$};
\node[whitestyle] (I) [above=of H] {$x_1$};
\node[whitestyle] (L) [above=of G] {$x_2$};
\path (C) edge [double=black,-] node[] {} (A);
\path (A) edge [double=black,-] node[] {} (C);
\path (A) edge [double=black,-] node[] {} (B);
\path (B) edge [double=black,-] node[] {} (A);
\path (C) edge [double=black,-] node[] {} (B);
\path (B) edge [double=black,-] node[] {} (C);
\path (B) edge [double=black,-] node[] {} (F);
\path (F) edge [double=black,-] node[] {} (B);
\path (F) edge [double=black,-] node[] {} (H);
\path (H) edge [double=black,-] node[] {} (F);
\path (H) edge [double=black,-] node[] {} (D);
\path (D) edge [double=black,-] node[] {} (H);
\path (D) edge [double=black,-] node[] {} (C);
\path (C) edge [double=black,-] node[] {} (D);
\path (D) edge [double=black,-] node[] {} (A);
\path (A) edge [double=black,-] node[] {} (D);
\path (C) edge [double=black,-] node[] {} (E);
\path (E) edge [double=black,-] node[] {} (C);
\path (E) edge[double=black,-] node[] {} (G);
\path (G) edge[double=black,-] node[] {} (E);
\path (G) edge[double=black,-] node[] {} (L);
\path (L) edge[double=black,-] node[] {} (G);
\path (I) edge[double=black,-] node[] {} (H);
\path (H) edge[double=black,-] node[] {} (I);
\path (I) edge[double=black,-] node[above] {} (L);
\path (L) edge[double=black,-] node[above] {} (I);
\path (F) edge[double=black,-] node[] {} (G);
\path (G) edge[double=black,-] node[] {} (F);
\draw[latex'-latex',double] (G) edge[double=black,-] node [left] {} (H);
\end{tikzpicture}
\end{minipage}
\caption{Example of a connected graph $G=(X,w,\kappa,\mu)$ (in the left picture) with $X=\left\{x_i \mid i=0,\ldots, 9\right\}$ and a proper subset $A=\left\{x_4,x_5,x_6,x_7,x_8,x_{9}\right\}$. A black line between two nodes $x_i,x_j \in X$ means that $x_i\sim x_j$. In the right picture, the interior $\mathring{A}=\{x_5,x_6,x_8\}$ is colored in green while the interior boundary $\mathring{\partial} A=\{x_4,x_7,x_9\}$ is colored in light green. The nodes in the exterior boundary $\mathbullet{\partial} A = \{x_0, x_3\}\subseteq X\setminus A$  are colored in light gray.}\label{fig:interior-boundary}
\end{figure}
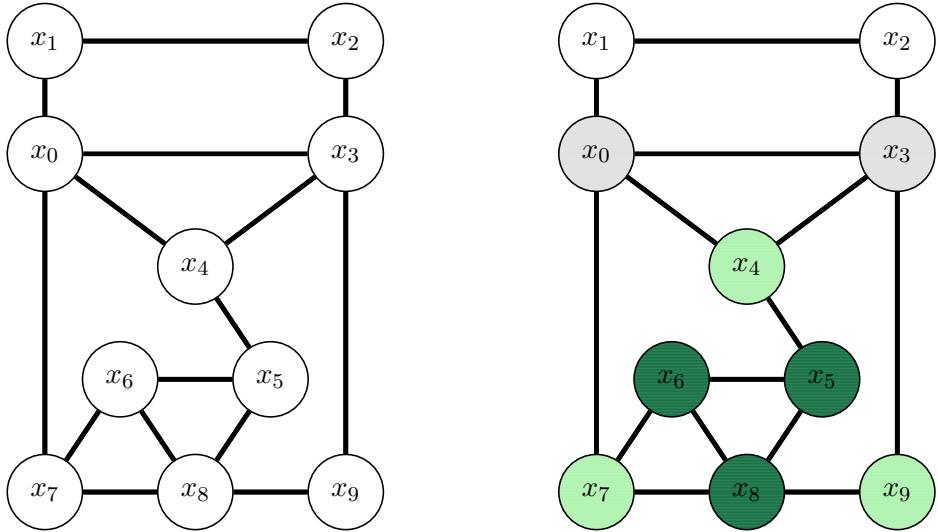	
As we will see in the upcoming sections, it will be useful to deal with subgraphs of a graph. We start by discussing the notion of the interior and two notions of boundary for a subset of the node set. Given a graph $G=(X,w,\kappa,\mu)$ and a subset $A\subset X$, then
	$$
	\mathring{A}\coloneqq\left\{ x \in A \mid x \nsim y \mbox{ for every } y \in X\setminus A \right\}
	$$
	is called the \emph{interior} of $A$ and the elements of $\mathring{A}$ are called \emph{interior nodes} of $A$. On other hand, the sets of nodes
	\begin{align*}
	&\mathring{\partial} A\coloneqq\left\{ x \in A \mid x \sim y  \mbox{ for some }  y \in X\setminus A \right\},\\
	&\mathbullet{\partial} A\coloneqq\left\{ y \in X\setminus A \mid y \sim x  \mbox{ for some }  x \in  A \right\}
	\end{align*}
	are called the \emph{interior boundary} and the \emph{exterior boundary} of $A$, respectively.
	Although these notions are rather standard in the graph setting, we illustrate the definitions
	with an example in Figure~\ref{fig:interior-boundary}. 	

We next introduce the concept of an induced subgraph.
	\begin{definition}[Induced subgraph]\label{def:induced_subgraph}
 We say that a graph $F=(A, w',\kappa', \mu')$ is an \emph{induced subgraph} of $G$, and we write $F\subset G$, if
	\begin{itemize}
		\item $A\subset X$;
		\item $w' \equiv w_{|A\times A}$;
		\item $\kappa'(x) = \kappa(x)$ for every $x \in\mathring{A}$;
		\item $\mu' \equiv \mu_{|A}$
	\end{itemize}
where $w_{|A\times A}$ and $\mu_{|A}$ denote the restrictions of $w$ and $\mu$ to the sets $A\times A$ and $A$, respectively.
	We call $G$ the \emph{host graph} or the \emph{supergraph}. 
	The corresponding formal graph Laplacian for a subgraph $F$ is defined according to \eqref{formal_laplacian1} and \eqref{formal_laplacian2} where the quadruple $(X,w,\kappa,\mu)$ is replaced by $(A,w',\kappa',\mu')$. Observe that we do not require that $\kappa' \equiv \kappa$ on $\mathring{\partial}A$. Different choices of $\kappa'$ on $\mathring{\partial} A$ will produce different subgraphs. We say that $F$ is the \emph{canonical induced subgraph} if $\kappa'=\kappa_{|A}$.
	\end{definition}
	\begin{remark}
	Our notion of induced subgraph is intrinsically related to the killing term $\kappa'$. If we do not consider any $\kappa$, then the definition of induced subgraph is equivalent to the classical one, see for example \cite[Definition 2.2]{estrada2015first}.
	\end{remark}
	Of particular interest is the killing term $\kappa_{\textnormal{dir}}$ that arises from \textquotedblleft Dirichlet boundary conditions.\textquotedblright

	\begin{figure}[!t]
		\centering
		\begin{minipage}{.45\textwidth}
			\begin {center}
			\begin{tikzpicture}[-latex ,auto ,node distance =1.5 cm and 1cm ,on grid ,
				semithick ,
				whitestyle/.style={circle,draw,fill=white,minimum size=1cm},
				blackstyle/.style ={ circle ,top color =black, bottom color = black,
					draw, white, minimum size =1cm},
				white-redstyle/.style ={ circle ,top color =lightgray!70 , bottom color = lightgray!70,
					draw,  black, text=black, minimum size =1cm},
				green-redstyle/.style ={ circle ,top color = lightgreen!70 , bottom color =lightgreen!70 ,
					draw, black , text=black, minimum size=1cm},
				greenstyle/.style={circle ,top color =darkspringgreen , bottom color = darkspringgreen,
					draw,black , text=black ,minimum size=1cm},
				ghost/.style={circle,fill=white,minimum size=0.7cm}]
				\node[greenstyle] (C){$x_8$};
				\node[greenstyle] (A) [above left=of C] {$x_6$};
				\node[greenstyle] (B) [above right =of C] {$x_5$};
				\node[green-redstyle] (D) [below left=of A] {$x_7$};
				\node[green-redstyle] (E) [below right=of B] {$x_9$};
				\node[green-redstyle] (F) [above right=of A] {$x_{4}$};
				\node[ghost] (g1) [above =of D] {};
				\node[ghost] (g2) [above =of E] {};
				\node[ghost] (g3) [above =of g1] {};
				\node[ghost] (g4) [above =of g2] {};
				\node[white-redstyle] (G) [above =of g4] {$x_3$};
				\node[white-redstyle] (H) [above =of g3] {$x_0$};
				\node[whitestyle] (I) [above=of H] {$x_1$};
				\node[whitestyle] (L) [above=of G] {$x_2$};
				\path (C) edge [double=black,-] node[] {} (A);
				\path (A) edge [double=black,-] node[] {} (C);
				\path (A) edge [double=black,-] node[] {} (B);
				\path (B) edge [double=black,-] node[] {} (A);
				\path (C) edge [double=black,-] node[] {} (B);
				\path (B) edge [double=black,-] node[] {} (C);
				\path (B) edge [double=black,-] node[] {} (F);
				\path (F) edge [double=black,-] node[] {} (B);
				\path (D) edge [double=black,-] node[] {} (C);
				\path (C) edge [double=black,-] node[] {} (D);
				\path (D) edge [double=black,-] node[] {} (A);
				\path (A) edge [double=black,-] node[] {} (D);
				\path (C) edge [double=black,-] node[] {} (E);
				\path (E) edge [double=black,-] node[] {} (C);
				\path (G) edge[double=black,-] node[] {} (L);
				\path (L) edge[double=black,-] node[] {} (G);
				\path (I) edge[double=black,-] node[] {} (H);
				\path (H) edge[double=black,-] node[] {} (I);
				\path (I) edge[double=black,-] node[above] {} (L);
				\path (L) edge[double=black,-] node[above] {} (I);
				\draw[latex'-latex',double] (G) edge[double=black,-] node [left] {} (H);
			\end{tikzpicture}
	\end{center}\end{minipage}
	\begin{minipage}{.45\textwidth}
		\centering
		\begin {tikzpicture}[-latex ,auto ,node distance =1.5 cm and 1cm ,on grid ,
		semithick ,
		whitestyle/.style={circle,draw,fill=white,minimum size=1cm},
		blackstyle/.style ={ circle ,top color =black, bottom color = black,
			draw, text=white, minimum size =1cm},
		white-redstyle/.style ={ circle ,top color =lightgray!70 , bottom color = lightgray!70 ,
			draw, double=red,very thick, black, text=black, minimum size =1cm},
		green-redstyle/.style ={ circle ,top color = lightgreen!70 , bottom color =lightgreen!70 ,
			draw,double=red,very thick, black , text=black, minimum size=1cm},
		greenstyle/.style={circle ,top color =darkspringgreen , bottom color = darkspringgreen,
			draw,black , text=black ,minimum size=1cm},
		ghost/.style={circle,fill=white,minimum size=0.7cm}]
		\node[greenstyle] (C){$x_8$};
		\node[greenstyle] (A) [above left=of C] {$x_6$};
		\node[greenstyle] (B) [above right =of C] {$x_5$};
		\node[green-redstyle] (D) [below left=of A] {$x_7$};
		\node[green-redstyle] (E) [below right=of B] {$x_9$};
		\node[green-redstyle] (F) [above right=of A] {$x_{4}$};
		\node[ghost] (g1) [above =of D] {};
		\node[ghost] (g2) [above =of E] {};
		\node[ghost] (g3) [above =of g1] {};
		\node[ghost] (g4) [above =of g2] {};
		\node[white-redstyle] (G) [above =of g4] {$x_3$};
		\node[white-redstyle] (H) [above =of g3] {$x_0$};
		\node[whitestyle] (I) [above=of H] {$x_1$};
		\node[whitestyle] (L) [above=of G] {$x_2$};

		\path (C) edge [double=black,-] node[] {} (A);
		\path (A) edge [double=black,-] node[] {} (C);
		\path (A) edge [double=black,-] node[] {} (B);
		\path (B) edge [double=black,-] node[] {} (A);
		\path (C) edge [double=black,-] node[] {} (B);
		\path (B) edge [double=black,-] node[] {} (C);
		\path (B) edge [double=black,-] node[] {} (F);
		\path (F) edge [double=black,-] node[] {} (B);
		\path (D) edge [double=black,-] node[] {} (C);
		\path (C) edge [double=black,-] node[] {} (D);
		\path (D) edge [double=black,-] node[] {} (A);
		\path (A) edge [double=black,-] node[] {} (D);
		\path (C) edge [double=black,-] node[] {} (E);
		\path (E) edge [double=black,-] node[] {} (C);
		\path (G) edge[double=black,-] node[] {} (L);
		\path (L) edge[double=black,-] node[] {} (G);
		\path (I) edge[double=black,-] node[] {} (H);
		\path (H) edge[double=black,-] node[] {} (I);
		\path (I) edge[double=black,-] node[above] {} (L);
		\path (L) edge[double=black,-] node[above] {} (I);
		\draw[latex'-latex',double] (G) edge[double=black,-] node [left] {} (H);
		
		\path (H) edge [double=red, dashed,-] node[midway,below] {$\textcolor{red2}{b_{\textnormal{dir}}}$} (F);
		\path (H) edge [double=red, dashed,-] node[] {$\textcolor{red2}{b_{\textnormal{dir}}}$} (D);
		\path (E) edge[double=red, dashed,-] node[] {$\textcolor{red2}{b_{\textnormal{dir}}}$} (G);
	\path (F) edge[double=red, dashed,-] node[] {$\textcolor{red2}{b_{\textnormal{dir}}}$} (G);
		\draw[latex'-latex',double] (G) edge[black,-] node [left] {} (H);
	\end{tikzpicture}
\end{minipage}
\caption{Continuation of the example in Figure \ref{fig:interior-boundary}. The left picture gives a canonical induced subgraph $F\subset G$: In general, $F$ is not ``affected'' by the complementary set $X\setminus A$ since the weights of the edges lying in $X\setminus A$ do not influence the graph $F$. The right picture gives instead a Dirichlet subgraph $G_{\textnormal{dir}}$: Due to the presence of the Dirichlet weight function $b_{\textnormal{dir}}$ the subgraph $G_{\textnormal{dir}}$ is still affected by the complementary set $X\setminus A$.}\label{fig:dirichlet_subgraph}
\end{figure}
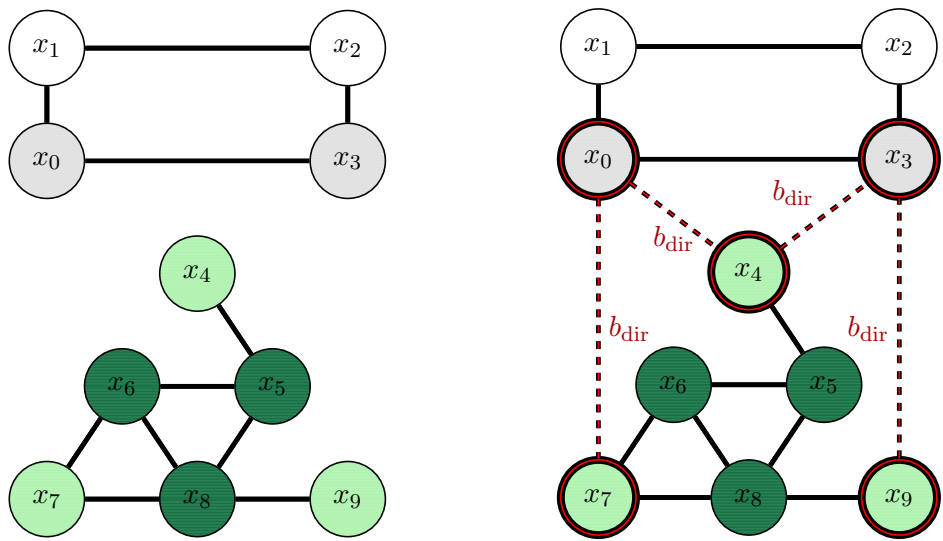

	\begin{definition}[Dirichlet subgraph]\label{def:dir_subgraph}
		An induced subgraph
		$$
		G_{\textnormal{dir}} \coloneqq\left(A, w_{|A\times A}, \kappa_{\textnormal{dir}}, \mu_{|A}\right)\subset G
		$$
		is called a \emph{Dirichlet subgraph} if
	\begin{equation}\label{def:dir_potential}
	\begin{cases}
	\kappa_{\textnormal{dir}}(x) \coloneqq \kappa_{|A}(x) + b_{\textnormal{dir}}(x),\\
	 b_{\textnormal{dir}}(x)\coloneqq\sum_{y \in  \mathbullet{\partial} A}w(x,y) = \sum_{y \not \in A}w(x,y).
	\end{cases}	
		\end{equation}
We note that $b_{\textnormal{dir}} \colon A \to \R$ is finite because of \ref{assumption:degree}. We call $b_{\textnormal{dir}}$ the \emph{boundary (Dirichlet) weight-function} and  $\kappa_{\textnormal{dir}}$ the \emph{Dirichlet killing term}. Clearly, $b_{\textnormal{dir}}(x)= 0$ for every $x \in \mathring{A}$. We will denote by $\Delta_{\textnormal{dir}}$ the graph Laplacian of $G_{\textnormal{dir}}$  in order to distinguish it from the graph Laplacian of $G$.
	\end{definition}
	
The Dirichlet killing term describes the edge deficiency of nodes in $G_{\textnormal{dir}}$ compared to the same nodes in $G$, see Figure \ref{fig:dirichlet_subgraph}. The name ``Dirichlet'' in the above definition comes from the following observation, see, for example,
 \cite[pg.~197 and Proposition 2.4]{keller2012dirichlet} and \cite[Proposition 2.23]{keller2021graphs}: Let $\boldsymbol{\mathfrak{i}} \colon  C(A) \hookrightarrow  C(X)$ be the canonical embedding and $\boldsymbol{\pi} \colon C(X) \to C(A)$ be the canonical projection, i.e.,
\begin{equation}\label{eq:embedding-projection}
	\boldsymbol{\mathfrak{i}}v(x)= \begin{cases}
	v(x) & \mbox{if } x\in A,\\
	0 & \mbox{if } x\in X\setminus A,
	\end{cases}
	\qquad
	\boldsymbol{\pi}u= u_{|A}.
\end{equation}
Under some mild assumptions, it is almost straightforward to prove, see Lemma \ref{lem:A1}, that
	\begin{itemize}
		\item $\Deltadir v(x)=\Delta\boldsymbol{\mathfrak{i}}v(x)$ for $v\in \dom(\Deltadir)$ and $x \in A$;
		\item $\Delta u(x) = \Deltadir\boldsymbol{\pi}u(x)$ for $u\in \dom\left(\Delta\right)\cap \left\{u \in C(X) \mid u \equiv 0 \mbox{ on } X\setminus A \right\}$ and $x \in A$.
	\end{itemize}
	Therefore, the Dirichlet graph Laplacian $\Deltadir$ can be viewed as the restriction of $\Delta$ having imposed Dirichlet conditions on the exterior boundary of $A$.

\subsection{m-accretive operators}\label{ssec:m-accretivity}	
In preparation for Section \ref{sec:existence_weak}, we introduce here the fundamental tools that will play a major role in the proofs of existence and uniqueness of solutions to the \ref{Model_Equation_graph} in the discrete setting. We begin by giving two equivalent definitions of accretivity. As a reference for this topic, see for example \cite{deimling2010nonlinear}.

\begin{definition}[Accretive and m-accretive operators]\label{def:m-accretivity}
If $\mathfrak{E}=(E,\|\cdot\|)$ is a real Banach space and $\mathcal{L} \colon \dom(\mathcal{L})\subseteq E \to E$ is a (not necessarily linear) operator, then $\mathcal{L}$ is said to be \emph{accretive} if $\mathcal{L}$ satisfies one of the following equivalent conditions:
\begin{enumerate}[(i)]
	\item\label{m-accretivity1} $\left\|(u - v) + \lambda \left(\mathcal{L}u - \mathcal{L}v\right) \right\|\geq \|u - v\|$ for every $u,v \in \dom\left(\mathcal{L}\right)$ and for every $\lambda >0$.
	\item\label{m-accretivity2} $\langle \mathcal{L}u -\mathcal{L}v, u - v \rangle_+ \geq 0$ for every $u,v \in \dom\left(\mathcal{L}\right)$ where for $z, k \in E$
	$$
	\langle z, k \rangle_+\coloneqq\|k\|\lim_{\lambda\to 0^+} \lambda^{-1}\left( \left\|k + \lambda z\right\| - \|k\|  \right).
	$$
\end{enumerate}
Concerning the well-posedness of condition \ref{m-accretivity2} and its application to $\ell^p$-spaces, see Remark \ref{rem:accretivity_l^2} below.
An accretive operator $\mathcal{L}$ is called \emph{$m$-accretive} if $\operatorname{id} + \lambda\mathcal{L}$ is surjective for every $\lambda >0$.
\end{definition}
Accretive operators are called monotone operators in the Hilbert space setting. Accretivity is a way to extend the property of monotonicity of real-valued functions of a real variable to spaces with a more complex structure. This follows by the trivial observation that a map $f \colon \dom(f) \subseteq \R \to \R$ is monotone (increasing) if and only if $\left(f(s_1)-f(s_2)\right)\left(s_1-s_2\right)\geq 0$ for all $s_1,s_2 \in \dom(f)$.

Let us highlight that $m$-accretivity is related to the self-adjointness of linear operators in the Hilbert case setting. Indeed, a linear operator $\mathcal{L}$ on a Hilbert space  is self-adjoint and nonnegative if and only if $\mathcal{L}$ is symmetric, closed and $m$-accretive (by the Minty theorem, m-accretive and maximal monotone are equivalent properties in Hilbert spaces), see \cite[ Problem V.3.32]{kato2013perturbation}. In this context, we note that there has been recent interest in the graph setting concerning the essential self-adjointness of the formal Laplacian and related operators restricted to finitely supported functions, see, for example, \cite{wojciechowski2008stochastic,milatovic2011essential,keller2012dirichlet,huang2013note,guneysu2014generalized,milatovic2014self,milatovic2015maximal,guneysu2016feynman,schmidt2020existence}. Concerning the $m$-accretivity of the graph Laplacian  we also highlight a couple of recent results: The first is obtained
in \cite{milatovic2015maximal}, where the authors establish the $m$-accretivity on $\ell^p(X,\mu)$ for $1\leq p <\infty$ in the more general setting of Hermitian vector bundles, under some hypotheses on the graph.  The second result is obtained in \cite{anne2020m} where the authors prove a criterion for the $m$-accretivity of a graph Laplacian (not necessarily self-adjoint) on directed graphs in the Hilbert case.

\begin{remark}\label{rem:accretivity_l^2}
Let us observe that condition \ref{m-accretivity2} in Definition \ref{def:m-accretivity} is well-posed. First of all, for $\lambda>0$
$$
-\|z\|\leq  \frac{\left\|k + \lambda z\right\| - \|k\|}{\lambda}   \leq \|z\|.
$$
Then we observe that, for every $0<s<\lambda$,
\begin{align*}
\| k + sz\| - \|k\| = \left\|\left(1-\frac{s}{\lambda}\right)k +  \frac{s}{\lambda} (k+\lambda z) \right\|	- \|k\| &\leq \left(1-\frac{s}{\lambda}\right)\|k\| +\frac{s}{\lambda}\|k+\lambda z\| - \|k\|\\
&=\frac{s}{\lambda} \left(\| k + \lambda z\| - \|k\|\right),
\end{align*}
i.e., the map $\lambda \mapsto \lambda^{-1}\left(\|k + \lambda z \| - \|k\|\right)$ is monotone increasing in $\lambda>0$. Therefore,
$$
\lim_{\lambda\to 0^+}\lambda^{-1}\left(\|k + \lambda z \| - \|k\|\right)
$$
exists and belongs to $[-\|z\|, \|z\|]$. 

Now, let us fix $E=\ell^p(X,\mu)$ with $p\in [1,\infty)$. By the convexity of the map $t\mapsto |t|^p$, for every $\lambda\in (0,1]$ we get
\begin{align*}
&\frac{|k \pm \lambda z |^p - |k|^p}{\lambda}= \frac{|(1-\lambda)k + \lambda(k\pm  z) |^p - |k|^p}{\lambda}\leq |k \pm z|^p - |k|^p
\end{align*}
and
$$|k|^p - |k - \lambda z |^p\leq |k + \lambda z |^p -|k|^p  $$
where the second inequality can be easily derived by $|f+g|^p\leq 2^{p-1}(|f|^p+|g|^p)$ with $f\coloneqq k-\lambda z$ and $g\coloneqq k+\lambda z$. Combining these inequalities gives
$$
|k|^p - |k - z|^p\leq \frac{ |k|^p - | k - \lambda z|^p}{\lambda} \leq\frac{|k + \lambda z |^p - |k|^p}{\lambda}\leq |k + z|^p - |k|^p \quad \forall\, \lambda\in (0,1],
$$
that is, $\lambda^{-1}| |k + \lambda z |^p - |k|^p |$ is dominated by an integrable function. Then, by the mean value theorem and dominated convergence, for every $p\geq1$ (and $\|k\|_p\neq 0$) we get 
\begin{align*}
\lim_{\lambda\to 0^+}\lambda^{-1}\left(\|k + \lambda z \|_p - \|k\|_p\right) &=\lim_{\lambda\to 0^+}\lambda^{-1}\left((\|k + \lambda z \|^p_p)^{1/p} - (\|k\|^p_p)^{1/p}\right) \\
&= \frac{1}{p}(\|k\|^p_p)^{\frac{1}{p}-1}\lim_{\lambda\to 0^+}\sum_{x \in X} \frac{|k(x) + \lambda z(x) |^p - |k(x)|^p}{\lambda}\mu(x)\\
&=\begin{cases}
	\|k\|_p^{1-p}\sum_{x \in X}z(x)|k(x)|^{p-1}\operatorname{sgn}(k(x)) \mu(x) &\mbox{for } p>1,\\
\sum\limits_{\substack{x\in X\colon\\ k(x)= 0}}|z(x)|\mu(x) + \sum\limits_{\substack{x\in X\colon\\ k(x)\neq 0}} z(x)\operatorname{sgn}(k(x))\mu(x) &\mbox{for } p=1
\end{cases}
\end{align*}
where
\begin{equation}\label{eq:sgn}
	\operatorname{sgn}(s)\coloneqq\begin{cases}
		1 & \mbox{if } s>0,\\
		0 & \mbox{if } s=0,\\
		-1 & \mbox{if } s<0.
	\end{cases}
\end{equation}
Summarizing, for $E=\ell^p(X,\mu)$ with $p\in [1,\infty)$,
\begin{equation}\label{accretivity_for_lp}
\langle z, k \rangle_+  = \begin{cases}
	\|k\|_p^{2-p} \sum\limits_{x \in X} z(x)k(x)|k(x)|^{p-2}\mu(x) & \mbox{if } p>1,\\
	\|k\|_1\left( \sum\limits_{\substack{x\in X\colon\\ k(x)= 0}}|z(x)|\mu(x) + \sum\limits_{\substack{x\in X\colon\\ k(x)\neq 0}} z(x)\operatorname{sgn}(k(x))\mu(x) \right)& \mbox{if } p=1.
\end{cases}
\end{equation}
\end{remark}

A simple example of an $m$-accretive operator is the graph Laplacian on finite graphs with respect to the $\ell^p$-norm. This result should be well-known but for completeness we give a proof in the following proposition.

\begin{proposition}\label{lem:m-accretivity_for_finite_graphs}
	Let $G$ be a finite graph. Then, the graph Laplacian $\Delta$ on $\ell^p(X,\mu)$ is $m$-accretive for $p\geq 1$. In particular,
	$$
	\left\| (\operatorname{id} + \lambda \Delta)u \right\|_p\geq \|u \|_p \qquad \mbox{for every } u \in C(X), \, \lambda >0.
	$$
\end{proposition}
\begin{proof}
Fix $p\in (1,\infty)$. Applying the linearity of $\Delta$ and \eqref{accretivity_for_lp} of Remark \ref{rem:accretivity_l^2}, $\Delta$ is accretive if and only if
	$$
	\|u\|_p^{2-p}\sum_{x\in X} \Delta u(x)  u(x)|u(x)|^{p-2}\mu(x)\geq 0.
	$$	
Using the fact that the sum is finite, we have
\begin{align*}
\sum_{x\in X} \Delta u(x) & u(x)|u(x)|^{p-2}\mu(x) \\
& = \sum_{x\in X} u(x)|u(x)|^{p-2} \sum_{y\in X}w(x,y)(u(x)-u(y)) + \sum_{x\in X}\kappa(x)|u(x)|^p\\
&\geq \sum_{x\in X} \sum_{y\in X}w(x,y)(|u(x)|^{p}-u(y)u(x)|u(x)|^{p-2})\\
&\geq \frac{1}{2} \sum_{x,y\in X}w(x,y)(|u(x)|^p + |u(y)|^p - |u(x)|^{p-1}|u(y)| - |u(x)||u(y)|^{p-1}).
\end{align*}
The conclusion now follows from the following inequality
$$
a^p + b^p - a^{p-1}b - ab^{p-1} \geq 0 \qquad \forall\; a,b \geq 0, \quad \forall\, p \in (1,\infty)
$$
which can be established by elementary calculus as we now show. The inequality holds for $a=0$ or $b=0$, so assuming that $b>0$, dividing through by $b^p$  and setting $t\coloneqq a/b$, it is equivalent to prove that
$$
\beta(t) \coloneqq t^p +1 - t^{p-1} - t \geq 0 \qquad \mbox{for } t \geq 0.
$$
Note that $\beta(0)=1$ and $\beta(t) \to \infty$ as $t \to \infty$. We have
$$
\beta'(t) = pt^{p-1} -(p-1)t^{p-2}-1
$$
and thus
$$
\beta'(t)<0 \mbox{ for all } t  \mbox{ small}, \quad \beta'(1)=0, \quad \beta'(t)\to \infty \mbox{ as } t \to \infty.
$$
Moreover,
$$
\beta''(t)= p(p-1)t^{p-2} -(p-1)(p-2)t^{p-3}= p(p-1)t^{p-3}\left(t-\frac{p-2}{p}\right).
$$
Thus, if $1<p\leq 2$, then $\beta''\geq 0$ and $\beta'$ is increasing. If $p>2$, then
$$
\beta''(t)<0 \mbox{ for } t<\frac{p-2}{p}<1, \qquad \beta''(t)\geq 0 \mbox{ for } t \geq \frac{p-2}{p}.
$$
Hence, $\beta'$ is decreasing until ${(p-2)}/{p}$ and increasing afterwards. In any case, $\beta'(t)<0$ for $t<1$ and $\beta'(t)>0$ for $t>1$, so
$$
\min_{t>0} \beta(t) = \beta(1) = 0.
$$

By the equivalence between \ref{m-accretivity2} and \ref{m-accretivity1} of Definition \ref{def:m-accretivity}, we then have
$$
	\left\| (\operatorname{id} + \lambda \Delta)u \right\|_p\geq \|u \|_p \qquad \mbox{for every } u \in C(X), \, \lambda >0.
	$$
We conclude the lemma for $p\in (1,\infty)$ by noticing that $\operatorname{id}+\lambda\Delta$ injective and $C(X)$ finite imply that $\operatorname{id}+\lambda\Delta$ surjective by linearity. The case $p=1$ is addressed in the more general case of Corollary~\ref{cor:accretivityL-finite} in Appendix~\ref{sec:appendix2}.
\end{proof}

As a final comment, we observe that if two operators are accretive on a given Banach space, this will not automatically imply that the composition (if defined) of the two operators is accretive. See the following simple example.

\begin{example}\label{ex:1}
Consider the finite birth-death chain $G=\left(X,w,\kappa,\mu\right)$ (see Figure \ref{fig:ex:1}) with
\begin{itemize}
	\item $X= \left\{ x_1, x_2, x_3, x_4  \right\}$;
	\item $w(x_i,x_j)=1$ if and only if $|i-j|=1$ and zero otherwise;
	\item $\kappa \equiv 0$;
	\item $\mu \equiv 1$.
\end{itemize}
 Define now $\phi \colon \R \to \R$ by $\phi(s)\coloneqq s|s|^3$. Since $\phi$ is monotone, thanks to Remark \ref{rem:accretivity_l^2}, $\Phi$, the canonical extension of $\phi$ to $\ell^2(X,\mu)$, is accretive. Consider now the graph Laplacian $ \Delta \colon \ell^2\left(X,\mu\right)\to \ell^2\left(X,\mu\right)$ associated to $G$ which in this case acts as
$$
\Delta u(x_i) =\begin{cases}
u(x_{1}) - u(x_{2}) & \mbox{if } i=1, \\
 2u(x_{i}) -u(x_{i-1}) - u(x_{i+1}) & \mbox{if } i = 2, 3,\\
u(x_{4}) - u(x_{3}) & \mbox{if } i=4.
\end{cases}
$$
By Proposition \ref{lem:m-accretivity_for_finite_graphs}, $\Delta$ is accretive. 

On the other hand, the operator $\mathcal{L}\coloneqq\Delta\Phi$ is not accretive. Indeed, a computation shows that if $u$ and $v$ are defined by
$$
 u(x_i)=\begin{cases}
3, & \mbox{for } i=1,\\
4, & \mbox{for } i=2,\\
0 & \mbox{otherwise},
\end{cases}\qquad v(x_i)=\begin{cases}
3, & \mbox{for } i=2,\\
0 & \mbox{otherwise},
\end{cases}
$$
then
$$
\langle \mathcal{L} u-  \mathcal{L} v, u-v  \rangle_{\ell^2}=-13 <0
$$
and therefore $\mathcal{L}$ is not accretive as claimed.
\end{example}
As it is shown in Corollary \ref{cor:accretivityL-finite}, the operator $\mathcal{L}=\Delta\Phi$  is accretive on $\ell^1(X,\mu)$ for every finite graph. This result and the above example show that accretivity is a property related not only to the action of the operator but to the norm on the underlying space as well.

\begin{figure}
	\begin {center}
	\begin{tikzpicture}[-latex, auto,node distance =3 cm ,on grid ,
	semithick ,
	state2/.style ={ circle ,top color =white , bottom color = gray ,
		draw,black , text=black , minimum width =1 cm},
	state3/.style ={ circle ,top color =white , bottom color = white ,
		draw,white , text=white , minimum width =1 cm},
	state/.style={circle ,top color =white , bottom color = white,
		draw,black , text=black , minimum width =1 cm}]
	
	\node[state] (E) [] {$x_{1}$};
	\node[state] (F) [right=of E] {$x_{2}$};
	\node[state] (G) [right=of F] {$x_3$};
	\node[state] (H) [right=of G] {$x_{4}$};

	\path (G) edge [black,-] node[] {$w(x_3,x_4)$} (H);
	\path (H) edge [black,-] node[] {} (G);
	\path (G) edge [black,-] node[] {} (F);
	\path (F) edge [black,-] node[] {$w(x_2,x_3)$} (G);
	\path (E) edge [black,-] node[] {$w(x_1,x_2)$} (F);
	\path (F) edge [black,-] node[] {} (E);
	\end{tikzpicture}
	\caption{Visual representation of the graph $G$ of the Example \ref{ex:1}.}\label{fig:ex:1}
\end{center}
\end{figure}
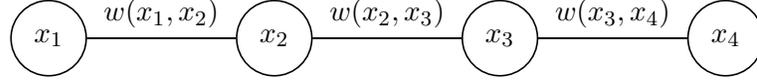

\section{The Cauchy model problem}\label{sec:Cauchy_model_problem}

Let $f \colon (0,T)\times X \to \R$ and $u_0 \colon X \to \R$. Given a graph $G$, let us consider the following Cauchy problem for the generalized porous medium equation (GPME) (or filtration equation):

\vspace{0.3cm}
\noindent\textbf{Problem:}
		\begin{equation}\tag*{Cauchy-GPME}\label{eq:C-D}
		\begin{cases}
		\partial_t u(t,x) + \Delta\Phi u(t,x) = f(t,x) & \mbox{for every } (t,x) \in (0,T)\times X,\\
		\lim_{t\to 0^+} u(t,x) = u_0(x) & \mbox{for every } x\in X
		\end{cases}
		\end{equation}	
where $\Phi \colon C(X) \to C(X)$ is the canonical extension of a function $\phi \colon  \R \to \R$ such that $\phi$ is strictly monotone increasing, $\phi(\R)=\R$ and $\phi(0)=0$. If $\phi(s)= s^m \coloneqq s|s|^{m-1}$, then we will call the above equation the porous medium equation (PME) when $m>1$ and the fast diffusion equation (FDE) when $0<m<1$. Clearly, when $m=1$ and $f \equiv 0$ we recover the classic heat equation. The function $f$ is called the \emph{forcing term}. For notational convenience, we specify the time interval $(0,T)$ but everything we will say can be generalized to any (not necessarily bounded above) interval $(a,b)\subset \R$.

\vspace{0.3cm}
\noindent The \textquotedblleft$+$\textquotedblright \, sign in our equation comes from the fact that we are considering the formal graph Laplacian which corresponds to minus the second derivative in the Euclidean case.

We now introduce  the various classes of solutions for the \ref{eq:C-D} problem in order of increasing regularity. The  weakest notion of solution is obtained by means of a discretization and approximation procedure in the time variable. More precisely, we first need to discretize the time interval $(0,T)$ with respect to the forcing term $f$ such that the corresponding time-discretization $\boldsymbol{f}_n$ of $f$ is \textquotedblleft close\textquotedblright to $f$ in a way that will be made clear next.

\begin{definition}[$\epsilon$-discretization]\label{def:epsilon-discretization}
	Given a time interval $[0,T]$ with $T<\infty$ and a forcing term $f \in L^1([0,T] ; \ell^1\left(X,\mu\right))$, we define a partition of the time interval
	$$
	\mathcal{T}_n \coloneqq \left\{ \{t_k\}_{k=0}^n \mid 0= t_0<t_1<\ldots<t_n\leq T \right\}
	$$
	and a time-discretization $\boldsymbol{f}_n$ of $f$
	$$
	\boldsymbol{f}_n\coloneqq\left\{ \{f_k\}_{k=1}^{n} \mid f_k \in \ell^1(X,\mu),\;  f_k(x)\coloneqq f(t_k,x)\right\}.
	$$
	Having fixed $\epsilon >0$, we call $\mathcal{D}_\epsilon\coloneqq (\mathcal{T}_n,\boldsymbol{f}_n)$ an \emph{$\epsilon$-discretization} of $([0,T]; f)$ if
	\begin{itemize}
		\item $t_k-t_{k-1} \leq \epsilon$ for every $k=1,\ldots, n$ and $T- t_n \leq \epsilon$;
		\item $\sum_{k=1}^{n} \int_{t_{k-1}}^{t_{k}} \left\| f(t) - f_k \right\| dt \leq \epsilon$.
	\end{itemize}
\end{definition}
\begin{remark}\label{rem:existence_discretization}
 Definition \ref{def:epsilon-discretization} is well-posed. If $f\in L^1([0,T] ; \ell^1\left(X,\mu\right))$, then for every $\epsilon >0$ there exists an $\epsilon$-discretization $\mathcal{D}_\epsilon$ of $([0,T]; f)$, see \cite[Lemma 4.1]{evans1977nonlinear}.
 \end{remark}

Now, given an $\epsilon$-discretization $\mathcal{D}_\epsilon$, consider the following system of difference equations which arises from an implicit Euler-discretization of the \ref{eq:C-D}:

\begin{equation}\label{implicit_Euler}
\frac{u_k - u_{k-1}}{\lambda_k} + \Delta\Phi u_k = f_k, \qquad
\lambda_k\coloneqq t_k - t_{k-1}\mbox{ and } k=1,\ldots,n
\end{equation}
with $u_0$ given. Writing $\mathcal{L} = \Delta\Phi $, we then require that every $u_k$ belongs to
\begin{equation*}
\dom(\mathcal{L})=\left\{ u \in  \ell^{1}\left(X,\mu\right) \mid  \Phi u\in \dom\left(\Delta\right), \Delta\Phi u \in  \ell^{1}\left(X,\mu\right)  \right\}.
\end{equation*}

\begin{definition}[$\epsilon$-approximate solution]\label{epsilon-approximation}
If the system \eqref{implicit_Euler} admits a solution $\boldsymbol{u}_\epsilon =\left\{u_k\right\}_{k=1}^n$ such that $u_k\in \dom(\mathcal{L})$ for every $k=1,\ldots, n$, then we define $u_\epsilon$ as the piecewise constant function
\begin{equation}\label{epsilon_approximation}
u_\epsilon(t)\coloneqq \begin{cases}
	\sum_{k=1}^n u_k\mathds{1}_{(t_{k-1},t_k]}(t) & \mbox{for } t\in (0, t_n],\\
	u_0 & \mbox{for } t=0
\end{cases}
\end{equation}
and we call $u_\epsilon$ an \emph{$\epsilon$-approximate solution} of the \ref{eq:C-D} (subordinate to $\mathcal{D}_\epsilon$).
\end{definition}
We then have the following definition of a \emph{mild solution} which first appeared as a formal definition in \cite{crandall1980regularizing}. It can be viewed as a uniform limit of ``numerical approximations'' to solutions of the \ref{eq:C-D} obtained by the system of difference equations \eqref{implicit_Euler}.

\begin{definition}[Mild solution]\label{def:weak_solution}
If $T<+\infty$, we say that $u \colon [0,T]\to \ell^1\left(X,\mu\right)$    is a \emph{mild solution} of the  \ref{eq:C-D} problem if $u \in C\left( [0,T]; \ell^1\left(X,\mu\right)\right)$ and $u$ is obtained as a uniform limit of $\epsilon$-approximate solutions. Namely, for every $\epsilon>0$ there exists an $\epsilon$-discretization  $\mathcal{D}_\epsilon$ of $([0,T]; f)$, as in Definition \ref{def:epsilon-discretization}, and an $\epsilon$-approximate solution $u_\epsilon$  subordinate to $\mathcal{D}_\epsilon$, as in \eqref{epsilon_approximation}, such that
	\begin{equation*}
	\left\|u(t) - u_\epsilon(t) \right\| < \epsilon \qquad \mbox{for every } t\in [0,t_n]\subseteq [0,T].
	\end{equation*}
	
If $T = +\infty$, then  we say that $u$ is a mild solution of the \ref{eq:C-D} if the restriction of $u$ to each compact subinterval $[0,a]\subset [0,+\infty)$ is a mild solution of the  \ref{eq:C-D} on $[0,a]$.
\end{definition}

We next introduce two further classes of solutions, namely, strong and classic solutions. Following the definitions, we will discuss the relationship between these notions.
\begin{definition}[Strong solution]\label{def:strong_solution}
	We say that  $u \colon [0,T] \to \ell^1\left(X,\mu\right)$  is a \emph{strong solution} of the \ref{eq:C-D} problem if
	\begin{itemize}
		\item $u(t) \in \overline{\dom(\mathcal{L})}$ for every $t \in [0,T]$;
		\item $u \in C\left([0,T]; \ell^1\left(X,\mu\right)\right)\cap W^{1,1}_{\textnormal{loc}}((0,T);\ell^1\left(X,\mu\right))$;
		\item $\partial_t u(t,x) + \Delta\Phi u(t,x) = f(t,x)$ for almost every $t\in (0,T)$;
		\item $u(0)=u_0$.
	\end{itemize}
\end{definition}

\begin{definition}[Classic solution]\label{def:classical_solution}
	We say that $u \colon [0,T] \to \ell^1\left(X,\mu\right)$ is a \emph{classic solution} of the \ref{eq:C-D} problem if
	\begin{itemize}
		\item $u(t) \in \overline{\dom(\mathcal{L})}$ for every $t \in [0,T]$;
		\item $u \in C\left( [0,T]; \ell^1\left(X,\mu\right)\right)\cap C^1\left( (0,T); \ell^1\left(X,\mu\right)\right)$;
		\item $\partial_t u(t,x) + \Delta\Phi u(t,x) = f(t,x)$ for every $t\in (0,T)$;
		\item $u(0)=u_0$.
	\end{itemize}
\end{definition}
In the literature, mild solutions are also known by other names: they are called \emph{limit solutions} in \cite{lakshmikantham1981nonlinear} and \emph{weak solutions} in \cite{kobayasi1984nonlinear}. In \cite{benilan1972solutions}, P. B{\'e}nilan and H. Br{\'e}zis introduced the definition of \emph{faible} (i.e., \emph{weak}) solutions of the abstract Cauchy problem
\begin{equation}\label{abstract_Cauchy}
	\begin{cases}
		\partial_tu(t) + \mathcal{A}u(t) = f(t) &\mbox{for } t\in (0, T),\\
		u(0)= u_0
	\end{cases}
\end{equation}
as a uniform limit of strong solutions $u_n$ of \eqref{abstract_Cauchy} with $f$ replaced by $f_n$ where $f_n\to f$ in $L^1([0,T]; E)$. 

Clearly, a classic solution is a strong solution. The fact that a strong solution is a mild solution  assuming that $f \in L^1_{\textnormal{loc}}([0,T] ; \ell^1\left(X,\mu\right))$ is a standard result, e.g. \cite[Theorem 1.4]{benilan1988evolution}. Therefore, assuming $f$ is strongly measurable and locally Bochner integrable, we have compatibility of the three different definitions in the sense that classic solution $\Rightarrow$ strong solution $\Rightarrow$ mild solution in order of descending ``regularity.''  As a final remark, we observe that a mild solution may not be differentiable and does not necessarily satisfy the \ref{eq:C-D} in a pointwise sense. Nonetheless, this notion is known as the most natural one of
the generalized notions of solutions of \eqref{abstract_Cauchy}.
\section{Existence and uniqueness of mild solutions}\label{sec:existence_weak}\label{sec:exisntence&uniqueness}
The theory of nonlinear operators on Banach spaces is well-established. 
We refer the interested reader to \cite[Chapter 3]{lakshmikantham1981nonlinear}, \cite[Chapter 4]{barbu2010nonlinear} or \cite[Appendix A]{andreu2004parabolic} and references therein for well-organized summaries of all of the main results.

If the operator $\mathcal{L}=\Delta\Phi$ is accretive and such that for every $\epsilon>0$ there exists an $\epsilon$-approximate solution as in Definition \ref{epsilon-approximation}, then it is possible to infer the existence and uniqueness of mild solutions for the \ref{eq:C-D}, relying on some consequences of a result due to P. B{\'e}nilan, see \cite{benilan1972equations} and \cite[Theorem 3.3]{benilan1988evolution}, which is an extension of the famous Crandall-Liggett theorem, see \cite{crandall1971generation}. The main idea is the following: Given $u_0\in \ell^1\left(X,\mu\right)$ and an $\epsilon$-discretization $\mathcal{D}_\epsilon$ as in Definition~\ref{def:epsilon-discretization}, then solving system \eqref{implicit_Euler} means solving the equation
\begin{equation*}
(\operatorname{id} + \lambda_k \Delta \Phi)u_k= u_{k-1} + \lambda_kf_k
\end{equation*}
with
$$
u_k \in \ell^1\left(X,\mu\right), \qquad  \Phi u_k\in\dom(\Delta), \qquad \Delta\Phi u_k\in \ell^1(X,\mu)
$$
for every $k=1,\ldots,n$ where $\lambda_k>0$ and $f_k\in \ell^1\left(X,\mu\right)$.
This is doable, in particular, if
\begin{equation}\label{eq:semi_lin_elliptic}
(\operatorname{id} + \lambda \Delta \Phi) u  = g
\end{equation}
is solvable for any $g \in \ell^1(X,\mu)$, $\lambda>0$ and the solution $u \in \ell^1(X,\mu)$ is such that $\Phi u\in\dom(\Delta)$ and $\Delta\Phi u\in \ell^1(X,\mu)$. Therefore, if $\mathcal{L}$ is $m$-accretive, then we would get existence and uniqueness of mild solutions in one step.

For example, in Euclidean case, when $X=\Omega$ is a bounded domain in $\R^n$, then the accretivity property holds for $\mathcal{L}\coloneqq \Delta\Phi$ defined on
$$
\dom\left(\mathcal{L}\right)\coloneqq\left\{u \in L^1\left(\Omega\right) \mid \Phi(u) \in W^{1,1}_0(\Omega), \Delta\Phi u\in L^1(\Omega) \right\}
$$
where the Laplace operator $\Delta$ is understood in the sense of distributions. The difficult part is to prove the $m$-accretivity, i.e., to prove that for every $g \in L^1(\Omega)$ there exists $u\in \dom\left(\mathcal{L}\right)$ that is a solution for equation \eqref{eq:semi_lin_elliptic} for any $\lambda >0$. To circumvent the direct approach, it is common to switch to an equivalent formulation, namely, having defined $v\coloneqq\Phi u$ and $u=\Psi v=\Phi^{-1} v$, the question is whether the equation
$$
(\Psi + \lambda \Delta) v=g
$$
admits a solution $v$ in
$$
\left\{v \in W^{1,1}_0(\Omega) \mid  \Delta v\in L^1(\Omega) \right\}.
$$
The positive answer to this question in the Euclidean case was given by H. Br{\'e}zis and W. Strauss in \cite{brezis1973semi}. In particular, the trick of this approach is to relate the $m$-accretive property of the nonlinear operator $\mathcal{L}$ to suitable properties of the linear operator $\Delta$ which is easier to handle.

The main issue in the discrete setting is to prove existence of solutions under minimal assumptions on the underlying graph $G$. Indeed, while the accretivity of $\mathcal{L}$ can be established for any finite graph, see Corollary \ref{cor:accretivityL-finite}  in Appendix \ref{sec:appendix2}, the accretivity of $\mathcal{L}$ can be a tricky property to prove for more general graphs. Moreover,
the hypothesis required to make use of the result in \cite{brezis1973semi}, which are satisfied by the Euclidean Laplacian $\Delta$ over bounded domains $\Omega$, are, in general, not satisfied by the graph Laplacian. Consequently, we are forced to step back and to prove ``by hand'' the existence of $\epsilon$-approximate solutions $u_\epsilon(t)$ for every  $\epsilon>0$ with the property that $u_\epsilon(t)$ belongs to a dense subset of $\dom(\mathcal{L})$ where $\mathcal{L}$ is accretive.

The idea is to find a solution $u$ to $(\operatorname{id}+\lambda \Delta\Phi) u=g$ by building $u$ as a limit of a sequence $\{ \Psi v_n\}$ where the $v_n$ are solutions of $(\Psi + \lambda \Delta_n) v_n=g_n$ on suitable restrictions of the graph $G$ to finite subgraphs. In particular, we will see that this can be achieved by decomposing $G$ as an infinite ascending chain $\{G_{\textnormal{dir},n}\}_{n=1}^\infty$ of finite connected Dirichlet subgraphs.

After this introduction, we are now ready to prove our main results. In Theorem \ref{thm:main1} we will prove the accretivity of the operator $\mathcal{L}$ on a suitable dense subset of $\dom(\mathcal{L})$ and the surjectivity of $\operatorname{id}+\lambda\mathcal{L}$ on the non-negative/positive cones of $\ell^1(X,\mu)$ and, under three different additional hypotheses, on the entire space $\ell^1(X,\mu)$.  In Theorem \ref{thm:main2}  we will then establish the existence and uniqueness of mild solutions for the \ref{eq:C-D} problem.  As a concluding application, in Corollary \ref{cor:application}, we prescribe some hypotheses on the graph that guarantee that a mild solution is indeed a classic solution.

\subsection{Proofs of the main theorems}\label{ssec:proofs}
Let us recall that $\mathcal{L}$ is the operator
\begin{align*}
	&\mathcal{L} \colon \dom\left( \mathcal{L} \right)\subseteq \ell^{1}\left(X,\mu\right) \to \ell^{1}\left(X,\mu\right) , \\
	&\dom\left( \mathcal{L} \right)=\left\{ u \in  \ell^{1}\left(X,\mu\right) \mid \Phi u\in \dom\left(\Delta\right),\, \Delta\Phi u \in  \ell^{1}\left(X,\mu\right)  \right\}
\end{align*}
whose action is given by
$$
\mathcal{L} u= \Delta\Phi u,
$$
and that for a subset $\Omega \subseteq \dom\left( \mathcal{L} \right)$, we write $\mathcal{L}_{|\Omega}$
for the restriction of $\mathcal{L}$ to $\Omega$. The extra hypotheses listed in Theorem \ref{thm:main1} are
	\begin{enumerate}[label={\upshape(\bfseries H\arabic*)},wide = 0pt, leftmargin = 3em]
	\item\label{m-accretivity_A2}$G$ is locally finite;
	\item\label{m-accretivity_B2} $\inf_{x \in X}\mu(x)>0$;
	\item\label{m-accretivity_C2}  $\sup_{x \in X}\frac{\sum_{y \in X}w(x,y)}{\mu(x)}<\infty$ and $\Phi(\ell^1(X,\mu)) \subseteq \ell^1(X,\mu)$.
\end{enumerate}
\vspace{0.2cm}

\noindent\textbf{Proof of Theorem~\ref{thm:main1}.} 
We divide the proof into several steps. From \textbf{Steps 0} to \textbf{IV} we will assume that $G$ is connected. To help orient the reader, we first give a brief outline of the structure of the proof:  In \textbf{Step 0},  we will introduce a sequence of operators $\Lmin \colon C_c(X) \to \ell^1(X,\mu)$ and discuss that for every graph $G$ there exists a dense subset $\Omega \subseteq \dom(\mathcal{L})$ where $\mathcal{L}$ is accretive. Later, we will show that all solutions of the equation $(\operatorname{id} + \lambda \mathcal{L})u=g$ that we construct along the way belong to $\Omega$. In \textbf{Step I}, assuming that $G$ is finite, we will prove that $\operatorname{id}+\lambda\mathcal{L}$ is surjective and preserves nonnegativity/nonpositivity and also give an upper bound of the norm of the solutions with respect to $g$. This step plays a crucial role in proving the surjectivity of $\operatorname{id}+\lambda\mathcal{L}$ in the infinite case where we will approximate the graph $G$ by an ascending chain of finite Dirichlet subgraphs. In \textbf{Step II}, given $\lambda>0$ and $g\in \ell^1(X,\mu)$, we will show that there exists a sequence of compactly supported functions $u_n$ and $u\in \ell^1(X,\mu)$ such that $\lim_{n\to \infty}\|u_n-u\|=0$ and $\lim_{n\to \infty}\|(\boldsymbol{\mathfrak{i}}_{n,\infty}\operatorname{id}_n\boldsymbol{\pi}_n+\lambda\Lmin)u_n- g\|=0$. Using this construction,  in \textbf{Step III} we will show that $u$ solves $(\operatorname{id} + \lambda \mathcal{L})u=g$ for every $g\in\ell^{1,\pm}(X,\mu)$ and that $u\in \Omega$. In \textbf{Step IV-H1,-H2,-H3} we will prove that $u$ solves $(\operatorname{id} + \lambda \mathcal{L})u=g$ for every $g\in\ell^{1}(X,\mu)$ and that $u\in \Omega$ under any one of the three different assumptions \ref{m-accretivity_A2}, \ref{m-accretivity_B2} and \ref{m-accretivity_C2}. Finally, in \textbf{Step V} we remove the assumption of connectedness that we used while proving \textbf{Steps 0} to \textbf{IV}. 

\vspace{0.5cm}
\noindent\textbf{Step 0} (When $G$ is connected, there exists a dense subset $\Omega$ of $\dom(\mathcal{L})$ where $\mathcal{L}_{|\Omega}$ is accretive)\textbf{:} This is exactly the content of  Lemma \ref{lem:L_Omega_accretive} in Appendix \ref{sec:appendix2}. 
To help orient the reader, we recall here the notations involved and the definition of $\Omega$. If $G$ is finite, then $\Omega=\dom(\mathcal{L})=C(X)$. If $G$ is infinite, we take an exhaustion
$\{X_n\}_{n=1}^\infty$ of $X$, i.e., a sequence of subsets $X_n$ of $X$ such that $X_n \subseteq X_{n+1}$
and $X = \cup_{n=1}^\infty X_n$, where we assume that each $X_n$ is additionally finite, along with the canonical embedding $\boldsymbol{\mathfrak{i}}_{n,\infty}$ and the canonical projection $\boldsymbol{\pi}_n$ for each $X_n$:
\begin{align*}
	&\boldsymbol{\mathfrak{i}}_{n,\infty} \colon C(X_n)\to C(X) \quad &\boldsymbol{\mathfrak{i}}_{n,\infty}u(x) \coloneqq \begin{cases}
		u(x) & \mbox{if } x \in X_n,\\
		0   & \mbox{if }  x \in X\setminus X_n;
	\end{cases}\\
	&\boldsymbol{\pi}_n \colon C(X) \to C(X_n) \quad &\boldsymbol{\pi}_nu(x) \coloneqq u(x) \mbox{ for every } x \in X_n.
\end{align*}
At this point, the exhaustion $\{X_n\}_{n=1}^\infty$ can be arbitrary but should consist of finite sets. 
We then define the operators $\Lmin$ as in Definition \ref{def:Lmin}, namely,
$$
\Lmin \colon \dom\left( \Lmin \right)\subseteq \ell^{1}\left(X,\mu\right) \to \ell^{1}\left(X,\mu\right)
$$
with
\begin{align*}
	\dom\left( \Lmin \right)\coloneqq C_c(X), \quad \Lmin u\coloneqq \boldsymbol{\mathfrak{i}}_{n,\infty}\Deltadirn\Phi\boldsymbol{\pi}_{n} u
\end{align*}
where $\Deltadirn$ is the graph Laplacian associated to the Dirichlet subgraph $G_{\textnormal{dir},n} \subseteq G$ on the node set $X_n$. Then, the set $\Omega$ is defined as
	\begin{equation*}\label{eq:Omega}
	\Omega\coloneqq \{u \in \dom(\mathcal{L}) \mid \exists\, \{u_n\}_n \mbox{ s.t. } \operatorname{supp}u_n\subseteq X_n,\; \lim_{n\to \infty}\|u_n -u \|=0,\; \lim_{n\to \infty}\|\Lmin u_n -\mathcal{L}u \|=0 \}
\end{equation*}
where $\operatorname{supp}u_n$ denotes the support of the function $u_n$.
We have $\overline{\Omega}=\overline{\dom(\mathcal{L})}=\ell^1(X,\mu)$ by Lemma \ref{lem:FC} and that $\mathcal{L}_{|\Omega}$  is accretive by Lemma \ref{lem:L_Omega_accretive}.

\vspace{0.5cm}
\noindent\textbf{Step I} (When $G$ is finite and connected, $\operatorname{id} +\lambda\mathcal{L}$ is bijective and preserves nonnegativity/nonpositivity)\textbf{:} Assume that $G$ is finite and connected, i.e., $|X|=n$. In this case,
$$
\dom(\Delta)=C(X)\simeq \R^n, \quad \ell^1\left(X,\mu\right)= \left(C(X), \|\cdot\| \right) \quad \mbox{ with } \|u\|= \sum_{x\in X}|u(x)|\mu(x).
$$

Fix now $\lambda >0$. Writing $\psi \coloneqq \phi^{-1}$ and $v\coloneqq \Phi u$, we can rewrite equation
\begin{equation}\label{eq:step-I-1.0}
(\operatorname{id} + \lambda \Delta\Phi) u=g
\end{equation}
in the equivalent form
\begin{equation}\label{eq:step-I-1}
(\Psi + \lambda \Delta) v  = g.
\end{equation}
Clearly, $\psi$ is strictly monotone increasing, $\psi(\R)=\R$ and $\psi(0)=0$.

Let us enumerate the nodes of $X$, that is, we write $X= \{ x_1, x_2, \ldots, x_n \}$. Owing to the isomorphism between $C(X)$ and $\R^n$, we identify $n$-dimensional vectors and real-valued functions on $X$ in the standard way, that is, given $v \in C(X)$ we associate to $v$ the vector $\boldsymbol{v}=(v_1,\ldots,v_n)\coloneqq(v(x_1),\ldots, v(x_n))$ and vice-versa. Define  $M \colon \R^n \to \R^n$ by
$$
M\boldsymbol{v}\coloneqq (\Psi + \lambda \Delta)\boldsymbol{v}.
$$
Let us observe that:
\begin{enumerate}[i)]
	\item $(\Psi\boldsymbol{v})_i=\psi(v_i)$ for every $i=1,\ldots,n$ where $\psi \colon \R \to \R$ is surjective and strictly monotone increasing;
	\item For every $\lambda>0$, $\lambda\Delta$ is a diagonally dominant matrix (e.g., \cite[Definition 6.1.9]{horn2012matrix}), i.e.,
	$$
	\left(\lambda\Delta\right)_{i,i}= \frac{\lambda}{\mu(x_i)}\left(\kappa(x_i) + \sum_{\substack{j=1\\j\neq i}}^n w(x_i,x_j)\right) \geq \frac{\lambda}{\mu(x_i)}\sum_{\substack{j=1\\j\neq i}}^n w(x_i,x_j)= \sum_{\substack{j=1\\j\neq i}}^n \left|\left(\lambda\Delta\right)_{i,j}\right|, \quad \forall\; i=1,\ldots,n.
	$$
\end{enumerate}
Therefore, by \cite[Theorem 1]{willson1968solutions}, for every $\boldsymbol{g} \in \R^n$ there exists a unique solution $\boldsymbol{v}$ to the equation
$$
M\boldsymbol{v} = \boldsymbol{g}.
$$

We now show that the norm of the solution is bounded above by the norm of $g$. Let $v$ be the solution of \eqref{eq:step-I-1} with right-hand side $g$. Since $\psi$ is strictly monotone increasing and $\psi(0)=0$, we get
$$
\operatorname{sgn}\left(v(x)\right) = \operatorname{sgn}\left(\Psi v(x)\right).
$$
By Proposition \ref{prop:nonnegativity}, and recalling that $v=\Phi u$, i.e., $u = \Psi v$, it follows that
$$
\sum_{x\in X} \Delta v (x) \operatorname{sgn}\left(\Psi v(x)\right)\mu(x)= \sum_{\substack{x\in X\colon\\u(x)\neq 0}} \Delta \Phi u(x)  \operatorname{sgn}\left(u(x)\right)\mu(x)\geq0.
$$
Therefore, we conclude
\begin{align}\label{eq:step-I-3}
	\|u\| =\|  \Psi v  \| &= \sum_{x\in X} \left|\Psi v(x) \right|\mu(x)\\\nonumber
	&=  \sum_{x\in X} \Psi v(x)\operatorname{sgn}\left(\Psi v(x)\right)\mu(x)\\\nonumber
	&=  \sum_{x\in X} g(x)\operatorname{sgn}\left(\Psi v(x)\right)\mu(x)  -  \lambda\sum_{x\in X} \Delta v(x) \operatorname{sgn}\left(\Psi v(x)\right)\mu(x)\\\nonumber
	&\leq  \sum_{x\in X} g(x)\operatorname{sgn}\left(\Psi v(x)\right)\mu(x)\\\nonumber
	&\leq \| g\|.
\end{align}

By Case 2) of Theorem~\ref{thm:min_principle}, if $g\geq 0$, then $v\geq 0$, and if $g\leq 0$, then $v\leq 0$. Consequently, $u=\Psi v$ has the same sign as $g$. Therefore, if $G$ is finite we can conclude that for every $g \in \ell^{1,\pm}(X,\mu)$, the unique solution $u$ of \eqref{eq:step-I-1.0} belongs to $\ell^{1,\pm}(X,\mu)$ and satisfies $\|u \| \leq \|g\|$.

\vspace{0.5cm}

\noindent\textbf{Step II} (Constructing a solution when $G$ is infinite and connected)\textbf{:}   We want to show that if $G$ is infinite and connected, then for every fixed $\lambda>0$ and $g \in \ell^1(X,\mu)$ there exists $u\in \ell^1(X,\mu)$ and a sequence $\{u_n\}_n$   such that
\begin{subequations}
	\begin{equation}\label{eq:l1convergence2}
	\operatorname{supp}u_n\subseteq X_n;
	\end{equation}
	\begin{equation}\label{eq:l1convergence}
		\lim_{n\to \infty}\|u_n -u\|=0;
	\end{equation}
	\begin{equation}\label{eq:l1convergence0}
		\lim_{n\to \infty}\|\left(\boldsymbol{\mathfrak{i}}_{n,\infty}\operatorname{id}_n\boldsymbol{\pi}_{n} + \lambda \Lmin\right) u_n- g\|=0
	\end{equation}
\end{subequations}
where $\operatorname{id}_n$ is the identity operator on $C(X_n)$. We divide this step into two sub-steps
consisting of the cases when $g$ is nonnegative (or nonpositive) and then general $g$.

\vspace{0.2cm}
\noindent\textbf{Step II-1} ($g\in \ell^{1,\pm}(X,\mu)$)\textbf{:} Assume that $g\in \ell^{1,+}(X,\mu)$. By Lemma \ref{lem:chain1}, we can choose the exhaustion $\{X_n\}_{n=1}^{\infty}$ with the following additional properties: 
\begin{equation}\label{eq:property_sets1}
	X_{n}\subset X_{n+1}, \qquad X=\bigcup_{n=1}^\infty X_n
\end{equation}
and such that the set
\begin{equation}\label{eq:property_sets2}
	\{ x \in X_n \mid x\sim y \mbox{ for some } y \in X_{n+1}\setminus X_n \}
\end{equation}
is not empty for all $n$. For each $n$, we define the subgraph 
\begin{equation}\label{eq:def_Dir_subgraph}
G_{\textnormal{dir},n} = (X_n, w_n, \kappa_{\textnormal{dir},n}, \mu_{n})\subset G
\end{equation} 
as a Dirichlet subgraph of $G$,  see Definition \ref{def:dir_subgraph}. That is,
\begin{itemize}
	\item $w_n \equiv w_{|X_{n}\times X_{n}}$;\
	\item $\mu_n \equiv  \mu_{|X_{n}}$;
	\item for every $x \in X_n$, $b_{\textnormal{dir},n}(x)= \sum_{y \in \mathbullet{\partial}X_n}w(x,y)$;
	\item for every $x \in X_n$, $\kappa_{\textnormal{dir},n}(x)  = \kappa_{|X_n}(x) + b_{\textnormal{dir},n}(x)$.
\end{itemize}
If we define
\begin{align*}
&\mathring{\partial}X_{n,n+1}\coloneqq\{ x \in X_{n} \mid \exists y \in X_{n+1}\setminus X_n \mbox{ such that } x \sim y  \}, \\
&\mathbullet{\partial}X_{n,n+1}\coloneqq\{ y \in X_{n+1} \setminus X_n \mid \exists x \in X_n \mbox{ such that } x \sim y  \}
\end{align*}
which are not empty by construction and
\begin{equation*}\label{eq:thm_existence_weak_1}
b'_{\textnormal{dir},n}(x)= \sum_{y \in \mathbullet{\partial}X_{n,n+1}}w(x,y)
\end{equation*}
then, for every $x\in X_n$, it holds that
\begin{align*}
\kappa_{\textnormal{dir},n}(x)  &= \kappa_{|X_n}(x) + b_{\textnormal{dir},n}(x)\\
&=\kappa_{|X_n}(x) + \sum_{y \in \mathbullet{\partial}X_n}w(x,y)\\
&= \kappa_{|X_{n+1}}(x) +\sum_{y \in X\setminus X_n}w(x,y)\\
&= \kappa_{|X_{n+1}}(x) + \sum_{y \in X\setminus X_{n+1}}w(x,y) + \sum_{y \in X_{n+1}\setminus X_{n}}w(x,y)\\
&= \kappa_{|X_{n+1}}(x) + \sum_{y \in \mathbullet{\partial}X_{n+1}}w(x,y) + \sum_{y \in \mathbullet{\partial}X_{n,n+1}}w(x,y)\\
&= \kappa_{|X_{n+1}}(x) + b_{\textnormal{dir},n+1}(x) + b'_{\textnormal{dir},n}(x)\\
&= \kappa_{\textnormal{dir},n+1}(x) +  b'_{\textnormal{dir},n}(x).
\end{align*}
So, the collection $\{G_{\textnormal{dir},n}\}_{n\in \N}$ is a sequence of connected finite Dirichlet subgraphs such that each subgraph $G_{\textnormal{dir},n}$ is a Dirichlet subgraph of $G_{\textnormal{dir},n+1}$, that is, $$G_{\textnormal{dir},1}\subset \ldots \subset G_{\textnormal{dir},n}\subset G_{\textnormal{dir},n+1} \subset \ldots \subset G.$$
Denoting by 
\begin{equation*}\label{eq:embedding-projection2}
\boldsymbol{\mathfrak{i}}_{n} \colon C(X_n) \hookrightarrow C(X_{n+1}), \quad \boldsymbol{\mathfrak{i}}_{n,\infty} \colon C(X_{n}) \hookrightarrow C(X), \quad  \boldsymbol{\pi}_n \colon C(X) \to C(X_n)
\end{equation*}
the canonical embeddings and projections, respectively, define
$$
g_n \coloneqq  \boldsymbol{\pi}_{n}g \quad \mbox{where } g \in \ell^{1,+}(X,\mu).
$$

From \textbf{Step I}, for every $n\in \N$ there exist $\hat{v}_n \in C(X_n)$ such that $\hat{v}_{n} \geq 0$  and
\begin{equation}\label{def:hatv_n}
(\Psi + \lambda \Delta_{\textnormal{dir}, n})\hat{v}_{n}  = g_n.
\end{equation}
Setting
$$
q_n(x)\coloneqq ( \Psi +\lambda\Delta_{\textnormal{dir},n+1}) \boldsymbol{\mathfrak{i}}_{n}\hat{v}_{n}(x) \in C(X_{n+1})
$$
by Lemma \ref{lem:A1}, and the fact that every $G_{\textnormal{dir},n}$ is a Dirichlet subgraph of $G_{\textnormal{dir}, n+1}$, we have
$$
(\Psi + \lambda \Delta_{\textnormal{dir}, n+1})\boldsymbol{\mathfrak{i}}_{n}\hat{v}_{n}(x) =(\Psi + \lambda \Delta_{\textnormal{dir}, n})\hat{v}_{n}(x)  \quad \forall x \in X_n
$$
and
\begin{equation*}
q_n(x)=\begin{cases}
g_n(x) & \mbox{if } x \in X_n,\\
-\frac{\lambda}{\mu(x)}\sum_{y \in X_n}w(x,y)\hat{v}_{n}(y)& \mbox{if } x \in \mathbullet{\partial}X_{n,n+1},\\
0 & \mbox{if } x \in X_{n+1}\setminus (X_n \cup \mathbullet{\partial}X_{n,n+1}).
\end{cases}
\end{equation*}
Since $0\leq \hat{v}_n$ and  $0\leq g_{n+1}$, it follows that $q_n(x) \leq g_{n+1}(x)$ for every $x \in X_{n+1}\setminus X_n$. In particular, from the fact that $g_{n+1}(x)=g_n(x)$ for every $x \in X_n$, we get $q_n\leq g_{n+1}$. By Corollary \ref{cor:min}, we have $\boldsymbol{\mathfrak{i}}_{n}\hat{v}_{n}\leq \hat{v}_{n+1}$, and by the fact that $\psi$ is monotone increasing and $\psi(0)=0$, we have
\begin{equation*}
\Psi\boldsymbol{\mathfrak{i}}_{n}\hat{v}_{n}\leq \Psi \hat{v}_{n+1}\quad \mbox{and}\quad \Psi\boldsymbol{\mathfrak{i}}_{n}\hat{v}_{n}(x)=\boldsymbol{\mathfrak{i}}_{n}\Psi \hat{v}_{n}(x)= \begin{cases}
	\psi (\hat{v}_n(x)) &\mbox{if } x \in X_n,\\
	0 &\mbox{if } x \in X_{n+1}\setminus X_n.
\end{cases}
\end{equation*}

Therefore, $\boldsymbol{\mathfrak{i}}_{n}\Psi \hat{v}_{n}\leq \Psi \hat{v}_{n+1}$. In particular, we get
\begin{align}
&\boldsymbol{\mathfrak{i}}_{n,\infty}\hat{v}_{n}(x)=\boldsymbol{\mathfrak{i}}_{n+1,\infty}\boldsymbol{\mathfrak{i}}_{n}\hat{v}_{n}(x) \leq \boldsymbol{\mathfrak{i}}_{n+1,\infty}\hat{v}_{n+1}(x),\label{eq:monotonicity1}\\
&\boldsymbol{\mathfrak{i}}_{n,\infty}\Psi \hat{v}_{n}(x)=\boldsymbol{\mathfrak{i}}_{n+1,\infty}\boldsymbol{\mathfrak{i}}_{n}\Psi \hat{v}_{n}(x) \leq \boldsymbol{\mathfrak{i}}_{n+1,\infty}\Psi \hat{v}_{n+1}(x)\label{eq:monotonicity2}
\end{align}
for every $x \in X$. 
Moreover, writing $\hat{u}_{n}(x)=\psi\left(\hat{v}_{n}(x)\right)\geq 0$  for every $x \in X_n$ and indicating by $\|\cdot \|_n$ the restriction of $\|\cdot\|$ to $C(X_n)$ from  \eqref{eq:step-I-3} 
we have
\begin{align*}
0\leq  \hat{u}_{n}(x) \mu(x)&\leq \sum_{x \in X_n}  \hat{u}_{n}(x) \mu(x)\\
&=  \|\hat{u}_n\|_n \leq \|g_n\|_n \leq \|g\|,
\end{align*}
that is, for every fixed $x$, $\hat{u}_{n}(x)$ is bounded uniformly in $n$. In particular, writing
\begin{equation*}\label{re-labelling}
 v_n  \coloneqq \boldsymbol{\mathfrak{i}}_{n,\infty}\hat{v}_{n} \in C_c(X)\qquad \mbox{and} \qquad u_n \coloneqq \boldsymbol{\mathfrak{i}}_{n,\infty}\hat{u}_n \in C_c(X)
\end{equation*}
 it follows that 
\begin{equation}\label{eq:boundedness}
\Psi v_{n}(x) =	u_n(x) \in \left[0, \frac{\|g\|}{\mu(x)}\right].
\end{equation}

Consequently, by \eqref{eq:monotonicity2} and \eqref{eq:boundedness}, for every fixed $x \in X$ we have a sequence
\begin{equation}\label{eq:sequence}
\left\{u_{n}(x)  \right\}_n \coloneqq \left\{\Psi v_{n}(x)  \right\}_n
\end{equation}
that is monotonic and bounded. We can then define $u,v \in C(X)$ such that
\begin{align}
&u(x) \coloneqq \lim_{n\to \infty} u_{n}(x) \quad \mbox{for } x\in X,\label{eq:limit1} \\
&v\coloneqq \Phi u.\label{eq:limit2}
\end{align}

Notice that, by construction, $x \in X_n$ eventually so that $u_{n}(x) = \hat{u}_n(x)$ eventually. In particular, by the continuity of $\phi$ and the fact that $\psi=\phi^{-1}$
$$
v(x) = \lim_{n\to \infty} v_{n}(x)
$$
and the limit is monotone. Moreover, $u_n$ satisfies \eqref{eq:l1convergence2}, i.e.,
$\operatorname{supp}u_n\subseteq X_n$ for every $n$ and 
\begin{equation*}
	(\operatorname{id}_n + \lambda\Deltadirn\Phi)\boldsymbol{\pi}_nu_n=  g_n.
\end{equation*}
Therefore, 
\begin{align*}
	(\boldsymbol{\mathfrak{i}}_{n,\infty}\operatorname{id}_n\boldsymbol{\pi}_{n}   + \lambda \Lmin )u_n &=  (\boldsymbol{\mathfrak{i}}_{n,\infty}\operatorname{id}_n\boldsymbol{\pi}_{n}   + \lambda \boldsymbol{\mathfrak{i}}_{n,\infty}\Deltadirn\Phi\boldsymbol{\pi}_{n} ) u_n\\
	&= \boldsymbol{\mathfrak{i}}_{n,\infty}(\operatorname{id}_n  + \lambda \Deltadirn\Phi )\boldsymbol{\pi}_{n}u_n\\
	&= \boldsymbol{\mathfrak{i}}_{n,\infty}  g_n.
\end{align*}
In particular, since $g \in \ell^1(X,\mu)$,
$$
\lim_{n\to \infty}\|\left(\boldsymbol{\mathfrak{i}}_{n,\infty}\operatorname{id}_n\boldsymbol{\pi}_{n} + \Lmin\right) u_n- g\|=0,
$$
which is exactly \eqref{eq:l1convergence0}.

Let $\|\cdot \|_n$ be the restriction of $\|\cdot\|$ to $C(X_n)$. Since every $X_n$ is finite, from \eqref{eq:step-I-3} in \textbf{Step I} we obtain
$$
\| \Psi \hat{v}_{n}\|_n\leq \|  g_n\|_n
$$
and, consequently, by Fatou's lemma
\begin{align}
	\|u \| &\leq \liminf_{n\to \infty}\| u_n\|
	=\liminf_{n\to \infty} \| \Psi v_{n} \|
	= \liminf_{n\to \infty} \| \Psi \hat{v}_{n} \|_n
	\leq \liminf_{n\to \infty} \| g_n \|_n = \| g \|.\label{eq:contractivity}
\end{align}
In particular, $u =\Psi v$ is in $\ell^{1,+}(X,\mu)$. Finally, by dominated convergence, we get \eqref{eq:l1convergence},
i.e., $\lim_{n\to \infty} \| u_n - u \| =0$.

\vspace{0.2cm}
\noindent\textbf{Step II-2} ($g\in \ell^1(X,\mu)$)\textbf{:} Using the same notation as in \textbf{Step II-1}, we define for $g \in \ell^1(X,\mu)$
\begin{equation*}
	g_n \coloneqq \boldsymbol{\pi}_{n}g,\quad
	g^+_n\coloneqq \max\{0, \; g_n \},\quad
	g^-_n\coloneqq \min\{0, \; g_n \}.
\end{equation*} 
From \textbf{Step I} there exist $\hat{v}_{n}, \hat{v}^+_{n}, \hat{v}^-_{n}\in C(X_n)$ that satisfy
\begin{equation}\label{eq:Laccretive}
	\begin{cases}
		(\Psi + \lambda\Delta_{\textnormal{dir},n})\hat{v}_{n} = g_n,\\
		(\Psi +\lambda\Delta_{\textnormal{dir},n})\hat{v}^+_{n}  = g^+_n,\\
		(\Psi +\lambda\Delta_{\textnormal{dir},n})\hat{v}^-_{n}  = g^-_n.
	\end{cases}
\end{equation}
Define 
\begin{equation}\label{eq:def_u_n}
	\hat{u}_n\coloneqq \Psi \hat{v}_n, \quad u_n \coloneqq \mathbf{\mathfrak{i}}_{n,\infty}\hat{u}_n.
\end{equation}
Clearly, $u_n$ satisfies \eqref{eq:l1convergence2} and, by construction,
\begin{equation*}
	(\operatorname{id}_n + \lambda\Deltadirn\Phi)\boldsymbol{\pi}_nu_n=  g_n.
\end{equation*}
Therefore, 
\begin{align*}
	(\boldsymbol{\mathfrak{i}}_{n,\infty}\operatorname{id}_n\boldsymbol{\pi}_{n}   + \lambda \Lmin )u_n &=  (\boldsymbol{\mathfrak{i}}_{n,\infty}\operatorname{id}_n\boldsymbol{\pi}_{n}   + \lambda \boldsymbol{\mathfrak{i}}_{n,\infty}\Deltadirn\Phi\boldsymbol{\pi}_{n} ) u_n\\
	&= \boldsymbol{\mathfrak{i}}_{n,\infty}(\operatorname{id}_n  + \lambda \Deltadirn\Phi )\boldsymbol{\pi}_{n}u_n\\
	&= \boldsymbol{\mathfrak{i}}_{n,\infty}  g_n,
\end{align*}
that is,
$$
\lim_{n\to \infty}\|\left(\boldsymbol{\mathfrak{i}}_{n,\infty}\operatorname{id}_n\boldsymbol{\pi}_{n} + \lambda \Lmin\right) u_n- g\|=0,
$$
which is exactly \eqref{eq:l1convergence0}. 

Define now
$$\hat{u}^+_n\coloneqq \Psi \hat{v}^+_n,\quad u_n^+\coloneqq\mathbf{\mathfrak{i}}_{n,\infty}\hat{u}^+_n, \quad u^+\coloneqq  \lim_{n\to \infty} u^+_n.$$
In particular, $u^+ \in \ell^1(X,\mu)$ is the monotone limit solution of \eqref{eq:step-I-1.0} obtained in \eqref{eq:sequence} and \eqref{eq:limit1} of \textbf{Step II-1}. In the same way, define $\hat{u}^-_n$, $u^-_n$ and $u^-$. Finally, define $$v^+\coloneqq \Phi u^+, \quad v^-\coloneqq\Phi u^-$$ as in \eqref{eq:limit2} of \textbf{Step II-1}. Let us observe that, by the definitions \eqref{eq:Laccretive}-\eqref{eq:def_u_n}  and Corollary~\ref{cor:min}, and monotone limits, it holds that
\begin{align*}
	& u_n(x) = \hat{u}_n(x)\leq  \hat{u}^+_n(x)\leq u^+(x) \quad \mbox{if } x\in X_n,\\
	& u_n(x) = 0\leq u^+(x) \quad \mbox{if } x\notin X_n,
\end{align*}
and
\begin{align*}
	&  u^-(x)\leq \hat{u}^-_n(x)\leq \hat{u}_n(x)=u_n(x)    \quad \mbox{if } x\in X_n,\\
	& u^-(x)\leq 0= u_n(x) \quad \mbox{if } x\notin X_n.
\end{align*}
In particular, 
$$
u^-(x)\leq  u_n(x)\leq  u^+(x), \quad \forall\, x\in X,\; n\in \N,
$$
that is, for every fixed $x\in X$ the sequence $u_n(x)$ is uniformly bounded in $n$ with
\begin{equation}\label{eq:uniform_bound}
	|u_n(x)| \leq c_x\coloneqq \max\left\{|u^-(x)|, |u^+(x)|\right\}<\infty \quad \forall\, n\in \N.
\end{equation}

Therefore, by passing to a subsequence using a diagonal sequence argument, the limit functions
\begin{align*}
	&u(x)\coloneqq \lim_{n\to \infty} u_{n}(x) = \lim_{n\to \infty} \Psi v_{n}(x),\\ 
	& v \coloneqq \Phi u
\end{align*}
exist and are well-defined on $X$.  By the same arguments in \eqref{eq:contractivity}, it follows that 
\begin{equation}\label{eq:contractivity2}
	\|u\| \leq \|g\|  \quad \mbox{and} \quad	u =\Psi v \in \ell^1(X,\mu).
\end{equation}
Moreover, from the previous \textbf{Step II-1} we know that $u^+,u^- \in \ell^1(X,\mu)$ and then from \eqref{eq:uniform_bound} it follows that $|u_n|$ is bounded above by an integrable function. By dominated convergence, we get \eqref{eq:l1convergence}, i.e., $\lim_{n\to \infty} \| u_n - u \| =0$.

\vspace{0.5cm}
\noindent\textbf{Step III} (When $G$ is infinite and connected, $\operatorname{id}+ \lambda \mathcal{L}_{|\Omega}$ maps bijectively onto $\ell^{1,\pm}(X,\mu)$)\textbf{:} We want to show that the function $u \in \ell^{1\pm}(X,\mu)$  that we constructed in \textbf{Step II-1} as limit of a sequence of finitely supported functions $\{u_n\}_n$  is a solution of \eqref{eq:step-I-1.0} which belongs to $\Omega$. In order to do so, it remains to show that: 
\begin{subequations}
\begin{equation}\label{eq:1}
	u\in \dom(\mathcal{L});
\end{equation}
\begin{equation}\label{eq:3}
(\operatorname{id} + \lambda\mathcal{L})u=g;
\end{equation}
\begin{equation}\label{eq:2}
	\lim_{n\to \infty}\|\Lmin u_n - \mathcal{L}u  \|=0. 
\end{equation}
\end{subequations} 

Let us now highlight that $v_{n}=\Phi u_n \in C_c(X) \subseteq \dom\left(\Delta\right)$ for every $n$, that is, $\Delta v_{n}$ is well-defined. Since, for every $x \in X$,
$$
\sum_{y\in X} \lim_{n\to \infty}w(x,y) v_{n}(y)= \lim_{n\to \infty}\sum_{y\in X}w(x,y) v_{n}(y)
$$
by \eqref{eq:monotonicity1} and monotone convergence, and recalling that every $G_{\textnormal{dir},n}$ is a Dirichlet subgraph of $G$ for every $n\in \N$, by Lemma \ref{lem:A1} we get
\begin{align}\label{eq:thm_existence_5}
	\nonumber \Psi v(x)&+\frac{\lambda}{\mu(x)}\left[ \deg(x)v(x) - \sum_{y\in X}w(x,y)v(y) \right]  \\\nonumber
	&= \lim_{n\to \infty}\left( \Psi v_{n}(x) + \frac{\lambda}{\mu(x)}\left[\deg(x)v_{n}(x) - \sum_{y\in X}w(x,y) v_{n}(y)  \right] \right) \\\nonumber
	&= \lim_{n\to \infty}  (\Psi + \lambda \Delta)v_{n}  (x)\\\nonumber
	&=  \lim_{n\to \infty} \left( \Psi + \lambda\Delta_{\textnormal{dir},n}  \right)\hat{v}_{n} (x)\\
	&= \lim_{n\to \infty}  g_n (x)= g(x).
\end{align}
Notice that along the way we used the fact that $\psi$ is continuous. 
Moreover, by Remark \ref{rem:1}, we observe that $v \in \dom\left( \Delta \right)$, that is, $\Phi u \in \dom\left( \Delta \right)$ and 
$$
(\Psi + \lambda\Delta)v(x) = g(x),
$$
namely, $v$ is a nonnegative solution of \eqref{eq:step-I-1} and thus $u$ is a nonnegative solution of \eqref{eq:3}.

 By the fact that $\lambda \Delta\Phi u = g - u$ and $g, u \in \ell^1(X,\mu)$, we obtain $\Delta\Phi u \in \ell^{1}(X,\mu)$. We can then conclude that $u \in \dom\left(\mathcal{L}\right)$, i.e., \eqref{eq:1}. Let us prove \eqref{eq:2}: By the fact that $\lambda\mathcal{L}u= g-u$
 we obtain
 \begin{align}\label{eq:thm_existence_Alb}
 	\|\Lmin u_n - \mathcal{L}u\| & \leq \frac{1}{\lambda} \left( \|(\boldsymbol{\mathfrak{i}}_{n,\infty}\operatorname{id}_n\boldsymbol{\pi}_n + \lambda\Lmin)u_n - g\| +  \| \boldsymbol{\mathfrak{i}}_{n,\infty}\operatorname{id}_n\boldsymbol{\pi}_n u_n - u \| \right) \nonumber \\
 	&= \frac{1}{\lambda} \left( \|(\boldsymbol{\mathfrak{i}}_{n,\infty}\operatorname{id}_n\boldsymbol{\pi}_n + \lambda\Lmin)u_n - g\| +  \|  u_n - u \| \right)
 \end{align}
and we conclude \eqref{eq:2} by using \eqref{eq:l1convergence} and \eqref{eq:l1convergence0}.

  In particular, we have shown that for every $\lambda >0$ and $g \in \ell^{1,+}(X,\mu)$ there exists a unique $u \in \Omega\cap \ell^{1,+}(X,\mu)$ such that $(\operatorname{id}+\lambda\mathcal{L})u=g$, and $\|u\|\leq \|g\|$.  If $g\in \ell^{1,-}(X,\mu)$, then the arguments of the proof are completely symmetric and the (nonpositive) solution $u$ can be built as monotone decreasing limit.

\vspace{0.5cm}
\noindent\textbf{Step IV} (When $G$ is infinite, connected and satisfies assumptions \ref{m-accretivity_A2}, \ref{m-accretivity_B2} or \ref{m-accretivity_C2}, then $\operatorname{id} +\lambda\mathcal{L}_{|\Omega}$ maps bijectively onto $\ell^{1}(X,\mu)$)\textbf{:} The statement follows immediately by the same arguments in \textbf{Step III} if we can show that $u\in \dom(\mathcal{L})$ and 
$$
(\operatorname{id}+\lambda\mathcal{L})u=g.
$$
We divide this step into three sub-steps. 

\vspace{0.2cm}
\noindent\textbf{Step IV-H1} ($G$ is locally finite)\textbf{:} Let us highlight that $\dom(\Delta)=C(X)$ because of the local finiteness of $G$ so that $v_{n}, v \in  \dom\left(\Delta\right)$, that is, $\Delta v_{n}$ and $\Delta v=\Delta\Phi u$ are well-defined. Furthermore, by the local finiteness of $G$, for every fixed $x$ there exists a finite number 
of nodes $y\in X$ such that $w(x,y)\neq 0$. 
By Lemma~\ref{lem:chain2}, we can assume moreover that the sequence $\{X_n\}_{n=1}^\infty$ in addition to \eqref{eq:property_sets1} and \eqref{eq:property_sets2} also satisfies
$$
\mathring{X}_{n}\subset \mathring{X}_{n+1}\quad \mbox{and}\quad\bigcup_{n=1}^\infty \mathring{X}_n =X.
$$
Let $N=N(x)$ be such that $x \in \mathring{X}_n$ for every $n \geq N$. 
As a consequence, for every fixed $x \in X$, the series are in fact finite sums and, passing to the limit, we get
\begin{align*}
	\lim_{n\to \infty} \sum_{y \in X} w(x,y) |v_{n}(y)|& 
	=\lim_{n\to \infty}\sum_{y \in X_N} w(x,y) |v_{n}(y)|\\
	&=\sum_{y \in X} \lim_{n\to \infty} w(x,y) |v_{n}(y)|.
\end{align*}
Therefore, by the above considerations, we have
\begin{align*}
	(\Psi +\lambda\Delta) v(x) & =\lim_{n\to \infty} \left( \Psi + \lambda\Delta  \right)v_{n} (x)\\
	&=  \lim_{n\to \infty} \left( \Psi + \lambda\Delta_{\textnormal{dir},n}    \right)\hat{v}_{n} (x)\\
	&= \lim_{n\to \infty}  g_n (x)= g(x)
\end{align*}
where the second equality follows from Lemma \ref{lem:A1} (and the fact that every $G_{\textnormal{dir},n}$ is a Dirichlet subgraph of $G$ for every $n\in \N$).  

Since $\lambda \Delta\Phi u = g - u$, we get $\Delta\Phi u \in \ell^{1}(X,\mu)$, i.e., $u\in \dom(\mathcal{L})$. By the same arguments as in \textbf{Step III}, see \eqref{eq:thm_existence_Alb},
we can check that $u \in \Omega$. By the accretivity of $\mathcal{L}_{|\Omega}$ we obtain the uniqueness of $u$. In particular, we have proven that for every $\lambda >0$ and $g \in \ell^{1}(X,\mu)$ there exists a unique $u \in \Omega$ such that $(\operatorname{id}+\lambda\mathcal{L})u=g$, and $\|u\|\leq \|g\|$.

\vspace{0.2cm}

\vspace{0.2cm}
\noindent\textbf{Step IV-H2} ($\inf_{x\in X}\mu(x)>0$)\textbf{:} Once we show that for every fixed $x \in X$
\begin{equation}\label{eq:thm_existence_4}
\sum_{y \in X} \lim_{n\to \infty} w(x,y) |v_{n}(y)|=\lim_{n\to \infty} \sum_{y \in X} w(x,y) |v_{n}(y)|<\infty,
\end{equation}
that is, $v \in \dom(\Delta)$ and $\lim_{n\to \infty} \Delta v_{n}(x) =  \Delta\left(\lim_{n\to \infty}v_{n}\right)(x)$ by dominated convergence, then we can conclude the proof as in the final part of \textbf{Step IV-H1}. Indeed, one of the main issues in the previous steps was to show that the solution $v$ of equation \eqref{eq:step-I-1} belongs to $\dom(\Delta)$ so that $\left(\Psi + \lambda\Delta\right)v$ is well-defined. If this is established, we can apply the same arguments as in \eqref{eq:thm_existence_5}. However, here  \eqref{eq:thm_existence_4} is immediate: By the uniformly lower boundedness of the measure $\mu$ it follows that $\ell^1(X,\mu)\subseteq \ell^\infty(X,\mu)$ and, since $\Psi v=u \in \ell^1(X,\mu)$, we get $v\in \ell^\infty(X,\mu)$ by the surjectivity of $\psi$. So, $v \in \dom(\Delta)$ and   \eqref{eq:thm_existence_4} follows.

\vspace{0.2cm}
\noindent\textbf{Step IV-H3} ($\sup_{x\in X}\frac{\sum_{y\in X}w(x,y)}{\mu(x)}\leq c <\infty$ and 
$\Phi(\ell^1) \subseteq \ell^1$)\textbf{:} The reasoning of the previous step applies here as well. By \eqref{eq:contractivity2} we have $u \in \ell^1(X,\mu)$ and from the hypothesis on $\Phi$ it follows that  $v=\Phi u \in \ell^1(X,\mu)$. Therefore,
$$
\sum_{y \in X} w(x,y)|v(y)|= \sum_{y \in X} \frac{w(x,y)}{\mu(y)}|v(y)|\mu(y)\leq \sum_{y \in X} \sup_{z}\frac{\sum_{x\in X}w(x,z)}{\mu(z)}|v(y)|\mu(y) \leq c\|v\| < \infty.
$$
Thus, $v \in \dom(\Delta)$ and \eqref{eq:thm_existence_4} holds.

\vspace{0.5cm}
\noindent\textbf{Step V} (Constructing a solution when $G$ is not connected)\textbf{:} Assume now that $G$ is not connected and write $X$ as a disjoint union of connected components, that is, $X=\bigsqcup_{k=1}^KY_k$ where $Y_k$ are connected components of $X$ and $K\in \N \cup \{\infty\}$. 

We first observe that, if $u \in \dom(\Delta)$ and $x \in Y_k$, then $\Delta u(x) = \Delta_{k} \boldsymbol{\pi}_k u(x)$ where $\Delta_{k}$ is the formal graph Laplacian associated to the canonical induced subgraph $G_{k}=(Y_k,w_{|Y_k\times Y_k},\kappa_{|Y_k}, \mu_{|Y_k})$ and $\boldsymbol{\pi}_{k}$ is the projection onto $C(Y_k)$. We then write $\mathcal{L}_k \colon \dom(\mathcal{L}_k)\subseteq \ell^1(Y_k,\mu_{|Y_k}) \to \ell^1(Y_k,\mu_{|Y_k})$ where
\begin{align*}
	&\dom(\mathcal{L}_k)\coloneqq \{v \in \ell^1(Y_k,\mu_{|Y_k}) \mid \Phi v \in \dom(\Delta_k),\; \Delta_k\Phi v\in \ell^1(Y_k,\mu_{|Y_k})\},\\
	&\mathcal{L}_k v\coloneqq \Delta_k\Phi v.
\end{align*}
Notice that if $u \in \dom(\mathcal{L})$, then $\boldsymbol{\pi}_{k}u \in \dom(\mathcal{L}_k)$ for every $k$ and $\mathcal{L}u(x)= \mathcal{L}_k\boldsymbol{\pi}_{k}u(x)$ for every $x \in Y_k$.

Now, for every $Y_k$, we fix an exhaustion $\{Y_{k,n}\}_n$ as in Lemma \ref{lem:chain1} and define the set $\Omega_{k}$ associated to the subgraph $G_{k}$ and $\{Y_{k,n}\}_n$ as in Definition \ref{def:Omega}. As we already know from \textbf{Step 0}, 
$\mathcal{L}_k$ is accretive on $\Omega_k$. 

We next define
\begin{equation*}
\Omega \coloneqq \{ u \in \dom(\mathcal{L}) \mid \exists\, \{u_k\}_k \mbox{ s.t. } u_k \in \Omega_k \mbox{ for } k=1,\ldots,K,\, \boldsymbol{\pi}_{k}u=u_k\}.
\end{equation*}
By Lemma \ref{lem:FC0}, $C_c(X)\subseteq \dom(\mathcal{L})$. Furthermore, by Lemma \ref{lem:FC}, for every $u \in C_c(X)$, $u_k\coloneqq \boldsymbol{\pi}_{k}u \in C_c(Y_k)\subseteq \Omega_k$. It follows that $C_c(X) \subseteq \Omega$ and, in particular,  $\overline{\Omega}=\overline{\dom(\mathcal{L})}=\ell^1(X,\mu)$. 
It is not difficult to show that $\mathcal{L}$ is accretive on $\Omega$, as each $\mathcal{L}_k$ is accretive on $\Omega_k$. Indeed, for every $u,v \in \Omega$
\begin{align*}
	\left\|(u - v) + \lambda \left(\mathcal{L}u - \mathcal{L}v\right) \right\|&=\sum_{k=1}^K \left\|(\boldsymbol{\pi}_{k}u - \boldsymbol{\pi}_{k}v) + \lambda \left(\mathcal{L}_k \boldsymbol{\pi}_{k}u - \mathcal{L} \boldsymbol{\pi}_{k}v\right) \right\|_{Y_k}\\
	&= \sum_{k=1}^K \left\|(u_k - v_k) + \lambda \left(\mathcal{L}_k u_k - \mathcal{L}_k v_k\right) \right\|_{Y_k}\\
	&\geq \sum_{k=1}^K\|u_k - v_k\|_{Y_k}= \|u - v\|.
\end{align*}
 
Finally, fix $g \in \ell^1(X,\mu)$ and $\lambda>0$. Clearly, if $g \in \ell^{1,\pm}(X,\mu)$, then $\boldsymbol{\pi}_{k}g\in \ell^{1,\pm}(Y_k,\mu_{|Y_k})$ for every $k$ and, if $G$ satisfies one of the assumptions \ref{m-accretivity_A2}, \ref{m-accretivity_B2} or \ref{m-accretivity_C2}, then $G_k$ satisfies the same property for every $k$. Therefore, for every $k$, let $u_k\in C(Y_k)$ be the unique function in $\Omega_k$ which solves
\begin{equation*}
	(\operatorname{id}_{k}+\lambda\Delta_{k}\Phi)u_k=\boldsymbol{\pi}_{k}g
\end{equation*}
as constructed in \textbf{Steps I} to \textbf{IV} above.  

Now, define
\begin{equation*}
	u(x) \coloneqq u_k(x) \mbox{ if } x \in Y_k.
\end{equation*}
The function $u$ has the following properties: 
\begin{enumerate}[i)]
	\item $\Phi u \in \dom(\Delta)$ since $\Phi u_k \in \dom(\Delta_{k})$ for every $k$;
	\item $u$ solves $(\operatorname{id}+\lambda\Delta\Phi)u=g$.
\end{enumerate}
 Moreover, by \eqref{eq:contractivity2}
\begin{equation*}
 \sum_{k=1}^K \|u_k\|_{Y_k}\leq \sum_{k=1}^K\|\boldsymbol{\pi}_{k}g\|_{Y_k}= \|g\|
\end{equation*}
and thus $\|u\|\leq \|g\|$. Therefore, $u\in \ell^1(X,\mu)$ and $\Delta\Phi u=g-u \in\ell^1(X,\mu)$, that is, $u \in \dom(\mathcal{L})$. In particular, $u \in \Omega$. 

This completes the proof
of Theorem~\ref{thm:main1}.
\qed
\vspace{0.2cm}

\begin{remark}
As the proof shows, the conclusion of \textbf{Step V} follows under the  weaker assumption that at least one of the conditions $g\geq 0$, $g\leq 0$, \ref{m-accretivity_A2}, \ref{m-accretivity_B2} or \ref{m-accretivity_C2} in the statement of Theorem \ref{thm:main1} holds in each connected component $Y_k$ of $X$, not necessarily the 
same condition for different $Y_k$.
\end{remark}

Using the constructions carried out in the proof of Theorem~\ref{thm:main1},
we can now prove the existence and uniqueness of solutions for the \ref{eq:C-D}.
\vspace{0.3cm}

\noindent\textbf{Proof of Theorem \ref{thm:main2}.}
From the definition of mild solutions, Definition \ref{def:weak_solution}, without loss of generality we can suppose that $T<\infty$. Then, since $f\in L^1_{\textnormal{loc}}\left([0,T]; \ell^1\left(X,\mu\right)\right)= L^1\left([0,T]; \ell^1\left(X,\mu\right)\right)$, there exists an $\epsilon$-discretization $\mathcal{D}_\epsilon$ of $([0,T];f)$ for every $\epsilon >0$. Let us observe that solving \eqref{implicit_Euler}, i.e.,
\begin{equation*}
	\frac{u_k - u_{k-1}}{\lambda_k} + \Delta\Phi u_k = f_k, \qquad
	\lambda_k\coloneqq t_k - t_{k-1} 
\end{equation*}
for $k=1,\ldots,n$ means to solve at each step the equation
\begin{equation*}
	(\operatorname{id} + \lambda_k\Delta\Phi) u_k= u_{k-1} + \lambda_kf_k
\end{equation*}
in such a way that
$$
u_k \in \ell^1\left(X,\mu\right), \qquad  \Phi u_k\in\dom(\Delta), \qquad \Delta\Phi u_k\in \ell^1(X,\mu)
$$
where $\lambda_k>0$ and $f_k\in \ell^1\left(X,\mu\right)$. Therefore, given $u_0$ and $\{f_k\}_{k=1}^n$, the solution $\{u_k\}_{k=1}^n$ (if any) of \eqref{implicit_Euler} is computed recursively starting from
\begin{equation}\label{eq:step1}
	(\operatorname{id} + \lambda_1\Delta\Phi) u_1= u_{0} + \lambda_1f_1.
\end{equation}

If $u_0, f_1 \in \ell^1(X,\mu)$ are nonnegative (nonpositive), then $g\coloneqq u_0+\lambda_1f_1 \in \ell^{1,\pm}(X,\mu)$ and by Theorem \ref{thm:main1} there exists a unique nonnegative (nonpositive) solution $u_1\in \Omega$ of \eqref{eq:step1}. Iterating the procedure, each  $u_{k-1} + \lambda_kf_k \in \ell^{1,\pm}(X,\mu)$. Therefore, for every $\epsilon>0$ there exists an $\epsilon$-approximate solution $u_\epsilon$ of the \ref{eq:C-D} (see \eqref{epsilon_approximation}), such that $u_\epsilon(t)\geq 0$ and $u_\epsilon(t)\in \Omega$ for every $t\in(0,T]$. In Theorem \ref{thm:main1} we also proved that $\mathcal{L}_{|\Omega}$ is accretive and $\overline{\Omega}=\overline{\dom(\mathcal{L})}=\ell^1(X,\mu)$ by Lemma \ref{lem:FC}. 

Therefore, summarizing, we have that:
\begin{enumerate}[1)]
	\item By hypotheses \ref{item:nonnegativity/nonpositivity}, \ref{hp1},  and \ref{hp2} we have $u_0 \geq 0$, $u_0 \in \ell^1(X,\mu)=\overline{\dom\left(\mathcal{L}_{|\Omega}\right)}=\overline{\Omega}$,  and $f \in L^1\left([0,T]; \ell^1\left(X,\mu\right)\right)$, respectively;
	\item $\mathcal{L}_{|\Omega}$ is accretive;
	\item For every $u_0\geq 0$ and $f(t)\geq 0$, there exists an $\epsilon$-approximate solution $u_\epsilon$ such that $u_\epsilon(t)\geq 0$ and $u_\epsilon(t)\in \dom\left(\mathcal{L}_{|\Omega}\right)=\Omega$ for every $t\in (0,T]$.
\end{enumerate}
Then, by standard results (see \cite[Theorem 3.3]{benilan1988evolution}  or \cite[Theorem 4.1]{barbu2010nonlinear}), there exists a unique mild solution $u$ of the \ref{eq:C-D} which satisfies \eqref{uniform_limit}. Since the limit is uniform and $u_\epsilon(t)\geq 0$, then $u(t)\geq 0$ and $u(t) \in \ell^1(X,\mu)$ for every $t \in [0,T]$. The validity of \eqref{contraction_of_solutions} is again standard, see \cite[Theorem 4.1]{barbu2010nonlinear}. If $u_0 \leq 0$ and $f(t)\leq 0$, then we get the same results in a completely analogous way.

Under the extra hypothesis \ref{m-accretivity_A}, \ref{m-accretivity_B} or \ref{m-accretivity_C} in Theorem \ref{thm:main1} we have established the $m$-accretivity of $\mathcal{L}_{|\Omega}$ which implies the existence of $\epsilon$-approximate solutions for every $\epsilon$ as above. Therefore, under the hypotheses \ref{hp1}, \ref{hp2} and \ref{hpA}, there exists a unique mild solution $u$ of the \ref{eq:C-D} which satisfies \eqref{uniform_limit} and \eqref{contraction_of_solutions}, see \cite[Corollary 4.1]{barbu2010nonlinear}.
\qed
\vspace{0.2cm}

We now recall the following general result see, e.g., \cite[Proposition 3]{crandall1986nonlinear} or \cite[Theorem 1.6]{benilan1988evolution}.
\begin{proposition}\label{prop:continuity}
	Let $f \in L^1_{\textnormal{loc}}([0,T] \, ;\, \ell^1\left(X,\mu\right))$. Let $\dom(\mathcal{L})$ be closed and let $\mathcal{L}$ be continuous on $\dom(\mathcal{L})$. If $u$ is a mild solution on $(0,T)$, then $u$ is a strong solution and $u$ satisfies for every $0<t<T$
	\begin{equation*}
	u(t) = u(0) - \int_0^t \mathcal{L} u(s) ds +\int_0^t f(s)ds.
	\end{equation*}
	Moreover, if $f \in C([0,T] \, ;\, \ell^1\left(X,\mu\right))$, then $u$ is a classic solution.
\end{proposition}

We now provide a direct application of the proposition above to the graph setting.

\begin{corollary}\label{cor:application}
Let $G=(X,w,\kappa,\mu)$ be a graph. If
\begin{enumerate}[(i)]
\item\label{cor_item1}
$
\sup_{x\in X} \Deg(x) < \infty;
$
\item\label{cor_item2} $\Phi \colon  \ell^1(X,\mu) \to \ell^1(X,\mu)$ is continuous;
\end{enumerate}
then $\dom(\mathcal{L})=\ell^1(X,\mu)$ and $\mathcal{L}$ is continuous. In particular, the conclusions of Proposition \ref{prop:continuity} hold.
   \end{corollary}
\begin{proof}
Let $u\in  \ell^1(X,\mu)$.  By \ref{cor_item2} we have that $\Phi u \in  \ell^1(X,\mu)$ and then by \ref{cor_item1}
\begin{equation*}
\sum_{y\in X} w(x,y) |\phi(u(y))| \leq c_1\sum_{\substack{y\in X}} |\phi(u(y))| \mu(y) <\infty
\end{equation*}
for some $c_1>0$, that is, $\Phi u \in \dom\left(\Delta\right)\cap \ell^1(X,\mu)$. Let us recall  from \cite[Theorem 9.2]{haeseler2012laplacians} or \cite[Theorem 2.15]{keller2021graphs} that the formal graph Laplacian $\Delta$ is bounded on $\ell^1(X,\mu)$ (indeed, on $\ell^p(X, \mu)$ for all $p\in [1,\infty]$) if and only if \ref{cor_item1} holds. Therefore,
$$
\|\Delta\Phi u\| \leq c_2 \|\Phi u\| < \infty
$$
for some $c_2>0$, namely, $\Delta\Phi u \in \ell^1(X,\mu)$ and $\dom(\mathcal{L})=\ell^1(X,\mu)$. Therefore, by \ref{cor_item2}, $\mathcal{L}$ is continuous as the composition of continuous operators is continuous.
\end{proof}
   
\begin{remark}
Observe that, if $G$ is finite, then both hypotheses \ref{cor_item1} and \ref{cor_item2} in Corollary~\ref{cor:application} are trivially satisfied and if $f$ is continuous, then the \ref{eq:C-D} always has a unique classic solution for any $\phi$. 

About hypothesis \ref{cor_item2}, if $G$ is not finite, one condition that ensures the continuity of $\Phi$ is 
if $\phi$ is Lipschitz continuous with uniform Lipschitz constant.
Another sufficient condition for the continuity of the operator $\Phi$ on $\ell^1(X,\mu)$ is that $\mu$ is bounded away from zero, i.e., assumption \ref{m-accretivity_B} and that $\phi$ is uniformly Lipschitz on every interval $[-R,R]$. This is for instance the case of the PME where $\phi(s)=s|s|^{m-1}$ with $m> 1$.

To prove these statements, recall first that we are assuming $\Phi (\ell^1(X,\mu))\subseteq \ell^1(X,\mu)$ and that if $\inf_{x\in X}\mu(x)\geq c>0$, then $\ell^1(X,\mu)\subseteq \ell^\infty(X,\mu)$ and $\|u\|_\infty \leq c^{-1}\|u\|_1$ for every $u \in  \ell^1(X,\mu)$. Therefore, if $\|u_n - u\|_1 \to 0$ then $|u_n-u|$ is uniformly bounded. In particular, there exists $R>0$ such that $u_n(x), u(x) \in [-R,R]$ for every $x\in X$ and for every $n\in \N$. Consequently, $|\phi(u_n(x)) - \phi(u(x))| \leq L_R|u_n(x) - u(x)|$ where $L_R$ is the Lipschitz constant of $\phi$ on $[-R,R]$ and then $\|\Phi u_n - \Phi u\|_1\to 0$.
   \end{remark}

\appendix
\section{Auxiliary results}\label{sec:auxiliary_results}
In this appendix we collect several results which are used in various parts of the paper.
The first result concerns the relationship between the Dirichlet Laplacian and restrictions of the formal Laplacian. For related material, see \cite{keller2012dirichlet,keller2021graphs}.
\begin{lemma}\label{lem:A1}
Let $G=(X,w,k,\mu)$  be a graph, let $A$ be a  subset of $X$ such that $\mathbullet{\partial}A\neq \emptyset$ and let $G_{\textnormal{dir}} =(A, w_{|A\times A}, k_{\textnormal{dir}}, \mu_{|A})$ be the Dirichlet subgraph associated to $A$. Define
$$
\Delta_{|A} \colon \dom\left(\Delta_{|A} \right) \subseteq C(X) \to C(A)
$$
with
\begin{align*}
&\dom\left(\Delta_{|A} \right)\coloneqq \left\{ u \in C(X) \mid u\equiv 0 \mbox{ on } X\setminus A, \, \sum_{y \in A} w(x,y)\left|u(y)\right|< \infty \quad \forall x \in A \right\},\\
&\Delta_{|A}u(x)\coloneqq \Delta u(x) \quad \mbox{for every } x \in A.
\end{align*}
Then, we have the following commutative diagrams
$$
D_1\colon\begin{tikzcd}
 \dom\left(\Deltadir\right) \arrow[dr,"\Deltadir"'] \arrow[hookrightarrow,r,"\boldsymbol{\mathfrak{i}}"]   & \dom\left(\Delta_{|A}\right) \arrow[d,"\Delta_{|A}"] \\
& C(A)
\end{tikzcd}\qquad
D_2\colon\begin{tikzcd}
\dom\left(\Delta_{|A}\right)  \arrow[dr,"\Delta_{|A}"'] \arrow[r,mapsto,"\boldsymbol{\pi}"]   & \dom\left(\Deltadir\right)  \arrow[d,"\Deltadir"] \\
& C(A)
\end{tikzcd}
$$
where $\boldsymbol{\mathfrak{i}}$ and $\boldsymbol{\pi}$ are the canonical embedding and the canonical projection, respectively, as defined in \eqref{eq:embedding-projection}. In particular, we have $\Deltadir \equiv \Delta_{|A}\boldsymbol{\mathfrak{i}}$, $\Deltadir\boldsymbol{\pi} \equiv \Delta_{|A}$ and
\begin{align*}
& \Deltadir v(x) = \Delta\boldsymbol{\mathfrak{i}}v(x) \qquad &\forall v \in \dom(\Deltadir)\subseteq C(A), \; \forall x \in A,\\
&\Deltadir\boldsymbol{\pi}u(x)= \Delta u(x) \qquad & \forall  u \in \dom\left(\Delta_{|A}\right)\subseteq C(X),\; \forall x \in A.
\end{align*}
If every node in $\mathbullet{\partial}A$ is connected to a finite number of nodes in $A$, then
$$
\dom\left(\Delta_{|A}\right) =  \dom\left(\Delta\right) \cap \left\{u \in C(X) \mid u \equiv 0 \mbox{ on } X\setminus A \right\}
$$
 and $\Delta_{|A} u$ can be uniquely extended to $X$ for every
 $u \in \dom\left(\Delta_{|A}\right)$ in such a way that
 $$
 \Delta_{|A} u(x)=\Delta u(x) \quad \forall\, x \in X,
 $$
 that is, $\Delta_{|A}$ is the restriction of $\Delta$ to the set of functions which vanish on $X\setminus A$.
\end{lemma}
\begin{proof}
Clearly, $\Delta_{|A}$ is well-defined and
$$
\dom\left(\Delta_{|A}\right) \supseteq  \dom\left(\Delta\right) \cap \left\{u \in C(X) \mid u \equiv 0 \mbox{ on } X\setminus A \right\}.
$$
If every node in $\mathbullet{\partial}A$ is connected to a finite number of nodes in $A$, then for every $u \in \dom\left(\Delta_{|A}\right)$ and $x \in X\setminus A$
$$
\sum_{y \in X}w(x,y)\left|u(y)\right|=\sum_{y \in \mathring{\partial}A}w(x,y)\left|u(y)\right|< \infty.
$$
Therefore,
$$
\dom\left(\Delta_{|A}\right) \subseteq  \dom\left(\Delta\right) \cap \left\{u \in C(X) \mid u \equiv 0 \mbox{ on } X\setminus A \right\}
$$
so that the two domains are equal as claimed.
Furthermore, for every $u \in \dom\left(\Delta_{|A}\right)$, we can uniquely extend $\Delta_{|A} u$ to $X\setminus A$ so as to satisfy $\Delta_{|A} u= \Delta u$ by defining
$$
\Delta_{|A} u(x) = -\frac{1}{\mu(x)}\sum_{y \in \mathring{\partial}A}w(x,y)u(y) \quad \mbox{for } x \in X\setminus A.
$$

Observe now that
$$
\dom\left(\Delta_{|A} \boldsymbol{\mathfrak{i}}\right) = \left\{ v \in C(A) \mid \boldsymbol{\mathfrak{i}} v \in\dom\left(\Delta_{|A}\right)  \right\} \subseteq C(A)
$$
and if $v \in C(A)$, then
\begin{align*}
\sum_{y \in X}w(x,y)\left|\boldsymbol{\mathfrak{i}}v(y)\right| &=  \sum_{y \in A}w(x,y)\left|v(y)\right| \quad \mbox{for every }x \in A.
\end{align*}
Therefore, it is immediate to check that $\dom\left(\Deltadir\right)= \dom\left(\Delta_{|A}  \boldsymbol{\mathfrak{i}}\right)$.

Finally, if $v \in \dom\left(\Deltadir\right)$, then for every $x \in A$ we have
\begin{align*}
\Deltadir v(x) &= \frac{1}{\mu_{|A}(x)}\sum_{y \in A} w_{|A\times A}(x,y)\left(v(x) - v(y)\right) + \frac{\kappa_{\textnormal{dir}}(x)}{\mu_{|A}(x)} v(x)\\
&= \frac{1}{\mu_{|A}(x)}\sum_{y \in A} w_{|A\times A}(x,y)\left(v(x) - v(y)\right) + \frac{\kappa_{|A}(x)+b_{\textnormal{dir}}(x)}{\mu_{|A}(x)} v(x)\\
&= \frac{1}{\mu(x)}\sum_{y \in A} w(x,y)\left(v(x) - v(y)\right) +  \frac{\sum_{y \in  \mathbullet{\partial} A}w(x,y)}{\mu(x)}v(x)+\frac{\kappa(x)}{\mu(x)} v(x)\\
&= \frac{1}{\mu(x)}\sum_{y \in X} w(x,y)\left(\boldsymbol{\mathfrak{i}} v(x) - \boldsymbol{\mathfrak{i}}v(y)\right) + \frac{\kappa(x)}{\mu(x)} \boldsymbol{\mathfrak{i}}v(x)\\
&= \Delta\boldsymbol{\mathfrak{i}}v(x).
\end{align*}
This concludes the proof of diagram $D_1$. The proof of diagram $D_2$ is basically the same following suitable modifications.
\end{proof}

The next theorem is a comparison principle for a nonlinear operator. This result generalizes \cite[Theorem 2, Section 23.1]{collatz1966functional} and \cite[Theorem 8 and Proposition 3.1]{keller2012dirichlet}, see also the proof of Theorem 1.3.1 in \cite{wojciechowski2008stochastic}. In particular, we relax the assumptions on the function $u$ by letting it not attain a minimum or maximum on $X$ if the graph $G$ does not have any infinite paths. We recall that for us a path is a walk without any repeated nodes.

\begin{thm}[Comparison principle]\label{thm:min_principle}
Let $G=\left(X, w, \kappa, \mu\right)$ be a connected graph. Let $\lambda>0$ and $v \in \dom\left(\Delta\right)$. Assuming that $\psi \colon \R \to \R$ is strictly monotone increasing and $\psi(0)\leq 0$ we consider three cases:

\emph{\textbf{Case 1)}} 
There exists $x_0 \in X$ such that $u$ attains a minimum at $x_0$, i.e.,
$$v(x_{0})=\inf_{x\in X} \{v(x)\} > -\infty.$$

\emph{\textbf{Case 2)}} $G$ does not contain any infinite path.

\emph{\textbf{Case 3)}} $G$ has an infinite path and for every infinite path $\{x_{n}\}$ we have $\sum_{x_{n}} \mu(x_{n})=\infty$ and there exists $p>0$ such that $\sum_{x_{n}}|v(x_{n})|^p\mu(x_{n})<\infty$ .

 In all the three cases, if $\left( \Psi + \lambda\Delta\right)v \geq 0$, then $v \geq 0$.

 \smallskip

Assuming instead $\psi(0)\geq 0$ and substituting $v(x_0)=\sup_{x\in X} \{v(x)\} <\infty$ for $v(x_0)=\inf_{x\in X} \{v(x)\} > -\infty$ in Case 1, if $\left( \Psi + \lambda\Delta\right)v \leq 0$, then $v \leq 0$. Moreover, in any case, if $v(x)=0$ for some $x\in X$, then $v\equiv 0$.
\end{thm}

\begin{proof}
Let $v \in \dom\left(\Delta\right)$ be such that $\left( \Psi + \lambda\Delta\right)v \geq 0$ for $\psi(0)\leq 0$ strictly monotone increasing. If $v\geq0$, then there is nothing to prove.
Hence, we assume that there exists $x_{0}\in X$ such that $v(x_{0}) <0$. 
We will show that this leads to a contradiction in all three cases.

Since $\psi(0)\leq0$ and $\psi$ is strictly monotone increasing,
\begin{equation}\label{eq:min_1}
	\psi(v(x_{0}))+ \lambda\frac{\kappa(x_{0})}{\mu(x_{0})}v(x_{0}) < 0.
\end{equation}
Furthermore, as $\left( \Psi + \lambda\Delta\right)v(x_0) \geq 0$,
\begin{equation}\label{eq:min_2}
	0\leq  \psi(v(x_{0}))+\frac{\lambda}{\mu(x_{0})}\sum_{y\in X}w(x_{0},y)\left(v(x_{0})-v(y)\right) + \lambda\frac{\kappa(x_{0})}{\mu(x_{0})}v(x_{0}).
\end{equation}
Combining the above inequalities \eqref{eq:min_1} and \eqref{eq:min_2}, we get
\begin{equation*}
	0<-\left[ \psi(v(x_{0})) +\lambda\frac{\kappa(x_{0})}{\mu(x_{0})}v(x_{0}) \right]\leq \frac{\lambda}{\mu(x_{0})}\sum_{y\in X}w(x_{0},y)\left(v(x_{0})-v(y)\right)
\end{equation*}
and because $w(\cdot,\cdot)\geq 0$ and $G$ is connected, there exists $y=x_{1}\sim x_{0}$ such that $v(x_{1})< v(x_{0})$. In particular, $v(x_1)<0$.

Hence, we see that every node where $v$ is negative is connected to a node where $v$ is strictly smaller. This is the basic observation that will be used in all three cases.

\vspace{0.5cm}
\noindent\textbf{Case 1)} From the discussion above, it is clear that $v$ cannot achieve a negative minimum.

\vspace{0.5cm}

\noindent\textbf{Case 2)} Iterating the procedure above, we find a sequence of distinct nodes $\{x_{k}\}_{k=0}^n$ such that $x_{0}\sim x_{1}\sim \cdots \sim x_{n}$ and
$$
v(x_{n})< v(x_{n-1}) < \ldots < v(x_{0}) <0.
$$
Since $G$ does not have any infinite path this sequence must end which leads to a contradiction.

\vspace{0.5cm}

\noindent\textbf{Case 3)} 
In this case, we can obtain an infinite sequence $\{v(x_{n})\}_n$ such that $\{x_{n}\}_n$  is an infinite path and
$$
\ldots < v(x_{n})< v(x_{n-1}) < \ldots < v(x_{0}) <0.
$$
It follows that $|v(x_{n})|>|v(x_{0})|>0$, for every $n$, and therefore $$\sum_{{n}}|v(x_{n})|^p\mu(x_{n})>|v(x_{0})|^p\sum_{n}\mu(x_{n})=\infty$$ for every $p>0$ which gives a contradiction.
\vspace{0.5cm}

Hence, we have established that $v \geq 0$ in all three cases. Now, if there exists $x_0 \in X$ such that $v(x_0)=0$, then
$$0\leq  \Psi v(x_{0}) + \lambda\Delta v(x_{0}) =-\frac{\lambda}{\mu(x_{0})}\sum_{y\in X}w(x_{0},y)v(y) \leq 0 $$
and thus $v(y)=0$ for all $y\sim x_0$.
Using induction and the assumption that $G$ is connected we get $v\equiv 0$.

The proof that $\left( \Psi + \lambda\Delta\right) v \leq 0$ implies $v \leq 0$ when $\psi(0)\geq 0$ is completely analogous.
\end{proof}

\begin{remark}
\label{remark:strongmin}
The above proof also shows that if $v$ satisfies $\left( \Psi + \lambda\Delta\right)v \geq 0$,  on any connected graph, $\psi(0)\leq 0$ and $\min_{x \in X} v(x)\leq 0$, then $v\equiv 0$.
\end{remark}

\begin{remark}
\label{cor:comparison}
A closer look at the proof shows that the conclusion of the theorem also holds for the operator
\[
\sigma \Psi+\lambda \Delta
\]
where $\sigma\in C(X)$ is positive.
\end{remark}

\begin{corollary}\label{cor:min}
With the hypotheses of Theorem \ref{thm:min_principle}, let $\psi(0)= 0$ for  $\psi \colon \R \to \R$ is strictly monotone increasing. 
Let $v_1, v_2$ be solutions of
$$
\left(\Psi + \lambda\Delta \right) v_k = g_k, \qquad \lambda>0 \mbox{ and } g_k \in C(X) \mbox{ for } k=1,2.
$$
If $g_1\geq g_2$, then $v_1 \geq v_2$.
\end{corollary}
\begin{proof}
Notice that, since $\psi$ is strictly increasing and $\psi(0)=0$, there exists a positive function $\sigma\in C(X)$ such that
\[
\psi(v_1(x))-\psi(v_2(x))=\sigma(x) \psi(v_1(x)-v_2(x))
\]
for all $x \in X$.
Indeed, we can define
\[
\sigma(x)=
\begin{cases}
1&\text{if } v_1(x)=v_2(x)\\
\frac{\psi(v_1(x))-\psi(v_2(x))}{\psi(v_1(x)-v_2(x))} &\text{otherwise}.
\end{cases}
\]
It follows that $v_1-v_2$ satisfies
\[
\sigma\Psi(v_1-v_2) + \lambda \Delta(v_1-v_2)= g_1-g_2\geq 0 \,\,\text{ on } X
\]
and the conclusion follows from Remark~\ref{cor:comparison}.
\end{proof}

\begin{remark}
\label{rmk:strongmin}
As in Remark \ref{remark:strongmin}, the proof shows that if $v_k$ satisfy $\left(\Psi + \lambda\Delta \right) v_k = g_k$, $g_1\geq g_2$ and $\min_{x \in X}\{v_1(x)-v_2(x)\}\leq 0$, then $v_1 \equiv v_2$ and $g_1 \equiv g_2$ on $X$
\end{remark}

Finally, in the following two lemmas we discuss how to exhaust the graph via finite subgraphs
which are nested and such that each subgraph is connected to the next subgraph. We note that as
we do not assume local finiteness, we have to take a little bit of care in how we choose the exhaustion.
Although this should certainly be well-known, for the convenience of the reader
we include a short proof.

Let $d$ denote the combinatorial graph metric, that is, the least number of edges
in a path connecting two nodes. Fix a node $x_1 \in X$ and let $S_r=S_r(x_1)$
denote the sphere of radius $r = 0, 1, 2, \ldots$ about $x_0$, that is,
$$S_r =\{ x \in X \mid d(x,x_0)=r \}.$$
For $x \in S_r$, we call $y \in X$ a \emph{forward neighbor} of $x$ if $y \sim x$
and $y \in S_{r+1}$. We will denote the set of forward neighbors of $x$ via $N_+(x)$, i.e.,
$$N_+(x) = \{ y \sim x \mid d(y,x_0) =d(x,x_0)+1 \}.$$
We now use the set of forward neighbors to inductively create our exhaustion sequence.

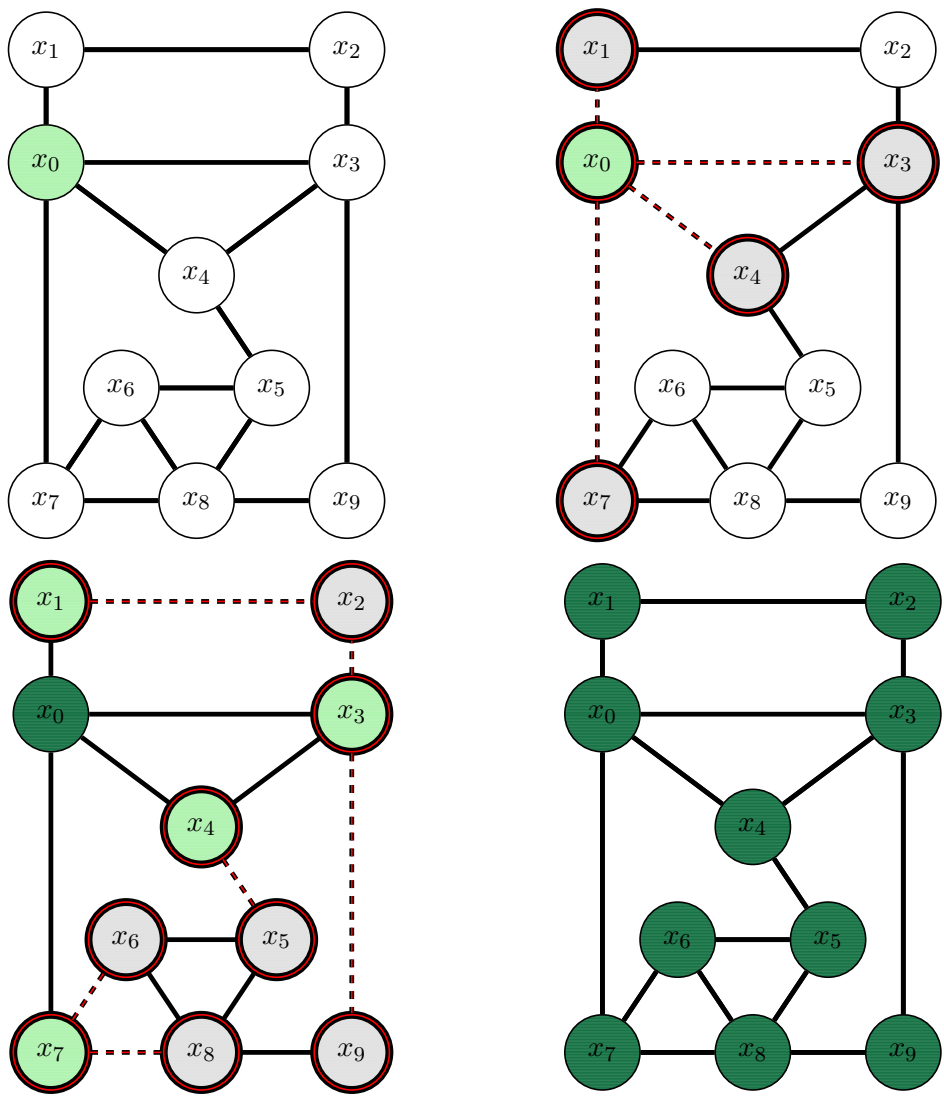
\begin{figure}[!bt]
	\centering
	\begin{minipage}{.45\textwidth}
		\begin {center}
		\begin {tikzpicture}[-latex ,auto ,node distance =1.5 cm and 1cm ,on grid ,
		semithick ,
		whitestyle/.style={circle,draw,fill=white,minimum size=1cm},
		blackstyle/.style ={ circle ,top color =black, bottom color = black,
			draw, white, minimum size =1cm},
		gray-orangestyle/.style ={ circle ,top color =white , bottom color = white ,
			draw, double=red,very thick, black, text=black, minimum size =1cm},
		green-redstyle/.style ={ circle ,top color = darkpastelgreen , bottom color =darkpastelgreen ,
			draw,double=red,very thick, black , text=black, minimum size=1cm},
		greenstyle/.style={circle ,top color =lightgreen!70, bottom color = lightgreen!70,
			draw,black , text=black ,minimum size=1cm},
		ghost/.style={circle,fill=white,minimum size=0.7cm}]
		\node[whitestyle] (C){$x_8$};
		\node[whitestyle] (A) [above left=of C] {$x_6$};
		\node[whitestyle] (B) [above right =of C] {$x_5$};
		\node[whitestyle] (D) [below left=of A] {$x_7$};
		\node[whitestyle] (E) [below right=of B] {$x_9$};
		\node[whitestyle] (F) [above right=of A] {$x_{4}$};
		\node[ghost] (g1) [above =of D] {};
		\node[ghost] (g2) [above =of E] {};
		\node[ghost] (g3) [above =of g1] {};
		\node[ghost] (g4) [above =of g2] {};
		\node[whitestyle] (G) [above =of g4] {$x_3$};
		\node[greenstyle] (H) [above =of g3] {$x_0$};
		\node[whitestyle] (I) [above=of H] {$x_1$};
		\node[whitestyle] (L) [above=of G] {$x_2$};
		\path (C) edge [double=black,-] node[] {} (A);
		\path (A) edge [double=black,-] node[] {} (C);
		\path (A) edge [double=black,-] node[] {} (B);
		\path (B) edge [double=black,-] node[] {} (A);
		\path (C) edge [double=black,-] node[] {} (B);
		\path (B) edge [double=black,-] node[] {} (C);
		\path (B) edge [double=black,-] node[] {} (F);
		\path (F) edge [double=black,-] node[] {} (B);
		\path (F) edge [double=black,-] node[] {} (H);
		\path (H) edge [double=black,-] node[] {} (F);
		\path (H) edge [double=black,-] node[] {} (D);
		\path (D) edge [double=black,-] node[] {} (H);
		\path (D) edge [double=black,-] node[] {} (C);
		\path (C) edge [double=black,-] node[] {} (D);
		\path (D) edge [double=black,-] node[] {} (A);
		\path (A) edge [double=black,-] node[] {} (D);
		\path (C) edge [double=black,-] node[] {} (E);
		\path (E) edge [double=black,-] node[] {} (C);
		\path (E) edge[double=black,-] node[] {} (G);
		\path (G) edge[double=black,-] node[] {} (E);
		\path (G) edge[double=black,-] node[] {} (L);
		\path (L) edge[double=black,-] node[] {} (G);
		\path (I) edge[double=black,-] node[] {} (H);
		\path (H) edge[double=black,-] node[] {} (I);
		\path (I) edge[double=black,-] node[above] {} (L);
		\path (L) edge[double=black,-] node[above] {} (I);
		\path (F) edge[double=black,-] node[] {} (G);
		\path (G) edge[double=black,-] node[] {} (F);
		\draw[latex'-latex',double] (G) edge[double=black,-] node [left] {} (H);
	\end{tikzpicture}
\end{center}\end{minipage}
\begin{minipage}{.45\textwidth}
\centering
\begin {tikzpicture}[-latex ,auto ,node distance =1.5 cm and 1cm ,on grid ,
semithick ,
whitestyle/.style={circle,draw,fill=white,minimum size=1cm},
blackstyle/.style ={ circle ,top color =black, bottom color = black,
	draw, white, minimum size =1cm},
gray-orangestyle/.style ={ circle ,top color =lightgray!70 , bottom color = lightgray!70 ,
	draw, double=red,very thick, black, text=black, minimum size =1cm},
green-redstyle/.style ={ circle ,top color = lightgreen!70 , bottom color =lightgreen!70 ,
	draw,double=red,very thick, black , text=black, minimum size=1cm},
greenstyle/.style={circle ,top color =darkspringgreen , bottom color = darkspringgreen,
	draw,black , text=black ,minimum size=1cm},
ghost/.style={circle,fill=white,minimum size=0.7cm}]
\node[whitestyle] (C){$x_8$};
\node[whitestyle] (A) [above left=of C] {$x_6$};
\node[whitestyle] (B) [above right =of C] {$x_5$};
\node[gray-orangestyle] (D) [below left=of A] {$x_7$};
\node[whitestyle] (E) [below right=of B] {$x_9$};
\node[gray-orangestyle] (F) [above right=of A] {$x_{4}$};
\node[ghost] (g1) [above =of D] {};
\node[ghost] (g2) [above =of E] {};
\node[ghost] (g3) [above =of g1] {};
\node[ghost] (g4) [above =of g2] {};
\node[gray-orangestyle] (G) [above =of g4] {$x_3$};
\node[green-redstyle] (H) [above =of g3] {$x_0$};
\node[gray-orangestyle] (I) [above=of H] {$x_1$};
\node[whitestyle] (L) [above=of G] {$x_2$};
\path (H) edge [double=red,  dashed,-]  (F);
\path (H) edge [double=red,  dashed,-]  (I);
\path (H) edge [double=red,  dashed,-]  (G);
\path (H) edge [double=red,  dashed,-]  (D);

\path (I) edge [double=black, -]  (L);
\path (G) edge [double=black, -]  (L);
\path (G) edge [double=black, -]  (E);
\path (G) edge [double=black, -]  (F);
\path (F) edge [double=black, -]  (B);
\path (D) edge [double=black, -]  (A);
\path (A) edge [double=black, -]  (C);
\path (D) edge [double=black, -]  (C);
\path (C) edge [double=black, -]  (E);
\path (C) edge [double=black, -]  (B);
\path (A) edge [double=black, -]  (B);
\end{tikzpicture}
\end{minipage}
\vspace{0.2cm}

\begin{minipage}{.45\textwidth}
\begin {center}
\begin {tikzpicture}[-latex ,auto ,node distance =1.5 cm and 1cm ,on grid ,
semithick ,
whitestyle/.style={circle,draw,fill=white,minimum size=1cm},
blackstyle/.style ={ circle ,top color =black, bottom color = black,
draw, white, minimum size =1cm},
gray-orangestyle/.style ={ circle ,top color =lightgray!70 , bottom color = lightgray!70 ,
draw, double=red,very thick, black, text=black, minimum size =1cm},
green-redstyle/.style ={ circle ,top color = lightgreen!70 , bottom color =lightgreen!70 ,
draw,double=red,very thick, black , text=black, minimum size=1cm},
greenstyle/.style={circle ,top color =darkspringgreen , bottom color = darkspringgreen,
draw,black , text=black ,minimum size=1cm},
ghost/.style={circle,fill=white,minimum size=0.7cm}]
\node[gray-orangestyle] (C){$x_8$};
\node[gray-orangestyle] (A) [above left=of C] {$x_6$};
\node[gray-orangestyle] (B) [above right =of C] {$x_5$};
\node[green-redstyle] (D) [below left=of A] {$x_7$};
\node[gray-orangestyle] (E) [below right=of B] {$x_9$};
\node[green-redstyle] (F) [above right=of A] {$x_{4}$};
\node[ghost] (g1) [above =of D] {};
\node[ghost] (g2) [above =of E] {};
\node[ghost] (g3) [above =of g1] {};
\node[ghost] (g4) [above =of g2] {};
\node[green-redstyle] (G) [above =of g4] {$x_3$};
\node[greenstyle] (H) [above =of g3] {$x_0$};
\node[green-redstyle] (I) [above=of H] {$x_1$};
\node[gray-orangestyle] (L) [above=of G] {$x_2$};
\path (I) edge [double=red, dashed,-]  (L);
\path (G) edge [double=red, dashed,-]  (L);
\path (G) edge [double=red, dashed,-]  (E);
\path (F) edge [double=red, dashed,-]  (B);
\path (D) edge [double=red, dashed,-]  (A);
\path (D) edge [double=red, dashed,-]  (C);

\path (H) edge [double=black, -]  (F);
\path (H) edge [double=black, -]  (I);
\path (H) edge [double=black, -]  (G);
\path (H) edge [double=black, -]  (D);
\path (F) edge [double=black, -]  (G);
\path (A) edge [double=black, -]  (B);
\path (A) edge [double=black, -]  (C);
\path (C) edge [double=black, -]  (B);
\path (C) edge [double=black, -]  (E);
\end{tikzpicture}
\end{center}\end{minipage}
\begin{minipage}{.45\textwidth}
\centering
\begin {tikzpicture}[-latex ,auto ,node distance =1.5 cm and 1cm ,on grid ,
semithick ,
whitestyle/.style={circle,draw,fill=white!40,minimum size=1cm},
blackstyle/.style ={ circle ,top color =black, bottom color = black,
draw, white, minimum size =1cm},
gray-orangestyle/.style ={ circle ,top color =white , bottom color = white ,
draw, double=red,very thick, black, text=black, minimum size =1cm},
green-redstyle/.style ={ circle ,top color = darkspringgreen , bottom color =darkspringgreen ,
draw,double=red,very thick, black , text=black, minimum size=1cm},
greenstyle/.style={circle ,top color =darkspringgreen , bottom color = darkspringgreen,
draw,black , text=black ,minimum size=1cm},
ghost/.style={circle,fill=white,minimum size=0.7cm}]
\node[greenstyle] (C){$x_8$};
\node[greenstyle] (A) [above left=of C] {$x_6$};
\node[greenstyle] (B) [above right =of C] {$x_5$};
\node[greenstyle] (D) [below left=of A] {$x_7$};
\node[greenstyle] (E) [below right=of B] {$x_9$};
\node[greenstyle] (F) [above right=of A] {$x_{4}$};
\node[ghost] (g1) [above =of D] {};
\node[ghost] (g2) [above =of E] {};
\node[ghost] (g3) [above =of g1] {};
\node[ghost] (g4) [above =of g2] {};
\node[greenstyle] (G) [above =of g4] {$x_3$};
\node[greenstyle] (H) [above =of g3] {$x_0$};
\node[greenstyle] (I) [above=of H] {$x_1$};
\node[greenstyle] (L) [above=of G] {$x_2$};
\path (C) edge [double=black,-] node[] {} (A);
\path (A) edge [double=black,-] node[] {} (C);
\path (A) edge [double=black,-] node[] {} (B);
\path (B) edge [double=black,-] node[] {} (A);
\path (C) edge [double=black,-] node[] {} (B);
\path (B) edge [double=black,-] node[] {} (C);
\path (B) edge [double=black,-] node[] {} (F);
\path (F) edge [double=black,-] node[] {} (B);
\path (F) edge [double=black,-] node[] {} (H);
\path (H) edge [double=black,-] node[] {} (F);
\path (H) edge [double=black,-] node[] {} (D);
\path (D) edge [double=black,-] node[] {} (H);
\path (D) edge [double=black,-] node[] {} (C);
\path (C) edge [double=black,-] node[] {} (D);
\path (D) edge [double=black,-] node[] {} (A);
\path (A) edge [double=black,-] node[] {} (D);
\path (C) edge [double=black,-] node[] {} (E);
\path (E) edge [double=black,-] node[] {} (C);
\path (E) edge[double=black,-] node[] {} (G);
\path (G) edge[double=black,-] node[] {} (E);
\path (G) edge[double=black,-] node[] {} (L);
\path (L) edge[double=black,-] node[] {} (G);
\path (I) edge[double=black,-] node[] {} (H);
\path (H) edge[double=black,-] node[] {} (I);
\path (I) edge[double=black,-] node[above] {} (L);
\path (L) edge[double=black,-] node[above] {} (I);
\path (F) edge[double=black,-] node[] {} (G);
\path (G) edge[double=black,-] node[] {} (F);
\draw[latex'-latex',double] (G) edge[double=black,-] node [left] {} (H);
\end{tikzpicture}
\end{minipage}
\caption{Exhaustion of a locally finite graph $G$ by a chain of Dirichlet subgraphs $\{G_{\textnormal{dir},n}\}_n$. The figure is read from left to right, from top to bottom. The chain $\{G_{\textnormal{dir},n}\}_n$ is built starting from an inner node $x_0$ by applying the recursive procedure described in the proof of Lemma \ref{lem:chain2}. At each step $n=1,2,3,\ldots$, the interior nodes in $\mathring{X}_n$ are green while the inner boundary nodes in $\mathring{\partial}X_n$ are light green. The nodes belonging to the exterior boundary $\mathbullet{\partial}X_n \subseteq X\setminus X_n$ are colored in light gray. We can visually see how every subgraph $G_{\textnormal{dir},n}$ is still ``chained'' to the supergraph $G$ by the Dirichlet killing term $\kappa_{\textnormal{dir},n}$ which is depicted by red rings and red dashed lines. In particular, each $G_{\textnormal{dir},n}$ is a Dirichlet subgraph of $G_{\textnormal{dir},n+1}$.}\label{fig:exhaustion}
\end{figure}

\begin{lemma}\label{lem:chain1}
	Let $G=(X,w,\kappa,\mu)$ be a connected and infinite graph.
	Then there exists a sequence of connected and finite induced subgraphs $G_n=(X_n, w_n, \kappa_n, \mu_n)$
	with $X_n\subset X_{n+1}$, $\bigcup_{n=1}^\infty X_n= X$ and
	$$
	\{ x \in X_n \mid x\sim y \mbox{ for some } y \in X_{n+1}\setminus X_n \}\neq \emptyset
	$$
	for all $n\in \mathbb{N}_0.$
\end{lemma}
\begin{proof}We arrange the forward neighbors of each node in a sequence.
	Let $X_1 = \{x_0\}$. For $X_2$ choose the first forward neighbor of $x_0$ and add it to $X_0$,
	that is, $X_2=\{x_0, x_1 \}$ where $N_+(x_0)=\{x_1,x_2, \dots\}$. We note that $N_+(x_0) \neq \emptyset$
	as the graph is infinite and connected.
 
	Now, proceed inductively as follows: Given $X_n$
	let $X_{n+1}$ consist of $X_n$ and, for every node  $x $ in $X_n$ we add to $X_n$ the first forward neighbor in $N_+(x)$ which is not included in $X_n$ to get $X_{n+1}$. Let $G_n$ denote the induced subgraph.
	
	As we only add at most a single forward neighbor for each node at every step, it follows that each $X_n$ is finite
	with $|X_n| \leq 2^{n}$. It is clear by construction that $G_n$ is connected.
	Furthermore, as the graph is infinite and connected, it follows that
	$\{ x \in X_n \mid x\sim y \mbox{ for some } y \in X_{n+1}\setminus X_n \}\neq \emptyset$
	for each $n \in \mathbb{N}$. Finally, to show that the union of the $X_n$ is the entire node set, let
	$x \in X$. Then, $x \in S_r$ for some $r$ which means that there exists a sequence
	$\{y_k\}_{k=0}^r$ with $y_0=x_0$, $y_r=x$ and $y_k \in S_k$ such that $y_{k+1} \in N_{+}(y_k)$.
	As each node $y_k$ will then be included in some set of the exhaustion $X_n$, it follows that
	$x \in \bigcup_{n=1}^\infty X_n$. This completes the proof.
\end{proof}

In the locally finite case, the above can be simplified by just using balls for our exhaustion sets. See Figure \ref{fig:exhaustion} for a visual representation. In this case, it is also possible to exhaust in such a way that we have an inclusion between the interiors of the exhaustion sets.

\begin{lemma}\label{lem:chain2}
Let $G=(X,w,\kappa,\mu)$ be connected, infinite and locally finite. Then there exists a sequence of connected and finite induced subgraphs $G_n=(X_n, w_n, \kappa_n, \mu_n)$
with $\mathring{X}_n\subset \mathring{X}_{n+1}$, $X_n\subset X_{n+1}$,  $\bigcup_{n=1}^\infty \mathring{X}_n= X$ and
$$
\{ x \in X_n \mid x\sim y \mbox{ for some } y \in X_{n+1}\setminus X_n \}\neq \emptyset
$$
for all $n\in \mathbb{N}$.
\end{lemma}
\begin{proof}
We modify the construction of the previous lemma: We take $X_1=\{x_0\}$ and, having constructed $X_n$, we  add to it all forward neighbors of nodes in $X_n$ to get $X_{n+1}$. Thus, $X_{n+1}=B_n(x_0)=\{x \mid d(x,x_0)\leq n\}.$  Since the graph is locally finite and connected it is clear that $\cup_n X_n=X$ and, since $G$ is infinite, for every $n$ at least one node in $X_n$ has a forward neighbor so that
$$
\{ x \in X_n \mid x\sim y \mbox{ for some } y \in X_{n+1}\setminus X_n \}\neq \emptyset.
$$
Finally, since a node $x_n$ in $X_n$ has no forward neighbors if and only if it belong to $\mathring{X}_n$ is follows that $\mathring{X}_n\subset \mathring{X}_{n+1}$.
\end{proof}

\section{Accretivity}\label{sec:appendix2}
In this appendix we prove that there exists a dense subset $\Omega$ of $\dom(\mathcal{L})$ where $\mathcal{L}$ is accretive. This subset is of particular importance because every solution that is constructed while carrying out the proof of Theorem \ref{thm:main1} belongs to $\Omega$. 

From now on, if $G$ is infinite, then we fix an exhaustion $\{X_n\}_{n=1}^\infty$ of $X$, i.e., a sequence of subsets $X_n$ of $X$ such that $X_n \subseteq X_{n+1}$
and $X = \cup_{n=1}^\infty X_n$, where we additionally assume that each $X_n$ is finite.  We denote by $\boldsymbol{\mathfrak{i}}_{n,\infty}$ the canonical embedding and by $\boldsymbol{\pi}_n$  the canonical projection for each $X_n$:
\begin{align*}
&\boldsymbol{\mathfrak{i}}_{n,\infty} \colon C(X_n)\to C(X) \quad &\boldsymbol{\mathfrak{i}}_{n,\infty}u(x) \coloneqq \begin{cases}
u(x) & \mbox{if } x \in X_n,\\
0   & \mbox{if }  x \in X\setminus X_n;
\end{cases}\\
&\boldsymbol{\pi}_n \colon C(X) \to C(X_n) \quad &\boldsymbol{\pi}_nu(x) \coloneqq u(x) \mbox{ for every } x \in X_n.
\end{align*}
We remark that, for the purpose of the results collected here, the exhaustion $\{X_n\}_{n=1}^\infty$ is not required to satisfy any additional properties other than that each $X_n$ is finite.

We recall that on a graph $G=(X,w,\kappa,\mu)$ the operator $\mathcal{L} \colon \dom\left( \mathcal{L} \right)\subseteq \ell^{1}\left(X,\mu\right) \to \ell^{1}\left(X,\mu\right)$ is given by
\begin{align*}
		&\dom\left( \mathcal{L} \right)\coloneqq\left\{ u \in  \ell^{1}\left(X,\mu\right) \mid  \Phi u\in \dom\left(\Delta\right), \Delta\Phi u \in  \ell^{1}\left(X,\mu\right)  \right\}\\
		&\mathcal{L}u\coloneqq \Delta\Phi u.
\end{align*}
For a subset $\Omega \subseteq \dom\left( \mathcal{L} \right)$, we write $\mathcal{L}_{|\Omega}$
for the restriction of $\mathcal{L}$ to $\Omega$.

We first introduce a sequence of operators $\Lmin$ whose purpose is to `nicely approximate' the operator $\mathcal{L}$.
\begin{definition}[The operators $\mathcal{L}_n$]\label{def:Lmin}
	Let $G=(X,w,\kappa,\mu)$ be a graph. We define 
	$$
	\Lmin \colon \dom\left( \Lmin \right)\subseteq \ell^{1}\left(X,\mu\right) \to \ell^{1}\left(X,\mu\right)
	$$
	by
	\begin{align*}
		\dom\left( \Lmin \right)\coloneqq C_c(X), \quad \Lmin u\coloneqq \boldsymbol{\mathfrak{i}}_{n,\infty}\Deltadirn\Phi\boldsymbol{\pi}_{n} u
	\end{align*}
	where $\Deltadirn$ is the graph Laplacian associated to the Dirichlet subgraph $G_{\textnormal{dir},n} \subseteq G$ on the node set $X_n$.
\end{definition}

We are going to prove that $\Lmin$ is accretive for every $n$. This result will be a consequence of the next proposition
for finite graphs.
\begin{proposition}\label{prop:nonnegativity}
	Let $G=(X,w,\kappa,\mu)$ be a  finite graph. Then,
	\begin{equation*}
		\sum_{\substack{x\in X\colon\\u(x)\neq v(x)}} \left(\Delta\Phi u(x)-\Delta \Phi v(x)\right)\operatorname{sgn}(u(x)-v(x))\mu(x) \geq 0 \quad \forall\; u, v \in C(X).
	\end{equation*}
\end{proposition}
\begin{proof}
	Define $h\coloneqq \Phi  u-\Phi  v \in C(X)$.
	Since $\phi$ is strictly monotone increasing and $\phi(0)=0$
	\begin{equation}\label{eq:thm1_0}
	\operatorname{sgn}( u(x)- v(x))= \operatorname{sgn}(\phi( u(x))-\phi( v(x)))=\operatorname{sgn}(h(x))
	\end{equation}
	and, therefore,
	\begin{equation}\label{eq:thm1_1}
		(\phi(u(x))-\phi(v(x))) \operatorname{sgn}(u(x)-v(x)) = h(x)  \operatorname{sgn}(h(x))= \left| h(x) \right|\geq 0 \quad \forall\, x \in X.
	\end{equation}
	By the linearity of $\Delta$ and \eqref{eq:thm1_0} we get
	\begin{equation*}\label{eq:appB1}
		\sum_{\substack{x\in X\colon\\u(x)\neq v(x)}} \left(\Delta\Phi u(x)-\Delta  \Phi v(x)\right)\operatorname{sgn}(u(x)-v(x))\mu(x)=\sum_{x\in X} \Delta h(x)\operatorname{sgn}(h(x))\mu(x).
	\end{equation*}
	Since $G$ is finite, from the Green's identity, see \cite{haeseler2011generalized, keller2021graphs}, we get
	\begin{align*}
		\sum_{x\in X} \Delta h(x)\operatorname{sgn}(h(x))\mu(x)&=  \frac{1}{2}\sum_{x,y \in X}  w(x,y)\left(\operatorname{sgn}(h(x)) - \operatorname{sgn}(h(y)) \right)\left(h(x) - h(y)\right)\\
		&+ \sum_{x\in X} \kappa(x)	h(x) \operatorname{sgn}(h(x)).
	\end{align*}
	Combining the above identity with \eqref{eq:thm1_1}, we obtain 
	\begin{equation*}
		\sum_{x\in X} \Delta h(x)\operatorname{sgn}(h(x))\mu(x)\geq \frac{1}{2}\sum_{x,y \in X}  w(x,y)\left(\operatorname{sgn}(h(x)) - \operatorname{sgn}(h(y)) \right)\left(h(x) - h(y)\right).
	\end{equation*}
	
	Setting, for ease of notation,
	$$
	\Gamma(x,y)\coloneqq w(x,y)\left(\operatorname{sgn}(h(x))- \operatorname{sgn}(h(y))\right)\left(h(x) - h(y)\right)
	$$
	we have:
	\begin{enumerate}[i)]
		\item If $\operatorname{sgn}(h(x))=0$, then $\Gamma(x,y)=w(x,y)|h(y)|\geq 0$;
		\item If $\operatorname{sgn}(h(y))=0$, then $\Gamma(x,y)=w(x,y)|h(x)|\geq 0$;
		\item If $\operatorname{sgn}(h(x))=\operatorname{sgn}(h(y))$, then $\Gamma(x,y)=0$;
		\item If $\operatorname{sgn}(h(y))=-\operatorname{sgn}(h(x))$,  then $\Gamma(x,y)= 2w(x,y)\left(|h(x)| + |h(y)|\right)\geq 0$.
	\end{enumerate}
	We obtain $\Gamma(x,y)\geq 0$ for every $x,y \in X$ and the required conclusion follows.
\end{proof}

We next show that $\mathcal{L}$ is accretive on finite graphs.

\begin{corollary}\label{cor:accretivityL-finite}
	Let $G=(X,w,\kappa,\mu)$ be a finite graph. Then, $\mathcal{L}$ is accretive.
\end{corollary}
\begin{proof}
	By condition \ref{m-accretivity2} in Definition \ref{def:m-accretivity}, an operator $\mathcal{L}$ is accretive if $\langle \mathcal{L} u -\mathcal{L} v, u - v \rangle_+ \geq 0$ for every $u,v \in \dom\left(\mathcal{L}\right)$. From \eqref{accretivity_for_lp} in Remark \ref{rem:accretivity_l^2}, in the case of the $\ell^1$-norm we have
	\begin{align*}
		\langle z, k \rangle_+  &=
		\|k\|_1\left( \sum\limits_{\substack{x\in X\colon\\ k(x)= 0}}|z(x)|\mu(x) + \sum\limits_{\substack{x\in X\colon\\ k(x)\neq 0}} z(x)\operatorname{sgn}(k(x))\mu(x) \right)\\
		&\geq \|k\|_1\sum\limits_{\substack{x\in X\colon\\ k(x)\neq 0}} z(x)\operatorname{sgn}(k(x))\mu(x) \quad \forall\, z,k \in \ell^1(X,\mu).
	\end{align*}
	Therefore, to prove that $\mathcal{L}$ is accretive on $\ell^1(X,\mu)$, it is sufficient to prove that
	\begin{equation}\label{eq:m_accretivity}
		\sum_{\substack{x\in X\colon\\u(x)\neq v(x)}} \left(\mathcal{L} u(x)- \mathcal{L} v(x)\right)\operatorname{sgn}(u(x)-v(x))\mu(x) \geq 0 \quad \forall\; u, v \in \dom(\mathcal{L}),
	\end{equation}
	that is,
	\begin{equation*}\label{eq:appB2}
		\sum_{\substack{x\in X\colon\\u(x)\neq v(x)}} \left(\Delta\Phi u(x)-\Delta \Phi v(x)\right)\operatorname{sgn}(u(x)-v(x))\mu(x) \geq 0 \quad \forall\; u, v \in C(X).
	\end{equation*}
	Since $G$ is finite, we conclude the proof by Proposition \ref{prop:nonnegativity}.
\end{proof}

We next establish that the operators $\Lmin$ are accretive.
\begin{corollary}\label{cor:accretivity_lmin}
Let $G=(X,w,\kappa,\mu)$ be a graph.	Then, $\Lmin$ is accretive for every $n$.
\end{corollary}
\begin{proof}
By \eqref{eq:m_accretivity} in Corollary \ref{cor:accretivityL-finite}, it suffices to show that  
	\begin{equation*}
		\sum_{\substack{x\in X\colon\\u(x)\neq v(x)}} \left(\Lmin u(x)- \Lmin v(x)\right)\operatorname{sgn}(u(x)-v(x))\mu(x) \geq 0 \quad \forall\; u, v \in \dom(\Lmin),
	\end{equation*}
	that is,
	\begin{equation*}
		\sum_{\substack{x\in X\colon\\u(x)\neq v(x)}} \left(\boldsymbol{\mathfrak{i}}_{n,\infty}\Deltadirn\Phi\boldsymbol{\pi}_{n} u(x)-\boldsymbol{\mathfrak{i}}_{n,\infty}\Deltadirn \Phi\boldsymbol{\pi}_{n} v(x)\right)\operatorname{sgn}(u(x)-v(x))\mu(x) \geq 0 \quad \forall\; u, v \in C_c(X).
	\end{equation*}
	Let us observe that the left-hand side of the above 
	is equal to
	\begin{align*}
		\sum_{\substack{x\in X_n\colon\\\boldsymbol{\pi}_{n}u(x)\neq \boldsymbol{\pi}_{n}v(x)}} \left(\Deltadirn\Phi\boldsymbol{\pi}_{n} u(x)-\Deltadirn \Phi\boldsymbol{\pi}_{n} v(x)\right)\operatorname{sgn}(\boldsymbol{\pi}_{n}u(x)-\boldsymbol{\pi}_{n}v(x))\mu(x).
	\end{align*}
	Since $\boldsymbol{\pi}_{n} u, \boldsymbol{\pi}_{n} v \in C(X_n)$ and $\Deltadirn$ is the graph Laplacian associated to the finite graph $G_{\textnormal{dir},n}$ with node set $X_n$, by Proposition \ref{prop:nonnegativity} we conclude that $\Lmin$ is accretive.
\end{proof}

The sequence of operators $\Lmin$ defines a subset $\Omega$ of $\dom(\mathcal{L})$. As we will see below, $\Omega$ is dense in $\dom(\mathcal{L})$ and $\mathcal{L}$ restricted to $\Omega$ is accretive. Let us introduce the following notation for the support of a function: Given $u \in C(X)$ we let
\begin{equation*}
\operatorname{supp}u \coloneqq \{x\in X \mid u(x)\neq 0\}.
\end{equation*}

We start by defining the subset of the domain of interest. 
\begin{definition}[The set $\Omega$]\label{def:Omega}
Let $G=(X,w,\kappa,\mu)$ be a graph. We define $\Omega \subseteq \dom(\mathcal{L})$ by letting
\begin{equation*}\label{eq:omega1}
	\Omega \coloneqq \dom(\mathcal{L}) = C(X)
\end{equation*}
if $G$ is finite and
	\begin{equation*}\label{eq:omega2}
		\Omega\coloneqq \{u \in \dom(\mathcal{L}) \mid \exists\, \{u_n\}_n \mbox{ s.t. } \operatorname{supp}u_n\subseteq X_n,\; \lim_{n\to \infty}\|u_n -u \|=0,\; \lim_{n\to \infty}\|\Lmin u_n -\mathcal{L}u \|=0 \}
	\end{equation*}
if $G$ is infinite. 
\end{definition}

While the definition of $\Omega$ depends on the choice of the exhaustion,
this set always contains all finitely supported functions as will be shown in Lemma \ref{lem:FC}. In order to establish this, we first prove that the finitely supported functions are contained in the domain of $\mathcal{L}$.

\begin{lemma}\label{lem:FC0}
Let $G=(X,w,\kappa,\mu)$ be a graph. Then, $C_c(X) \subseteq \dom(\mathcal{L})$.
\end{lemma}
\begin{proof}
Let $u \in C_c(X)$. Then, $u(x) = \sum_{j=1}^n \alpha_j \delta_{x_j}(x)$ where $\alpha_j \in \R$ and 
$$
\delta_{x_j}(x)=\begin{cases}
	1 & \mbox{if } x=x_j,\\
	0 &\mbox{otherwise}.
\end{cases}
$$
Therefore, by linearity, $\Delta(C_c(X))\subseteq \ell^1(X,\mu)$ if and only if $\Delta\delta_{z} \in  \ell^1(X,\mu)$ for every $z\in X$. Fix $z \in X$ and observe that
\begin{align*}
	\sum_{x\in X}|\Delta \delta_{z}(x)|\mu(x) &\leq \sum_{x\in X}\Deg(x)\delta_{z}(x)\mu(x) + \sum_{x\in X}\sum_{y\in X}w(x,y)|\delta_{z}(y)|\\
	&= \Deg(z)\mu(z) + \sum_{x\in X}w(x,z) < \infty
\end{align*}
so that $\Delta(C_c(X)) \subseteq \ell^1(X,\mu)$.
Since $\Phi u \in C_c(X)$ for every $u \in C_c(X)$, it follows that $\Phi u\in \dom(\Delta)$ and $\Delta\Phi u \in \ell^1(X,\mu)$, that is, $u \in \dom(\mathcal{L}$).
\end{proof}

We now show that the set $\Omega$ contains the finitely supported functions.

\begin{lemma}\label{lem:FC}
Let $G=(X,w,\kappa,\mu)$ be a graph. Then, $C_c(X) \subseteq \Omega$. In particular, $$\overline{\Omega}= \overline{\dom(\mathcal{L})}=\ell^1(X,\mu).$$	 
\end{lemma}
\begin{proof}
Let us fix $u \in C_c(X)$. From Lemma \ref{lem:FC0}, we know that $u\in \dom(\mathcal{L})$. Define 
	$$
	u_n(x)\coloneqq \boldsymbol{\mathfrak{i}}_{n,\infty}\boldsymbol{\pi}_n u(x)= \begin{cases}
		u(x) & \mbox{if } x \in X_n,\\
		0 &\mbox{otherwise}.
	\end{cases}
	$$
	Clearly, $\operatorname{supp}u_n \subseteq X_n$ and $\lim_{n\to \infty}\|u_n -u \|=0$. Since, $u\in C_c(X)$, there exists an $N>0$ such that $u(x)=0$ for every $x\in X\setminus X_N$. In particular, $\Phi u_n(x)=0$ for every $x\in X\setminus X_n$, $n\geq N$ and $\Phi u_n=\Phi u$ for every $n\geq N$. 
	
	By Lemma \ref{lem:A1}, we have $\Phi u_n \in \dom(\Delta_{|X_n})$ for every $n\geq N$ and then 
	\begin{equation*}
		\Deltadirn \boldsymbol{\pi}_n \Phi u_n (x) =  \Delta\Phi u_n (x) =  \Delta\Phi u (x) \quad \forall\; x\in X_n,\; n\geq N,
	\end{equation*}
	that is,
	\begin{equation*}
		\Deltadirn \boldsymbol{\pi}_n \Phi u_n = \boldsymbol{\pi}_n\Delta \Phi u \quad \forall\; n\geq N.
	\end{equation*}
	Therefore, using the trivial fact that $\Phi\boldsymbol{\pi}_n=\boldsymbol{\pi}_n\Phi$,
	\begin{equation*}
		\Lmin u_n = \boldsymbol{\mathfrak{i}}_{n,\infty}\Deltadirn \Phi \boldsymbol{\pi}_n u_n=\boldsymbol{\mathfrak{i}}_{n,\infty}\Deltadirn \boldsymbol{\pi}_n \Phi  u_n= \boldsymbol{\mathfrak{i}}_{n,\infty}\boldsymbol{\pi}_n\Delta \Phi u = \boldsymbol{\mathfrak{i}}_{n,\infty}\boldsymbol{\pi}_n\mathcal{L} u, \quad \forall\; n\geq N
	\end{equation*}
	and, since $\mathcal{L}u \in \ell^1(X,\mu)$, it follows that $\lim_{n\to \infty}\| \Lmin u_n - \mathcal{L} u\|=0$
	by dominated convergence.
\end{proof}


To conclude this appendix, we prove that  $\mathcal{L}_{|\Omega}$ is accretive.

\begin{lemma}\label{lem:L_Omega_accretive}
Let $G=(X,w,\kappa,\mu)$ be a graph. Then, $\mathcal{L}_{|\Omega}$ is accretive.
\end{lemma}
\begin{proof}
If $G$ is finite, then $\Omega = \dom(\mathcal{L})$ and $\mathcal{L}$ is accretive by Corollary \ref{cor:accretivityL-finite}. If $G$ is infinite, let $u,v \in \Omega$. Then, by the definition of $\Omega$, there exists $\{u_n\}_n, \{v_n\}_n$ such that 
$$
\lim_{n\to \infty}\|u_n-u\|=\lim_{n\to \infty}\|v_n-v\|=0, \qquad \lim_{n\to \infty}\|\Lmin u_n-\mathcal{L}u\|=\lim_{n\to \infty}\|\Lmin v_n-\mathcal{L}v\|=0.
$$
By the accretivity of $\Lmin$ established in Corollary \ref{cor:accretivity_lmin} above, it follows readily that
\begin{align*}
	\left\|(u - v) + \lambda \left(\mathcal{L}u - \mathcal{L}v\right) \right\|= \lim_{n\to \infty} \left\|(u_n - v_n) + \lambda \left(\Lmin u_n - \Lmin v_n\right) \right\|\geq \lim_{n\to \infty}\|u_n - v_n\|= \|u - v\|.
\end{align*}
This completes the proof.
\end{proof}

\begin{remark}\label{rem:Omega_Lmin}
	One might be tempted to identify $\Omega$ with 
	$$
	\Omega'=\{u\in \dom(\mathcal{L})\mid \exists\, \{u_n\}_n \mbox{ s.t. } u_n\in C_c(X),\; \lim_{n\to \infty}\|u_n -u \|=0,\; \lim_{n\to \infty}\|\mathcal{L}_{\operatorname{min}} u_n -\mathcal{L}u \|=0 \}
	$$
	where $\mathcal{L}_{\operatorname{min}}$ is the minimal operator, that is, $\mathcal{L}_{\operatorname{min}}\coloneqq\mathcal{L}_{|C_c(X)}$. It is possible to show that $\mathcal{L}$ is accretive on $\Omega'$ but unfortunately $\Omega$ is not equal to $\Omega'$. In particular, the solutions that are constructed in the proof of Theorem \ref{thm:main1} may not belong to $\Omega'$. The main problem is that 
	$$
	| (\operatorname{id} + \lambda \mathcal{L}_{\operatorname{min}})u_n(x) - g(x)|=\begin{cases}
		0 &\mbox{if } x\in X_n,\\
		\sum_{{y\in X_n}}w(x,y)u_n(y) + g(x) &\mbox{if } x \in X\setminus X_n
	\end{cases}
	$$
	and then $\|(\operatorname{id} + \lambda \mathcal{L}_{\operatorname{min}})u_n - g\|$ does not necessarily tend to $0$. As a consequence, we cannot infer that $\lim_{n\to \infty}\|\mathcal{L}_{\operatorname{min}} u_n -\mathcal{L}u \|=0$.
\end{remark}

	\section*{Acknowledgments}
	Part of this work was carried out by the first and third authors while they were at the Department of Science and High Technology of the University of Insubria in Como, Italy.
The third author is financially supported by PSC-CUNY Awards, jointly funded by the Professional Staff Congress and the City University of New York, and the Collaboration Grant for Mathematicians, funded by the Simons Foundation. The authors would like to thank Delio Mugnolo for helpful comments and for pointing out references.

\printbibliography
\end{document}